\documentclass[a4paper,english,11pt]{smfart}
\usepackage[dvips]{graphicx}
\usepackage{amssymb}
\usepackage{amsmath}
\usepackage[utf8]{inputenc}
\usepackage[T1]{fontenc}
\usepackage[francais]{babel}
\usepackage{graphicx}
\usepackage{amsmath}
\usepackage{amsthm}
\usepackage{amsfonts}
\usepackage{amssymb}
\usepackage{enumerate}
\usepackage{vaucanson-g}
\usepackage{graphicx}
\usepackage{setspace}
\usepackage{enumitem}

\newtheorem{thm}{Theorem}[section]
\newtheorem{prop}[thm]{Proposition}
\newtheorem{lem}[thm]{Lemma}
\newtheorem{coro}[thm]{Corollary}

\theoremstyle{definition}
\newtheorem{defi}[thm]{Definition}

\theoremstyle{remark} 
\newtheorem{rem}[thm]{Remark}
\theoremstyle{definition}

\newcounter{constant}
\newcommand{\newconstant}[1]{\refstepcounter{constant}\label{#1}}
\newcommand{\useconstant}[1]{c_{\ref{#1}}}
\newcommand{\defconstant}[1]{ \newconstant{c_{#1}}\expandafter\newcommand\csname c#1\endcsname{\useconstant{c_{#1}}} }  
\newcounter{Constant}
\newcommand{\newConstant}[1]{\refstepcounter{Constant}\label{#1}}
\newcommand{\useConstant}[1]{C_{\ref{#1}}}
\newcommand{\defConstant}[1]{ \newConstant{C_{#1}}\expandafter\newcommand\csname C#1\endcsname{\useConstant{C_{#1}}} }

\newcommand{\mL}{{\mathcal L}}

\newcommand{\Q}{\overline{\mathbb Q}}

\newcommand{\y}{{\boldsymbol{y}}}
\newcommand{\C}{\mathbb{C}}
\newcommand{\Z}{\mathbb{Z}}

\newcommand{\N}{\mathbb{N}}
\newcommand{\R}{\mathcal R}
\newcommand{\M}{\mathcal M}

\newcommand{\U}{\overline U}
\newcommand{\K}{\mathcal K}
\newcommand{\A}{\mathcal A}
\renewcommand{\L}{\widetilde{\mathbb L}}

\newcommand{\J}{\mathcal J}

\newcommand{\E}{\mathcal E}

\newcommand{\be}{{\boldsymbol{e}}}
\newcommand{\x}{{\boldsymbol{x}}}

\newcommand{\bK}{{\rm \bf{K}}}

\newcommand{\z}{{\boldsymbol{z}}}
\newcommand{\lambd}{{\boldsymbol{\lambda}}}
\newcommand{\f}{{\boldsymbol f}}
\renewcommand{\O}{\mathcal{O}}
\renewcommand{\S}{\mathcal{S}}
\renewcommand{\k}{{\boldsymbol{k}}}
\renewcommand{\j}{{\boldsymbol{j}}}

\newcommand{\balpha}{{\boldsymbol{\alpha}}}
\newcommand{\X}{\boldsymbol{X}}
\newcommand{\bt}{\boldsymbol{t}}
\newcommand{\Y}{\boldsymbol{Y}}

\newcommand{\omeg}{{\boldsymbol{\omega}}}
\newcommand{\bet}{{\boldsymbol{\eta}}}

\newcommand{\bmu}{{\boldsymbol{\mu}}}
\newcommand{\bnu}{{\boldsymbol{\nu}}}

\newcommand{\w}{{\boldsymbol{w}}}

\newcommand{\bpsi}{{\boldsymbol{\psi}}}
\newcommand{\bphi}{{\boldsymbol{\phi}}}
\newcommand{\bchi}{{\boldsymbol{\chi}}}
\newcommand{\gamm}{{\boldsymbol{\gamma}}}

\newcommand{\kapp}{{\boldsymbol{\kappa}}}
\newcommand{\btau}{{\boldsymbol{\tau}}}
\newcommand{\bomega}{{\boldsymbol{\omega}}}
\newcommand{\bbeta}{{\boldsymbol{\beta}}}

\newcommand{\p}{{\mathfrak{p}}}

\newcommand{\g}{{\boldsymbol g}}

\newcommand{\vphi}{\varphi}
\newcommand{\bbe}{{\overline{\be}}}
\newcommand{\kk}{{\overline{\k}}}
\numberwithin{equation}{section} 

\title[Mahler's method in several variables I]{Mahler's method in several variables I: The theory of  regular singular systems}

\author{Boris Adamczewski}
\address{
Univ Lyon, Universit\'e Claude Bernard Lyon 1\\
 CNRS UMR 5208, Institut Camille Jordan \\
 F-69622 Villeurbanne Cedex, France}
\email{Boris.Adamczewski@math.cnrs.fr}

\author{Colin Faverjon}
\email{colin.faverjon@riseup.net}
\date{}

\thanks{This project has received funding from the European Research Council (ERC) under the European Union's Horizon 2020 research and innovation programme 
under the Grant Agreement No 648132. }
\begin{document}

\begin{abstract}  
This is the first part of a work devoted to the study of linear Mahler systems in several variables from the 
perspective of transcendence and algebraic independence.  
We prove two main results concerning systems that are regular singular at the origin.   

Given some vector of analytic solutions of 
such a Mahler system, say $(f_1(\z),\ldots,f_m(\z))\in \Q\{\z\}^m$, and some suitable algebraic point $\balpha$, we first show that 
any homogeneous algebraic relation over $\Q$ between the complex numbers $f_1(\balpha),\ldots,f_m(\balpha)$  
can be lifted to a similar algebraic relation over $\Q(\z)$ between the functions $f_1(\z),\ldots,f_m(\z)$.   
This phenomenon was first brought to light in the framework of linear differential equations by Nesterenko and Shidlovskii, Beukers, 
and Andr\'e. More recently, Philippon and the authors also established a similar result for linear Mahler systems in one variable. 

Our second main result highlights the fact that the values of Mahler functions associated with sufficiently different matrix 
transformations behave independently.  
More precisely, we show that the ideal formed by the algebraic relations between the values at algebraic points of Mahler functions associated 
with different systems is generated by the pure algebraic relations, that is, the algebraic relations  between the values of the  
functions occurring in each system. Though results in the same spirit were conjectured by van der Poorten in the 1980s, 
only very sporadic examples, due independently to Ku. Nishioka and Masser, have been obtained so far. 

Our interest in Mahler's method comes from the possible applications of these results 
to old problems involving automata theory and which concern the expansion of both natural numbers and real numbers   
in integer bases. In particular, problems which involve finite automata and base change.  
Such applications are studied in the companion paper \cite{AF3}.  
\end{abstract}

\bibliographystyle{abbvr}
\maketitle
\setcounter{tocdepth}{2}
\tableofcontents

\section{Introduction}\label{Introduction}

Any non-trivial algebraic (resp., linear) relation over  
$\overline{\mathbb Q}(z)$ between given analytic functions $f_1(z),\ldots,f_n(z)\in \overline{\mathbb Q}\{z\}$,   
leads by specialization at a given algebraic point $\alpha$ to a non-trivial algebraic (resp., linear) relation over 
$\overline{\mathbb Q}$ between the complex numbers $f_1(\alpha),\ldots,f_n(\alpha)$, assuming that these functions 
are well-defined at $\alpha$. 
As discussed in \cite{Ma1969b}, we cannot expect the converse assertion to be true in general, 
but there are a few known instances where it holds true.     
In each case, some additional structure is required and   
the analytic functions under consideration must satisfy some functional equations, 
such as  a system of linear differential equations or of linear difference equations.

Mahler's equations provide an example of such a framework.  
Let $f_1(\z),\ldots,f_m(\z)\in \overline{\mathbb Q}\{\z\}$ be multivariate analytic functions 
which converge in 
some neighborhood of the origin and which are related by a system of functional equations of the form   
\eqref{eq:mahler}. 
Let $\balpha\in \Q^n$ be such that these functions are well-defined at $\balpha$.   
 Mahler's method aims at transferring results about the absence of algebraic (resp., linear) relations over 
$\overline{\mathbb Q}(\z)$ between the functions $f_1(\z),\ldots,f_m(\z)$  to the absence of algebraic 
(resp., linear) relations over $\overline{\mathbb Q}$ between the complex numbers 
$f_1(\balpha),\ldots,f_m(\balpha)$. This problem goes back to the pioneering work of 
Mahler \cite{Ma29,Ma30a,Ma30b} at the end of the 1920s. In fact,  
a large part of transcendental number theory is concerned with similar questions.   

Throughout this paper, we focus on the so-called \emph{regular singular} linear Mahler systems 
introduced in Definition \ref{def: regular singular}, that is, 
those which  are conjugated, through 
a ramified meromorphic gauge transform, to a system associated with a constant invertible matrix.  
We develop a general theory for these systems 
from the perspective of transcendence and algebraic independence. 
Our interest in Mahler's method comes from the possible applications 
of these results to old problems concerning the expansion of both natural and real numbers in  
integer bases. In particular, old problems which involve finite automata and base change.  
Such applications will be discussed in the companion paper \cite{AF3}.  
Unfortunately, the fact that our results are restricted to regular singular systems affects the generality of 
their application to these problems. In this regard, it would be of great interest to extend the results of this paper   
to the case of general linear Mahler systems.  

\subsection{Mahler's transformations and linear Mahler systems}

Let $n$ be a positive integer and $T=(t_{i,j})_{1\leq i,j\leq n}$ be an $n\times n$ matrix with non-negative 
integer coefficients. We let $T$ act on $\mathbb C^n$ by  
$$
T \balpha = (\alpha_1^{t_{1,1}}\alpha_2^{t_{1,2}}\cdots \alpha_n^{t_{1,n}},\ldots,\alpha_1^{t_{n,1}}\alpha_2^{t_{n,2}}\cdots \alpha_n^{t_{n,n}})\, , 
$$ 
where $\balpha=(\alpha_1,\ldots,\alpha_n)$. In order to avoid confusion, 
we use $T(\balpha)$ to denote the usual matrix product.  
We will also consider $T$ as acting on monomials 
associated with a $n$-tuple of indeterminates $\z=(z_1,\ldots,z_n)$.  
We let  $\Q$ denote the field of algebraic numbers which embeds into the field $\mathbb C$ 
of complex numbers.  We let $\Q^\star$ denote the set $\Q\setminus \{0\}$. 
A \emph{linear $T$-Mahler system}, or simply a \emph{Mahler system}, is a system of functional equations of the form 

\begin{equation}
\label{eq:mahler}
\left(\begin{array}{c} f_1(T\z) \\ \vdots \\ f_m(T\z) \end{array} \right) = 
A(\z)\left(\begin{array}{c} f_1(\z) \\ \vdots \\ f_m(\z) 
\end{array} \right)\, ,
\end{equation}
where $A(\z)\in{\rm GL}_m(\Q(\z))$. 

Let $\mathbb K$ be a subfield of the complex numbers. 
We let $\mathbb K\{\z\}$ denote the set of multivariate power series in $\z$, with coefficients in $\mathbb K$, 
and which are convergent in some neighborhood of $\boldsymbol 0$.   
In the sequel, we refer to the elements of $\mathbb K\{\z\}$ as being \emph{analytic}, and to 
the elements of $\widehat\bK_1$, the field of fractions of $\mathbb \Q\{\z\}$, as being \emph{meromorphic}. 
More generally, given a positive integer $d$, we let $\widehat\bK_d$ denote the field of fractions of $\Q\{\z^{1/d}\}$, where 
$\z^{1/d}=(z_1^{1/d},\ldots,z_n^{1/d})$. 
We also set 
$$
\widehat\bK:= \cup_{d\geq 1} \widehat\bK_d\, .
$$
We define now the regular singular Mahler systems.

\begin{defi}
\label{def: regular singular}
A system of the form \eqref{eq:mahler} is said to be \emph{regular singular} at the origin, 
or for short regular singular, 
if there exists a 
matrix $\Phi(\z) \in {\rm GL}_m(\widehat\bK)$ 
such that  $\Phi(T\z)A(\z)\Phi^{-1}(\z) \in {\rm GL}_m(\Q)$.
\end{defi}

A similar definition in the case of linear Mahler systems in one variable 
can be found in \cite{Ro}. According to Loxton and van der Poorten \cite{LvdP82}, 
if the matrix $A(\z)$ is well-defined and non-singular at $\boldsymbol 0$, then the corresponding 
Mahler system is regular singular. 

\subsection{Previous results} 

A well-known feature of Mahler's method is that, 
independently of the choice of the matrix $A(\z)$ defining the system \eqref{eq:mahler}, 
some quite natural restrictions on the transformation matrix $T$ and 
on the point $\balpha$ are required. Such conditions already appeared in the 
work of Mahler.

\begin{defi}\label{def:admissible}
Let $T$ be an $n\times n$ matrix with non-negative integer coefficients and $\balpha\in(\Q^\star)^n$. 
The pair $(T,\balpha)$ is said to be \emph{admissible} if there exist two real numbers $\rho>1$ and $c>0$ such that 
the following three conditions hold true. 

\begin{enumerate}[label=(\alph*)]
\item \label{condition: A} 
The entries of the matrix $T^k$ are bounded by $c\rho^k$, for every positive integer $k$.  

\item \label{condition: B} Set $T^k{\boldsymbol \alpha}:=(\alpha_1^{(k)},\ldots,\alpha_n^{(k)})$. 
Then 
$$
\log \vert \alpha_i^{(k)}\vert \leq -c\rho^k\, ,
$$
for all positive integers $k$ and all integers $i$, $1\leq i\leq n$.  

\item \label{condition: C} If $f({\z})$ is any non-zero element of $\mathbb C\{{\z}\}$, then there 
are infinitely many integers $k$ such that $f(T^k\boldsymbol \alpha)\not=0$. 
\end{enumerate}
\end{defi}

\begin{rem} When $n=1$, the operator $T$ takes the simple 
form $z\mapsto z^q$, where $q\geq 2$ is an integer. In that case, it is easy to check that Conditions (a)--(c) 
are always satisfied  for every algebraic number $\alpha$ with $0<\vert \alpha\vert <1$.   
In particular, the non-vanishing Condition (c) follows immediately from the identity theorem.
\end{rem}

In addition to the admissibility of the pair $(T, {\boldsymbol \alpha})$, a further restriction has to be 
imposed on the point $\balpha$, 
namely that it be a regular point. 
The latter depends both on the matrices $A(\z)$ and $T$. 

\begin{defi}\label{def:reg}
A $n$-tuple $\balpha:=(\alpha_1,\ldots,\alpha_n) \in (\Q^\star)^n$ is \emph{regular} with respect to the Mahler system 
\eqref{eq:mahler} if the matrix $A(\z)$ is well-defined and invertible at $T^k\balpha$ for all non-negative integers $k$. 
\end{defi}

\subsubsection{The case $n=1$} 
In 1990, Ku. Nishioka \cite{Ni90} established the equality 
\begin{equation}\label{eq: degtr}
{\rm tr.deg}_{\Q}(f_1(\alpha),\ldots,f_m(\alpha))= {\rm tr.deg}_{\mathbb C(z)}(f_1(z),\ldots,f_m(z)) \, 
\end{equation}
for all matrices $A(z)$ and all regular algebraic points $\alpha$ in the open unit disc of $\mathbb C$.    
The great advantage here is that Nishioka's theorem 
applies to all Mahler systems and not only to the regular singular ones.
More recently, Philippon \cite{PPH} refines Nishioka's theorem by proving a result analogous to our 
Theorem \ref{thm: permanence}. Some striking consequences of Philippon's theorem, concerning automatic numbers and 
the transcendence of values of Mahler functions at 
algebraic points, are given by the authors in \cite{AF1,AF2}. For instance, it is proved that there exists an algorithm 
that performs the following task: Given any Mahler function $f(z)\in \Q\{z\}$ and any algebraic number $\alpha$, it   
decides whether $f(\alpha)$ is algebraic or transcendental. 

\subsubsection{The case $n\geq 2$} In contrast, and despite many attempts in this direction, 
no general result has been proved so far when $n\geq 2$.   In 1982, Loxton and van der Poorten \cite{LvdP82} 
published a paper claiming that 
\begin{equation}\label{eq:degretranscendance}
{\rm tr.deg}_{\Q}(f_1(\balpha),\ldots,f_m(\balpha))= {\rm tr.deg}_{\mathbb C(\z)}(f_1(\z),\ldots,f_m(\z)) \,
\end{equation}
when the matrix $A(\boldsymbol 0)$ is well-defined and non-singular, the pair $(T,\balpha)$ is admissible and $\balpha$ is a 
regular algebraic point.  
Unfortunately, some argument in their proof is flawed. This is reported, for instance, by  Ku. Nishioka in \cite{Ni90}. 
To date, Mahler's method in several variables has been applied successfully only for the two following 
much more restricted classes of matrices.   

\begin{itemize}

\item[$\bullet$] First, Kubota \cite{Ku77} proved in 1977 that Equality \eqref{eq:degretranscendance} holds 
true when the matrix $A(\z)$ is 
\emph{almost diagonal}, that is, when 
the functions $f_i(\z)$ satisfy equations of the form 
\begin{equation}\label{eq:degre1}
\left(\begin{array}{c} 1\\ f_1(T\z) \\ \vdots \\ f_m(T\z)  \end{array} \right) = 
\left(\begin{array}{c|ccc} 
1 & 0\cdots &  & 0\\
\hline 
 b_1(\z)& a_1(\z)&&
\\  \vdots && \ddots &
\\ b_m(\z) &&& a_m(\z) 
 \end{array}\right)\left(\begin{array}{c} 1\\ f_1(\z) \\ \vdots \\ f_m(\z) 
\end{array} \right)\, 
\end{equation}
where $a_i(\z), b_i(\z) \in \Q(z)$, with no pole at $\boldsymbol 0$, and $\ a_i(\boldsymbol 0) \neq 0\, $.  

\medskip

\item[$\bullet$] Then, Nishioka \cite{Ni96}  proved in 1996 that Equality \eqref{eq:degretranscendance} holds true when the 
matrix $A(\z)$ is \emph{almost constant}, that is,  
for systems of the form  
\begin{equation}\label{eq:constante}
\left(\begin{array}{c} 1\\ f_1(T\z) \\ \vdots \\ f_m(T\z)  \end{array} \right) = 
\left(\begin{array}{c|ccc} 
1 & 0&\cdots   & 0
\\ 
\hline
b_1(\z) &&& 
\\
\vdots   & &B & 
\\ b_m(\z) &&& 
 \end{array}\right)\left(\begin{array}{c} 1\\ f_1(\z) \\ \vdots \\ f_m(\z) 
\end{array} \right)\, 
\end{equation}
where $B\in {\rm GL}_m(\Q)$, and $b_1(\z),\ldots,b_m(\z)$ are rational functions with no pole at $\boldsymbol 0$.  
\end{itemize}

In these two examples, the matrix defining the system is always assumed to 
be well-defined and non-singular at $\boldsymbol 0$.  
In particular, the corresponding Mahler systems are regular singular. 

\section{Main results}\label{sec: main}

As a first contribution, we provide a complete proof of the main result claimed by Loxton and 
van der Poorten in \cite{LvdP82}.  Furthermore, we refine the conclusion given 
by the quantitative Equality \eqref{eq:degretranscendance}, by proving 
that any algebraic relation over $\Q$ between the values $f_1(\balpha),\ldots,f_m(\balpha)$  can be lifted to 
a similar algebraic relation over $\Q(\z)$ between the functions $f_1(\z),\ldots,f_m(\z)$.  
We stress that such a qualitative refinement is a key for applications. 

\begin{thm}[Lifting]
\label{thm: permanence}
Let $f_1(\z),\ldots,f_m(\z)\in \Q\{\z\}$ be solutions to a regular singular Mahler system of type \eqref{eq:mahler}. 
Let us assume furthermore that $\balpha\in  (\Q^\star)^n$ is a regular point and 
that the pair $(T,\balpha)$ is admissible. Then for any homogeneous polynomial $P \in \Q[X_1,\ldots,X_m]$ such that 
$$
P(f_{1}(\balpha),\ldots,f_{m}(\balpha)) = 0\,,
$$
there exists a polynomial $Q\in\Q[\z,X_1,\ldots,X_m]$, homogeneous in the variables $X_1,\ldots,X_m$, 
such that 
\begin{eqnarray*}
Q(\z,f_{1}(\z),\ldots,f_{m}(\z)) = 0 & \mbox{ and } &
Q(\balpha,X_1,\ldots,X_m)=P(X_1,\ldots,X_m).
\end{eqnarray*}
\end{thm}

Similar results  
have first been obtained in the framework of linear differential equations by Nesterenko and Shidlovskii \cite{NS96}, 
by Beukers \cite{Beu06} using some results of Andr\'e \cite{An1,An2} on the theory of $E$-operators,  
and then by Andr\'e \cite{An3}. In the case $n=1$, the authors already establish Theorem \ref{thm: permanence}  
for general linear Mahler systems in \cite{AF2}, as a consequence of the slightly weaker non-homogeneous version 
due to Philippon \cite{PPH}. 

As a straightforward application, we deduce the following result. 

\begin{coro}
We continue with the assumptions of Theorem \ref{thm: permanence}. If the functions 
$f_1(\z),\ldots,f_m(\z)$ are linearly independent over $\Q(\z)$, then the numbers 
$f_{1}(\balpha),\ldots,f_{m}(\balpha)$ are linearly independent over $\Q$. 
\end{coro}

\begin{rem} 
Theorem \ref{thm: permanence} also applies to non-homogeneous relations, for we can always  
turn an inhomogeneous relation into an homogeneous one by adding the constant function 
$f_{0}\equiv 1$ to the system and replace the matrix $A(\z)$ by 
$$\left(\begin{array}{c|ccc} 
1 &  &0&
\\ \hline 
\\ 
 0 &  &A(\z) &
 \\ &&&
\end{array}\right)\, .$$
\end{rem}

Let us turn to our second main result. In 1987, van der Poorten \cite{vdP87} claimed that  
 Mahler's method could be generalized in such a way that it would become possible to consider simultaneously 
 several linear systems of type \eqref{eq:mahler} associated with sufficiently different transformations $T_i$. 
In order to guarantee some uniform speed of convergence to $\boldsymbol 0$ for the orbits of algebraic points $\balpha_i$ 
under the different transformations $T_i$, van der Poorten suggested  iterating each matrix to different powers. 
We then leave the classical Mahler method, which considers the action from $\N$ on $\C^n$ induced by a single transformation $T$, 
to consider an action of $\N^r$ onto some $\C^{n_1+\cdots +n_r}$ induced by several transformations $T_1,\ldots,T_r$.  
He also pointed out several striking consequences that would follow from such a theory. 
However, only very sporadic examples \cite{Ni94,Mas99} have been obtained so far in this direction.   
In 1994, Nishioka \cite{Ni94} studied this problem for the almost diagonal Mahler systems of the form \eqref{eq:degre1}.   
In particular, she deduced from her main theorem the following result. 
Given a non-zero algebraic number $\alpha$, $\vert \alpha\vert <1$, 
and a set of quadratic irrational numbers $(\omega_i)_{i\in\mathcal I}$ such that $\mathbb Q(\omega_i)\not=\mathbb Q(\omega_j)$ if 
$i\not= j$, the complex numbers that belong to the sets 
$$
\left(\sum_{k=0}^\infty \alpha^{d^k}\,\right)_{k  \geq 2}\, ,\;\; \left(\prod_{k=0}^\infty (1-\alpha^{d^k})\right)_{k  \geq 2}\,  
\,,\;\; \left(\sum_{k=0}^\infty \lfloor k\omega_i\rfloor \alpha^k\right)_{i\in \mathcal I} 
$$
are all algebraically independent over $\Q$. 

We develop here a similar theory that applies to a much larger class of Mahler systems, transformation matrices, 
and algebraic points.  
In this respect,  Theorem \ref{thm: purity} solves van der Poorten's problem in a very satisfactory way for 
all regular singular systems.  Before stating this result, we introduce some notation. 
Given a finite set of complex numbers $\mathcal E:= \{\zeta_1,\ldots,\zeta_m\}$,  we let 
$$
{\rm Alg}_{\Q}(\mathcal E):= \left\{P(X_1,\ldots,X_m)\in\Q[X_1,\ldots,X_m] : P(\zeta_1,\ldots,\zeta_m)=0\right\} \,
$$ 
denote the ideal of algebraic relations over $\Q$ between the elements of $\mathcal E$. 
Now, we consider several sets of complex numbers 
$$\mathcal  E_1=\{\zeta_{1,1},\ldots,\zeta_{1,m_1}\},\ldots, \mathcal  E_r
=\{\zeta_{r,1},\ldots,\zeta_{r,m_r}\} \,,$$
and we set $\mathcal E=\cup \mathcal E_i$. Let  
$P\in {\rm Alg}_{\Q}(\mathcal E)$. If the polynomial $P$ belongs to the extended ideal 
${\rm Alg}_{\Q}(\mathcal E_i\mid \mathcal E)$ generated by 
${\rm Alg}_{\Q}(\mathcal E_i)$ in $\Q[X_1,\ldots,X_M]$, where $M=m_1+\cdots+m_r$, 
then we say that $P$ is a \emph{pure algebraic relation} with respect to $\mathcal E_i$. 
Our second main result then reads as follows. 

\begin{thm}[Purity--Independent transformations]
\label{thm: purity}
Let $r\geq 2$ be an integer. For every integer $i$, $1\leq i \leq r$, we consider a regular singular Mahler system   
\begin{equation}\stepcounter{equation}
\label{eq:mahleri}\tag{\theequation .i}
\left(\begin{array}{c} f_{i,1}(T_i\z_i) \\ \vdots \\ f_{i,m_i}(T_i\z_i) \end{array} \right) = 
A_i(\z_i)\left(\begin{array}{c} f_{i,1}(\z_i) \\ \vdots \\ f_{i,m_i}(\z_i) 
\end{array} \right) 
\end{equation}
where $A_i(\z_i)$ belongs to ${\rm GL}_{m_i}(\Q(\z_i))$, $\z_i:=(z_{i,1},\ldots,z_{i,n_i})$ is a family of indeterminates, 
$T_i$ is an $n_i\times n_i$ matrix,  
with non-negative  integer coefficients and with spectral radius 
$\rho(T_i)$. For every $i$, $1\leq i\leq r$,  let us consider 
$$\mathcal E_i\subseteq \{f_{i,1}(\balpha_i),\ldots,f_{i,m_i}(\balpha_i)\}$$  
and set $\mathcal E:=\cup_{i=1}^r\mathcal E_i$. 
Suppose that 

\begin{enumerate}

\item[\rm{(i)}]  for every $i$, $\balpha_i\in(\Q^\star)^{n_i}$ is a regular point 
with respect to the system \eqref{eq:mahleri} and the pair $(T_i,\balpha_i)$ is admissible, and

\item[\rm{(ii)}]   for every pair $(i,j)$, $i\not=j$, $\log \rho(T_i)/\log \rho(T_j)\not\in \mathbb Q$. 

\end{enumerate}
Then
$$
{\rm Alg}_{\Q}(\mathcal E) = \sum_{i=1}^r {\rm Alg}_{\Q}(\mathcal E_i\mid \mathcal E)\,.
$$
\end{thm}

In other words, the algebraic relations between all elements of $\mathcal E$ are generated by the 
pure algebraic relations with respect to each subset $\mathcal E_i$. We stress that Theorem \ref{thm: purity} is 
a strong statement about algebraic independence.  

\begin{coro}\label{cor:2}
We continue with the assumptions of Theorem \ref{thm: purity}. Furthermore, we assume that  
for every $i$, $1\leq i\leq r$, all complex numbers that belong to the set $\mathcal E_i$ are algebraically independent  over $\Q$. 
Then all complex numbers that belong to the set $\mathcal E$ are algebraically independent over $\Q$. 
\end{coro}

\medskip

This paper is organized as follows. Clearly, the strength of Theorems \ref{thm: permanence} and \ref{thm: purity} 
strongly depends on our ability to provide simple and natural conditions that ensure the admissibility of 
pairs $(T, {\boldsymbol \alpha})$. 
This problem is addressed in Section \ref{sec: admissibility1}, where concrete and optimal conditions are 
given. A well-known feature of Mahler's method (and also of transcendence theory in general) 
is the great importance of proving \emph{vanishing theorems}, that is, of finding general conditions 
that allow to guarantee non-vanishing conditions of type (c). 
In the case of a single transformation, Masser \cite{Mas82} solved this problem in a rather definitive way. 
More recently, Corvaja and Zannier \cite{CZ05} used the Subspace Theorem to prove a general result 
concerning the vanishing at $S$-units of analytic multivariate power series with algebraic coefficients. 
Based on this result, we prove in Section \ref{sec: vanishing} our own vanishing theorem, that applies to the study 
of several Mahler systems associated with different transformations. This is a key ingredient for proving Theorem \ref{thm: purity}.  
In Section \ref{sec: families}, we state Theorem \ref{thm: families}, which is an axiomatic result concerning 
the algebraic relations between values of functions belonging to several Mahler systems. It is restricted to algebraic points satisfying 
some ad hoc admissibility conditions called (A), (B), and (C), which replace Conditions (a), (b), and (c) in this more general framework. 
Theorem \ref{thm: families} is proved in Section \ref{sec: mainproof}, 
while concrete and optimal conditions for admissibility in Theorem \ref{thm: families} are obtained in Section \ref{sec: admgen}.   
Theorems \ref{thm: permanence} and \ref{thm: purity} are then derived from these results  
in Section \ref{sec: final}. 

\subsubsection*{Notation} We fix here some notation that we use in this paper. 
Given a field $\mathbb K$, we let denote by $\mathbb K^\star$ the set $\mathbb K\setminus\{0\}$. 
Let $d$ be a positive integer. If $\balpha:=(\alpha_1,\ldots,\alpha_d)\in (\mathbb K^\star)^d$ and 
$\k:=(k_1,\ldots,k_d)\in\mathbb Z^d$, then $\balpha^{\k}$ stands for $\alpha_1^{k_1}\cdots \alpha_d^{k_d}$. 
Given a $d$-tuple of natural numbers 
$\k:=(k_1,\ldots,k_d)$, we set $\vert \k\vert =k_1+\cdots +k_d$. 
The maximum norm of $\C^d$ and the maximum norm of $\M_d(\C)$ are both denoted by the same symbol $\Vert\cdot\Vert$. 
We let $H(\cdot)$ denote the \emph{absolute Weil height} over the projective space $\mathbb P^d(\Q)$.  
Given $\boldsymbol \beta = (\beta_1,\ldots,\beta_d) \in \Q^d$, we also write $H(\boldsymbol \beta)$ instead of $H(\beta_1:\cdots:\beta_d:1)$. 

\section{Admissibility conditions for Theorems \ref{thm: permanence} and \ref{thm: purity}}\label{sec: admissibility1}

Conditions (a), (b), and (c) in Definition \ref{def:admissible} are necessary in order to apply Mahler's method (cf. \cite{Ma29}). 
Though they appear quite naturally in transcendance proofs, it is not easy, at first glance, to see how to check them. 
We provide here a simple characterization of matrices and algebraic points satisfying these conditions, gathering and slightly completing 
results of Masser, Kubota, Loxton and van der Poorten. 


\begin{defi} 
Let $T$ be an $n\times n$ matrix with non-negative integer coefficients and with spectral radius $\rho$. 
We say that $T$ belongs to the class $\M$ if it satisfies the following three conditions.  

\smallskip

\begin{itemize}
\item[(i)] It is non-singular. 

\smallskip

\item[(ii)] None of its eigenvalues  are a root of unity.

\smallskip

\item[(iii)] There exists an eigenvector with positive coordinates associated with the eigenvalue $\rho$. 
\end{itemize}
\end{defi}

In particular, a matrix in the class $\M$ has a spectral radius $\rho > 1$. 

\begin{rem}
Let us consider $r$ Mahler systems associated with transformations $T_1,\ldots,T_r$ in $\M$, all having the 
same spectral radius $\rho$. Then the diagonal matrix $T:={\rm diag}(T_1,\ldots,T_r)$ also belongs to the class $\M$. 
More generally, if  $T_1$ and $T_2$ are two matrices in $\M$ whose spectral radii are multiplicatively dependent, say 
$\rho(T_1)^p=\rho(T_2)^q$, then the matrix 
$$
T := \left(\begin{array}{cc} T_1^p & 0 \\ 0 & T_2^q \end{array} \right) \, 
$$
also belongs to the class $\M$. Thus, if one considers $r$ Mahler systems, associated with transformation 
matrices $T_1\ldots,T_r\in \M$, and with pairwise
multiplicatively dependent spectral radius, it is possible to 
gather them into a bigger Mahler system whose transformation matrix also belongs to the class $\M$, 
and then to apply  Theorem \ref{thm: permanence}. 
\end{rem}

Given a one-variable Mahler system associated with a matrix $A(z)$, we could always consider the same twice 
system but with different variables. 
That is, the system associated with the matrix
$$
\left(\begin{array}{cc} A(z_1) & 0 \\ 0 & A(z_2) \end{array} \right) \,.
$$
This shows that some kind of minimal independence between the coordinates of the point $\balpha=(\alpha_1,\alpha_2)$ 
is required in order to apply Mahler's method. This leads to the following natural definition. 

\begin{defi}
An algebraic point $\balpha\in (\Q^\star)^n$ is  said to be {\it $T$-independent} 
if there is no non-zero $n$-tuple of integers $\bmu$ for which $(T^k\balpha)^\bmu=1$ for all $k$ 
in an arithmetic progression.  
\end{defi}

With these definitions, we can gather results of Kubota \cite{Ku77}, Loxton and van der Poorten 
\cite{LvdP77,LvdP77II}, and mainly Masser \cite{Mas82}, to give the following useful characterization 
of the notion of admissibility.  

\begin{thm}
\label{thm:masser}
Let $T$ be an $n\times n$ matrix with non-negative integer coefficients and $\balpha\in (\Q^\star)^n$.  
Then the pair $(T,\balpha)$ is admissible if and only if  
 $T$ belongs to the class $\M$,  $\lim_{k\to \infty}T^k\balpha=\boldsymbol 0$ and $\balpha$ is $T$-independent. 
 \end{thm}

\begin{rem}
Note that it is easy to check whether or not a matrix belongs to the class $\M$.  
Furthermore, if $\alpha_1,\ldots,\alpha_n\in \mathbb C^\star$ are multiplicatively independent complex numbers,  
then $\balpha=(\alpha_1,\ldots,\alpha_n)$ is \emph{a fortiori} $T$-independent.  So, Theorem \ref{thm:masser} makes 
Theorems \ref{thm: permanence} and \ref{thm: purity} very easy to apply concretely.  
\end{rem}

Let us denote by $\mathcal U(T)$ the set of points $\balpha$ of $(\mathbb C^\star)^n$ such that Condition (b) holds. 
Loxton and van der Poorten \cite{LvdP77,LvdP77II} stated that when $T$ belongs to $\mathcal M$, 
the set $\mathcal U(T)$ is a punctured neighborhood of the origin. 
We provide here a proof of the following refinement.   

\begin{lem}\label{lem:U(T)}
Let $T\in\M$, then 
$$
\mathcal U(T)= \left\{\balpha \in  (\C^\star)^n : \lim_{k\to\infty}T^k\balpha = \boldsymbol 0\right\}\, .
$$
In particular, $\mathcal U(T)$  is open in $(\C^\star)^n$, and contains the punctured open unit disk 
of $\mathbb C^n$ with the maximum norm thereon.
\end{lem}

\begin{proof}
Let us first show that $\mathcal U(T)$ contains the punctured open unit disk of $(\C^\star)^n$. 
Let $\balpha:=(\alpha_1,\ldots,\alpha_n) \in (\C^\star)^n$ such that $\Vert\balpha\Vert< 1$. Set 
$$
  L(\balpha) := \left( - \log |\alpha_{1}|,\ldots,-\log |\alpha_{n}| \right) > 0 \, .
$$
By assumption, $T$ has a positive eigenvector $\bmu$ associated with $\rho(T)$. 
We can assume that all coordinates of $\bmu$ are smaller than those of $L(\balpha)$. 
We set $\bnu = L(\balpha) - \bmu > 0$. Following \cite[Lemma 3]{LvdP77}, we get that  
$$
- \log \Vert T^{k}\balpha\Vert  = \Vert  T^{k}(L(\balpha)) \Vert  = \Vert T^{k}(\bmu) 
+ T^{k}(\bnu)\Vert  \geq \Vert  T^{k}(\bmu)\Vert  = \rho^{k}\Vert \bmu\Vert  \, ,
$$
for all $k\in \mathbb N$ for $T^k\bnu$ has positive coordinates. Condition (b) is thus satisfied inside the 
unit open disk of  $(\C^\star)^n$. 
Now if $\balpha \in (\C^\star)^n$ is such that $\lim_{k\to\infty}T^k\balpha =\boldsymbol 0$, then there exists $k_0$ such 
that $T^{k_0}\balpha \in \mathcal U(T)$. 
It follows that $\balpha$ also belongs to $\mathcal U(T)$.
\end{proof}

Now we show that only matrices in the class  $\M$ can be admissible. 

\begin{lem}
\label{lem:ClasseM}
Let us assume that there exists an algebraic point $\balpha$ such that the pair $(T,\balpha)$ is admissible.  
Then $T$ belongs to $\M$.
\end{lem}

\begin{proof}
Kubota \cite{Ku77} already noticed that if $T$ has zero or a root of unity as an eigenvalue, then there do not 
exist any point $\balpha$ satisfying Condition (c). We show now that $\rho = \rho(T)$ and, 
using Conditions (a) and (b), that $T$ has eigenvector with positive coordinates associated with the eigenvalue $\rho$. 
Let us recall some classical results about matrices with non-negative integer coefficients (see for instance \cite{Grantmacher}).  
A matrix  $T$ with non-negative coefficients is said to be \emph{irreducible} if there is no permutation such that 
$T$ takes the form 
$$
\left(\begin{array}{cc} A & 0 \\ B & C \end{array} \right)\, ,
$$
where $A$ and $C$ are square matrices. By Frobenius' theorem \cite[Chapter III, Theorem 2]{Grantmacher}, if 
$T$ is irreducible, then (iii) holds. Furthermore, if $T$ has exactly $h$ eigenvalues $\lambda_1,\ldots,\lambda_h$ of modulus 
$\rho(T)$, then $\lambda_i^h=\rho(T)$ for every $i$, $1 \leq i \leq h$. When $h=1$, $T$ is said to be {\it primitive}. 
Every matrix  $T$ with non-negative integer coefficients can be written, up to permutation, in the form 
\begin{equation}
\label{normalform}
T=\left(\begin{array}{cccccc} 
T_1 & & & & & \\
 & \ddots &  & &0 &\\
  0  &  & T_\kappa & & &  \\ 
 S_{1,1} & \cdots & S_{1,\kappa} & T_{\kappa+1} &  &\\
 \vdots & & & \ddots& \ddots &\\
 S_{\nu,1} & & \cdots&  &  & T_{\kappa+\nu}\end{array} \right) \,,
\end{equation}
where $T_1,\ldots,T_{\kappa+\nu}$ are irreducible square matrices, and such that  for each $i$, $1\leq i \leq \nu$, at least one 
of the matrices $S_{i,j}$, $1 \leq j < i$ is non-zero. This expansion is called the \emph{normal form} of $T$ and is unique, 
up to permutations of the blocks $T_1,\ldots,T_\kappa$, the blocks $T_{\kappa+1},\ldots,T_{\kappa+\nu}$, and also of 
the indices inside each block \cite[Chapter 4]{Grantmacher}. Following \cite[Chapter 3, Theorem 6]{Grantmacher}, 
$T$ satisfies Condition (iii) if and only if its normal form satisfies the following two conditions. 
\begin{enumerate}[label=\roman*)]
\item[(1)] $\rho(T_1)=\cdots=\rho(T_\kappa)=\rho(T)$,

\smallskip

\item[(2)] $\rho(T_{\kappa+i})<\rho(T)$ for $1 \leq i \leq \nu$.
\end{enumerate}
Let us first show that $\rho=\rho(T)$. Condition (a) ensures that the coefficients of $T^k$ are in $\mathcal O(\rho^k)$. 
This implies that $\rho(T) \leq \rho$. Let us denote by $L$ the map 
$$
L : \z \mapsto (- \log |z_1|,\ldots,-\log |z_n|)\, 
$$
and $\x = L(\balpha)$. By construction, one has 
$$
L(T^k\balpha) =  T^k(\x) \, .
$$
Condition (b) ensures that $\Vert T^k(\x)\Vert  \geq \gamma \rho^k$ for some positive number $\gamma$. 
Thus $\rho \leq \rho(T)$ which gives that $\rho = \rho(T)$. 
We now show that the normal form of $T$ satisfies Conditions (1) 
and (2). We argue by contradiction. We assume that (1) is not satisfied. 
Then there is a matrix $T_i$, $1\leq i \leq \kappa$, such that $\rho(T_i) < \rho(T)$. 
Without loss of generality, we can assume that $i=1$. 
Let us denote by $\x_1$ the restriction of $\x$ to the coordinates of the block $T_1$, and by $\y$ the projection of $\x_1$ 
on the eigenspace associated to the eigenvalue  $\rho(T_1)$ of $T_1$. The matrix $T_1$ being irreducible, one has  
$$
T_1^{hk}(\x_1) = \rho(T_1)^{hk}\y + \mathcal O(\rho'^{hk})\, 
$$
for some $\rho' < \rho(T_1)$. This contradicts Condition (c). We thus have 
$\rho(T_1)=\cdots=\rho(T_\kappa)=\rho$. 
Let us now assume that  $\rho(T_{\kappa+j})=\rho(T)$ for some $j$, $1 \leq j\leq \nu$. 
For the sake of simplicity, 
we assume that the normal form of $T$ is 
$$
\left(\begin{array}{cc}
T_1 & 0
\\ S & T_2
\end{array}
\right)\, ,
$$
where  $S \neq 0$, $\rho(T_1)=\rho(T_2)$, and $T_1, T_2$ are irreducible. 
The proof is similar to this case when there are more blocks. 
Raising $T$ to the power of $h$ if necessary, we can assume that $h=1$.  The matrices $T_1$ et $T_2$ are thus 
primitive. For every positive integer $k$, one has 
$$
T^k=\left(\begin{array}{cc}
T_1^k & 0
\\ \sum_{j=0}^{k-1} T_2^jST_1^{k-j-1} & T_2^k
\end{array}
\right)\, .
$$
As $T_1$ and $T_2$ are primitive and $S$ is non-zero, the matrix $\sum_{j=0}^{k-1} T_2^jST_1^{k-j-1}$ 
has a coefficient, that is asymptotically equivalent to $\gamma k\rho^k$ as $k$ tends to infinity, for some positive number 
$\gamma$. This contradicts Condition (a). 
A similar argument yields the general case. 
\end{proof}

\begin{proof}[Proof of Theorem \ref{thm:masser}] 
Let us assume that $(T,\balpha)$ is admissible. By Lemma \ref{lem:ClasseM}, $T$ belongs to $\M$. Then Condition (b) implies that  
$\lim_{k\to\infty}T^k\balpha=0$. On the other hand, it is easy to see that Condition (c) implies that $\balpha$ is $T$-independent. 

Conversely, let us assume that $T$ belongs to $\M$, and that $\balpha\in(\Q^\star)^n$ is $T$-independent and 
satisifes $\lim_{k\to\infty}T^k\balpha=0$. By Lemma \ref{lem:U(T)}, $\balpha\in \mathcal U(T)$. Following Loxton and van der 
Poorten \cite{LvdP77}, since the matrix $T\in\M$ and $\balpha\in \mathcal U(T)$, Conditions (a) and (b) are satisfied with $\rho = \rho(T)$. 
Finally, Masser's vanishing theorem \cite{Mas82} implies that Condition (c) holds since $\balpha$ is $T$-independent. 
Hence, the 
pair $(T,\balpha)$ is admissible.  
\end{proof}

\section{A new vanishing theorem}\label{sec: vanishing}

As already mentioned in the introduction, a well-known feature of Mahler's method  
is the great importance of finding natural and general conditions that ensure   
non-vanishing conditions of type (c). Of course, our goal is to obtain a \emph{vanishing theorem} 
that can be applied to transformation matrices and points which are as general as possible. 
Our contribution to this problem is Theorem \ref{thm:lemmedezero}. 

In the case of a single transformation, 
after first results of Mahler, Kubota, Loxton and van der Poorten, 
Masser \cite{Mas82} solved this problem in a rather definitive way.   
However, in order to deal with several Mahler systems associated 
with different transformations, a more general vanishing theorem is needed. 
First results of this type were proved 
by Ku. Nishioka \cite{Ni94} and again by Masser \cite{Mas99}. 
More recently, Corvaja and Zannier \cite[Theorem 3]{CZ05} 
used the subspace theorem to prove  
a general theorem about the vanishing at $S$-units of analytic multivariate power series with algebraic coefficients. 
These authors already noticed that their result could be applied to Mahler's method.  
Though it is restricted to power series with algebraic coefficients, 
 the vanishing theorem of Corvaja and Zannier is very flexible. 
In this section, this flexibility is used to derive from their result our own vanishing theorem.  

In the framework of Mahler's method, several vanishing theorems have been formulated 
by saying that a non-zero multivariate power series cannot vanish  
at all points in some well-structured large sets, the latter are obtained by the iteration of the transformation matrix 
and usually involve arithmetic progressions.    
In order to prove Theorem \ref{thm: purity}, we need to replace these \og well-structured sets\fg{} by 
sets which remain large but offer more flexibility.  
We use the notion of  a piecewise syndetic set, which is classical in Ramsey theory and 
especially in its ergodic counterpart. As we just said, it can be though of as a notion of largeness for subsets of 
$\mathbb N$. Furthermore, Brown's lemma (see Lemma \ref{lem:syndetique}) shows that such sets are partition regular, and 
thus much more robust in terms of partitions than arithmetic progressions are.

\begin{defi}
A set $\mL\subset\N$ is said to be \emph{piecewise syndetic} if there exists a natural number $B$  such that  
for any given integer $M\geq 2$  there exist $l_1 < \cdots < l_M$ in $\mL$ such that
$$
l_{i+1} - l_i \leq B, \qquad 1 \leq i < M\, .
$$
In this case, we say that $B$ is a {\it bound}  for $\mL$.
\end{defi}

Let us recall that a subset of $\mathbb N$ is  said to be {\it syndetic}, or sometimes \emph{relatively dense}, if it has bounded gaps. 
A subset of $\mathbb N$ is said to be \emph{thick} if it contains arbitrarily long intervals. Thus piecewise syndetic sets are those that 
can be obtained as the intersection of  a syndetic set and a thick set. In the sequel of this section, as well as all along 
Section \ref{sec: admgen}, we will use heavily the following results. 

\begin{lem}\label{lem:syndetique}
Let $\mL \subset \N$ be a piecewise syndetic set with bound $B$. Then the following properties hold.  

\begin{enumerate}[label=(\roman*)]
\item[{\rm (i)}] If $\mL \subset \mL' \subset \N$, then $\mL'$ is also piecewise syndetic.
\item[{\rm (ii)}] If $\mL \subset \cup_{i=1}^s \mL_i$, then at least one of the $\mL_i$'s is piecewise syndetic.
\item[{\rm (iii)}] Let $l_0$ be a natural number. The set 
$$
\mL_0:=\left\{l \in \mL : (l+\{l_0,\ldots,l_0+B\}) \cap \mL \neq \emptyset \right\}
$$
is piecewise syndetic.
\item[{\rm (iv)}] The set $\mL$ contains arbitrarily long arithmetic progressions. 
\end{enumerate} 
\end{lem}

\begin{proof}
The point (i) immediately follows from the definition, while points (ii) and (iv)  correspond to classical results 
respectively known as  Brown's lemma (see \cite{Brown}) and Szemer\'edi's theorem \cite{Sz}. 
Let us prove (iii). Let $l_0$ and $M$ be two natural numbers and let $a$ 
be the smallest integer such that $aB \geq l_0$. 
Since $\mL$ is piecewise syndetic, there exist a sequence $l_1 < l_2 < \cdots < l_{(a+M)B}$ of elements of $\mL$ 
such that $l_{i+1}-l_i < B$. Let $i \leq M$, we have that $l_{i+aB} \geq l_0$, then there exists an integer $j \leq aB$ 
such that $l_0 \leq l_{i+j} \leq l_0+B$. Thus we have $l_i \in \mL_0$. This shows that $l_1,\ldots,l_M$ 
all belong in the set $\mL_0$. Consequently, $\mL_0$ is piecewise syndetic.  
\end{proof}

In order to prove Theorem \ref{thm: purity}, we need the following result that refines the vanishing 
theorem of Corvaja and Zannier in the context of Mahler's method, 
and also that extends it to series with coefficients in any finite dimensional $\Q$-vector space. 

\begin{thm}
\label{thm:lemmedezero}
Let $T_1,\ldots,T_r$ be matrices in $\M$ such that  
$$
\log \rho(T_i)/\log \rho(T_j)\not\in \mathbb Q\, \;\;\;\;\forall i\not=j\,.
$$  
Let us denote by $n_i$ the size of the matrix $T_i$ and set $N:=\sum_{i=1}^r n_i$.  
Set 
\begin{equation}\label{eq:Theta}
\Theta:=\left( \frac{1}{\log \rho(T_1)},\ldots,\frac{1}{\log \rho(T_r)}\right) \, .
\end{equation}
For every $l \in \mathbb N$, we let 
$\k_l:=(k_{l,1},\ldots,k_{l,r})$ denote a $r$-tuple of positive integers. Let us assume that  
\begin{equation}\label{eq:bounded}
\Vert \k_l -l \Theta \Vert = \mathcal O(1) \, .
\end{equation}
Let $\balpha:=(\balpha_1,\ldots,\balpha_r)$ be an algebraic point in $(\C^\star)^N$ 
such that $\balpha_i$ is $T_i$-independent for every $i$.  
Let $L \subset \C$ be a finite dimensional $\Q$-vector space and let 
$g \in L\{\z\}$ be a non-zero analytic function. Then the set
$$
\left\{ l \in \mathbb N : g(T_1^{k_{l,1}}\balpha_1,\ldots,T_r^{k_{l,r}}\balpha_r) = 0 \right\}
$$
is not piecewise syndetic. 
\end{thm}

 Applying Mahler's method to several Mahler systems requires some uniform speed of convergence to the origin for the orbits 
 of each algebraic point $\balpha_i$ under the matrix transformations $T_i$.  As noticed by van der Poorten \cite{vdP87}, one way to overcome 
 this difficulty is to iterate  each transformation $T_i$  $k_i$-times, and to choose the iteration vector $\k =(k_1,\ldots,k_r)$ so that asymptotically 
 the matrices $T_i^{k_i}$ have essentially the same radius of convergence.  
 As we shall see in Section \ref{sec: admgen}, this forces us to consider only iteration vectors $\k$ that remain at a bounded 
 distance from the real line $\mathbb R\Theta$, where $\Theta$ is defined by \eqref{eq:Theta}. This explains why the assumption \eqref{eq:bounded} 
 is natural in this framework.  
In the rest of this section, we set $T_{\k_l}\balpha:=(T_1^{k_{l,1}}\balpha_1,\ldots,T_r^{k_{l,r}}\balpha_r)$. 
Before proving Theorem \ref{thm:lemmedezero}, we need the following result.

\begin{lem}
\label{lem:admislocglob}
Let us keep the assumptions of Theorem  \ref{thm:lemmedezero}. 
Then, 
for every non-zero integer $N$-tuple $\bmu$, the set  
$$
\mL_0 := \left\{ l \in \N : (T_{\k_l}\balpha)^{\bmu} = 1 \right\}\, 
$$
is not piecewise syndetic.
\end{lem}

\begin{proof}
We argue by contradiction, assuming that $\mL_0$ is piecewise syndetic. 

For every pair of non-negative integers $(l,e)$, with $e >0$, we define the $r$-tuple $\be:=\be(l,e)=(e_1,\ldots,e_r)$ by  
\begin{equation}
\label{eq:sturm}
\be =\k_{l+e}-\k_l \,
\end{equation}
and we set 
$$
\E := \{\be(l,e),\ l,\, e \in \N  \} \,.
$$
Since $\k_l=l\Theta + \mathcal O(1)$, we obtain that 
\begin{equation}\label{eq:equivalencekl1}
\be(l,e) = e\Theta + \mathcal O(1) \, ,
\end{equation}
which shows that $\E$ is infinite. However, given any fixed positive integer $e_0$, the set 
$\{\be(l,e_0) : \ l\in \N\}$ is finite. 

We remark that there do not exist two complex numbers 
$\bbeta_1,\bbeta_2\not\in\{0,1\}$ such that there is a pair $(i,j)$, $1\leq i<j\leq r$, 
satisfying 
$$
\bbeta_1^{e_i}=\bbeta_2^{e_j}
$$
for infinitely many $\be=(e_1,\ldots,e_r) \in \E$. 
Indeed, let us assume that there exists an infinite set $\E_1$     
of such $\be=(e_1,\ldots,e_r) \in \E$. 
We first observe that $\gamma:=\log \bbeta_1/\log \bbeta_2=e_j/e_i$, thus $\gamma\in\mathbb Q$.  
On the other hand, one has  
$e_i= k_{i,l+e}-k_{i,l}$ and $e_j=k_{j,l+e}-k_{j,l}$, which gives 
$$
\gamma = \frac{k_{j,l+e}-k_{j,l}}{k_{i,l+e}-k_{i,l}} \,
$$
where we let $k_{n,m}$ denote the $n$-th coordinate of $\k_m$. 
Since $\E_1$ is infinite, there exist infinitely many $\be \in \E_1$  
such that $\be=\be(l,e)$ and where $e$ can be arbitrarily large.  
Letting $e$ tend to infinity, Equality \eqref{eq:equivalencekl1} implies  that 
\begin{eqnarray*}
\gamma = \frac{\log  \rho(T_i)}{\log \rho(T_j)} \in\mathbb Q \, ,
\end{eqnarray*}
which contradicts the multiplicative independence of $\rho(T_i)$ and $\rho(T_j)$. 
Let us recall that, by assumption, none of the eigenvalues of the matrices $T_i$ are equal to zero or a root of unity. 
Thus there exists a positive integer $e_0$ such that for every $e \geq e_0$, every $l \in \N$, 
 every eigenvalue $\lambda_i$ of $T_i$ and every eigenvalue $\lambda_j$ of $T_j$, $i \neq j$, then 
$$
\lambda_i^{e_i} \neq \lambda_j^{e_j} \, 
$$ 
where $\be=\be(l,e)=(e_1,\ldots,e_r)$. 
In particular, for such $\be$, every vector subspace $V$ of $\C^N$ that is invariant under the (right) action of the matrix  
\begin{equation}
\label{eq:matriceTe}
T_\be:=\left(\begin{array}{ccc} T_1^{e_1} && \\ & \ddots & \\ && T_r^{e_r} \end{array} \right)
\end{equation}
can be decomposed as 
$$
V = \bigoplus_{i=1}^r \pi_i^{-1}(V_i)\, ,
$$
where each  $V_i \subset \C^{n_i}$ is a vector space invariant by $T_i^{e_i}$, 
and where we let $\pi_i : \C^N=\C^{n_1+\cdots + n_r} \rightarrow \C^{n_i}$ 
denote the projection on the block corresponding to the matrix $T_i$. 

We are now ready to proceed with the proof of the lemma. 
Let us consider the column vector $\x$ whose transpose is the vector 
$$
(\log \alpha_{1,1},\log\alpha_{1,2},\ldots,\log\alpha_{1,n_1},\log\alpha_{2,1},\ldots,\log \alpha_{r,n_r})\ .
$$ 
We also set $\x_i:=\pi_i(\x)$.  
By assumption, we have that 
$$
\langle \bmu \, ,\,T_{\k_l}(\x)\rangle = 0 \, 
$$
for all $l \in \mL_0$. 
Let us denote by $U$ the orthogonal complement to the vector $\bmu$ in $\C^N$. 
This is a proper subspace of $\C^N$ defined over $\mathbb Q$, which contains all  vectors 
 $T_{\k_l} (\x)$, $l \in \mL_0$. 
 Given $\mL' \subset \mL_0$, we let $U(\mL')$ denote the smallest vector subspace of $\C^N$ over $\mathbb Q$  
 and containing all 
$T_{\k_l} (\x)$, $l \in \mL'$. 
It follows that $U(\mL_0) \subset U$. Furthermore, if $\mL'' \subset \mL'$, then $U(\mL'') \subset U(\mL')$. 
The subspace $U(\mL_0)$ being finite dimensional, there exists a subset $\mL_1 \subset \mL_0$ that is piecewise syndetic, 
and such that for all piecewise syndetic set $\mL' \subset \mL_1$,  one has 
$$
U(\mL')=U(\mL_1) \, .
$$
Let $B$ be a bound for $\mL_1$ and set 
$$
\E_0 := \{ \be(l,e) : e\in [e_0,e_0+B],l\in\mL_1,l+e \in \mL_1\}\, , 
$$
where $e_0$ is defined as in the first part of the proof. 
This is a finite set. 
Let 
$$
\mL_2 := \{l\in\mL_1 : \exists e \in [e_0,e_0+b] \mbox{ such that } l+e\in\mL_1\}\,.
$$
By Lemma \ref{lem:syndetique}, the set $\mL_2$ is piecewise syndetic. 
Now for $\be \in \E_0$, we set 
$$
\mL_\be :=\{l \in \mL_2 : T_{\be}(T_{\k_{l}}(\x))=T_{\k_{l+e}}(\x)\in U(\mL_1) \mbox{ with } \be=\be(l,e) \} \, .
$$  
If $l\in\mL_2$, then there exists $e\in [e_0,e_0+b]$ such that $l\in\mL_{\be}$ with $\be=\be(l,e)$. 
Hence, $\mL_2 \subset \cup_{\be \in \E_0} \mL_\be$. Since $\mL_2$ is piecewise syndetic, Lemma \ref{lem:syndetique} 
ensures the existence of $\be \in \E_0$  such that $\mL_{\be}$ is piecewise. Furthermore, $\mL_{\be}\subset \mL_1$.  
Thus we obtain that 
$$
U(\mL_{\be})= U(\mL_1) \, .
$$
Hence, the vector space $U(\mL_1)$ is closed under $T_{\be}$, for if $T_{\k_{l}}(\x)\in U(\mL_{\be})=U(\mL_1)$, then 
$T_{\be}(T_{\k_{l}}(\x))\in U(\mL_1)$.  
The first part of the proof shows that there is a decomposition of the form 
$$
U(\mL_1) = \bigoplus_{i=1}^r \pi_i^{-1}(U_i)\, , 
$$
where, for every $i$, $U_i = \pi_i(U(\mL_1)) \subset \C^{n_i}$ is a vector space closed under $T_i^{e_i}$, 
and defined over $\mathbb Q$, where $\be=(e_1,\ldots,e_r)$.  
Since $U(\mL_1)$ is a proper subspace of $\C^N$, there exists $i$, $1 \leq i \leq r$, such that 
$U_{i}$ is a proper subspace of $\C^{n_{i}}$. This space being defined over $\mathbb Q$, it has  a non-zero vector 
$\bnu_0 \in \Z^{n_{i}}$ in its orthogonal complement.  We thus have 
$$
\langle \bnu_0\, ,\, T_{i}^{e_{i}k_{i,l}}(\x_i)\rangle = 0\,,
$$
for all $l \in \mL_1$. 
The set $\mL_1$ being piecewise syndetic, we infer from the definition of the $\k_l$, that the sequence $(k_{i,l})_{l\in \mL_1}$ 
also forms a piecewise syndetic subset of $\N$. 
By property (iv) of Lemma \ref{lem:syndetique}, it contains arbitrarily long arithmetic progressions. 
Let us consider an arithmetic progression of length $n_i$ in $(k_{i,l})_{l\in \mL_1}$, say 
$$
a,a+b,a+2b\cdots,a+(n_i-1)b\, ,
$$
where $a,b \in \N$. 
We consider the sequence of  vector space 
$$V_0\subset \cdots \subset V_{n_i-1}\subset \bnu_0^\perp$$ 
defined by 
$$
V_j = {\rm Vect}_\mathbb{Q}\left\{T_i^{e_ia}(\x_i),\ldots,T_i^{e_i(a+jb)}(\x_i) \right\} \, .
$$
Since $\dim V_{n_i-1}<n_i$, there exists $j_0$ such that 
$V_{j_0}=V_{j_0+1}$. The vector space $V_{j_0}$ is then closed under $T^{e_ib}$ 
and we get that 
$$
\langle \bnu_0 \, ,\, T^{e_i(a+kb)}(\x_i)\rangle =0 \, ,
$$
or equivalently that 
$$
\left(T_i^{e_i(a+kb)}\balpha_i\right)^{\bnu_0} = 1 \,,  
$$
for all $k \in \N$. It follows that $\balpha_i$ is not $T_i$-independent, which provides a contradiction. 
This ends the proof. 
\end{proof}

 We are now ready to prove Theorem \ref{thm:lemmedezero}. 

\begin{proof}[Proof of Theorem \ref{thm:lemmedezero}] 
We keep the notation of the proof of Lemma \ref{lem:admislocglob}.  
We argue by induction on the dimension $t$ of the $\Q$-vector space $L$. 

Let us first assume that $t=1$. Then, dividing if necessary $g$ by some constant, 
there is no loss of generality to assume that $g \in \Q\{\z\}$. 
Let $u\in \Q$. We first show that for every non-zero integer $N$-tuple $\bmu$, the set 
$$
\mL_0:=\left\{ l \in \N : \left(T_{\k_l}\balpha\right)^\bmu = u \right\}
$$
is not piecewise syndetic. 
Let us assume by contradiction that  $\mL_0$ is syndetic and let $B$ be a bound for $\mL_0$. 
Set 
$$
\mathcal E := \{\be(l,e) : l \in \mL_0, l+e \in \mL_0, e\leq B\} \,.
$$
This is a finite set. 
For every $\be \in \mathcal E$, set 
$$
\mL_\be:=\left\{ l \in \N  : \left(T_{\k_l}\balpha\right)^{\bmu-\bmu T_\be} = 1 \right\}\, 
$$
and 
$$
\mL_1 :=\{ l\in \mL_0 : \exists e\leq B \mbox{ such that } l+e \in \mL_0\} \, .
$$ 
Lemma \ref{lem:syndetique} implies that $\mL_1$ is piecewise syndetic. 
For $l\in \mL_1$, there exists $\be=e(l,e) \in \mathcal E$ such that $e \leq B$ and $l+e \in \mL_0$. 
Then we obtain that  
\begin{eqnarray*}
\left(T_{\k_l}\balpha\right)^{\bmu-\bmu T_\be}&=&\frac{\left(T_{\k_l}\balpha\right)^{\bmu}}{\left(T_\be T_{\k_l}\balpha\right)^{\bmu}} \\
&=& u/u \\
&=& 1\, .
\end{eqnarray*}
We thus have $\mL_1 \subset \cup_{\be \in \mathcal E} \mL_\be$ and Lemma \ref{lem:syndetique} ensures the existence of  
$\be \in \mathcal E$ such that $\mL_\be$ is piecewise syndetic. 
By Lemma \ref{lem:admislocglob}, it thus follows that $\bmu - T_\be\bmu = 0$ for such a vector $\be$, which 
contradict  the fact that none of the $T_i$ has a root of unity as eigenvalue. Thus we 
conclude that $\mL_0$ is not piecewise syndetic. 

We set 
$$
\mL'_0:=\left\{ l \in \N : g(T_{\k_l}\balpha)=0 \right\}\, .
$$
Conditions (a) and (b) allow us to apply Theorem 3 of \cite{CZ05} to the sequence of points $(T_{\k_l}\balpha)_{l \in \N}$. 
In order to apply their results, we need to prove that the following three conditions are satisfied.  
\begin{itemize}
\item[{\rm (i)}] There exists a finite set of places $\S$ such that the algebraic points $T_{\k_l}\balpha$ are $\S$-units. 
\item[{\rm (ii)}] The sequence $(T_{\k_l}\balpha)_{l \in \N}$ tends to $0$. 
\item[{\rm (iii)}] One has $\log H(T_{\k_l}\balpha) = \mathcal O(-\log \Vert T_{\k_l}\balpha\Vert )$, where we let $H$ denote 
the absolute Weil height as defined at the end of section \ref{sec: main}.
\end{itemize}
Condition (i) is easy to check. Indeed, any finite number of non-zero algebraic numbers are $\S$-units  for some $\S$. 
The coordinates of the vector $\balpha$ are thus $\S$-units for some $\S$, and it follows directly that  
all $T_{\k_l}\balpha$ are then $\S$-units. 
Since by assumption $\balpha_i \in \mathcal U(T_i)$, the sequence $(T_{\k_l}\balpha)_{l \in \N}$ tends to $0$ and (ii) is satisfied. 
Next we check that (iii) holds. The matrix $T_i$ belonging to the class $\M$, it follows from \cite{LvdP77} that 
$$
\Vert T_i^{k}\Vert  = \mathcal O(\rho(T_i)^{k}) \qquad \text{ and } \qquad  -\log \Vert T_i^{k}\balpha_i\Vert  = \mathcal O(\rho(T_i)^{k})  
$$
for all non-negative integer $k$. The way we choose the vector $\Theta$ and of the vectors $\k_l$ ensures that 
$$
 \log H(T_{\k_l}\balpha) = \mathcal O(\rho^{|\k_l|}) \qquad \text{ and } \qquad -\log \Vert T_{\k_l}\balpha\Vert  \geq c \rho^{|\k_l|}\, , 
$$
where $\rho = e^{1/|\Theta|}$ and where $c$ is a positive real number 
(see the proof of lemma \ref{lem:convergence} for further detail). 
We deduce that $\log H(T_{\k_l}\balpha) = \mathcal O (-\log \Vert T_{\k_l}\balpha\Vert )$. 
Thus we can apply Theorem 3 of \cite{CZ05} to the sequence of algebraic points $(T_{\k_l}\balpha)_{l \in \N}$ 
and the function $g(\z)$. 
We obtain the existence of a finite number of $N$-tuples $\bmu_1,\ldots,\bmu_s$ and of algebraic numbers $u_1,\ldots,u_s$, 
such that 
$$
\mL'_0 \subset \bigcup_{i=1}^s \mL'_i 
$$ 
where 
$$
\mL'_i:=\left\{ l \in \N : \left(T_{\k_l}\balpha\right)^{\bmu_i} = u_i \right\}\, .
$$
As we have already proved that none of the sets $\mL'_i$ are piecewise syndetic, it follows from Lemma \ref{lem:syndetique} 
that $\mL'_0$ is not piecewise syndetic. This proves the theorem when $t=1$.

We assume now that $t \geq 2$. By induction, we also assume that the theorem is true when the dimension of $L$ 
is less than $t$. Let $a_1,\ldots,a_t$ be a basis of $L$ over $\Q$. 
We consider the decomposition
\begin{equation}
\label{eq:decompositionsurQ}
g(\z)=\sum_{i=1}^t a_ig_i(\z)\, ,
\end{equation}
where $g_i(\z) \in \Q\{\z\}$ for $1\leq i \leq t$. 
 Set 
$$
\mL_0:=\left\{ l \in \N : g(T_{\k_l}\balpha)=0 \right\}\, 
$$
and let us assume that $\mL_0$ is piecewise syndetic with bound $B$. 
We set  
$$
\mathcal E := \{\be(l,e) : l \in \mL_0, l+e \in \mL_0, e\leq B\} \,.
$$
For every $\be \in \mathcal E$, we consider the power series 
$$
h_\be(\z) := \sum_{i=1}^{t-1}a_i(g_i(\z)g_t(T_\be \z) - g_i(T_\be \z)g_t(\z)) \, .
$$
We also set 
$$
\mL_\be:=\left\{ l \in \N : h_\be(T_{\k_l}\balpha)=0 \right\}\, ,
$$
and 
$$
\mL_1 := \{l \in \mL_0 : \exists e\leq B \mbox{ such that } l+e \in \mL_0\}\, . 
$$ 
Therefore, for every $l \in \mL_1$, we have 
\begin{eqnarray*}
h_\be(T_{\k_l}\balpha) &=& g(T_{\k_l}\balpha)g_t(T_{k_{l+e}}\balpha) -  g(T_{\k_{l+e}}\balpha)g_t(T_{k_{l}}\balpha) \\
&=&0 \, ,
\end{eqnarray*}
for a $\be \in \mathcal E$. 
This shows that $\mL_1 \subset \cup_{\be\in\E} \mL_\be$. 
Since $\mL_1$ is piecewise syndetic and $\E$ is finite, Lemma \ref{lem:syndetique} implies that $\mL_\be$ 
is piecewise syndetic for some $\be\in\E$.  
By induction, we thus get that $h_\be(\z)=0$.  
Then, we infer from the $\Q$-linear independence of the $a_i$'s 
that  
$$
g_i(\z)g_t(T_\be \z) = g_i(T_\be \z)g_t(\z)
$$
for every $i$, $1\leq i\leq t-1$. 
We can now apply a result due to Ku. Nishioka \cite[Theorem 3.1]{Ni_Liv} that we recall now.   
Let $T$ be a non-singular square matrix with non-negative integer coefficients and such that no root of unity is an 
eigenvalue of $T$. If
 $h(\z) \in\C((\z))$ satisfies the equation 
$$
h(T\z)=c h(\z)+d 
$$
for some $c,d\in\C$, then $h \in \C$. 
The matrix $T_\be$ satisfies the assumption of this theorem, so we can apply it to 
the power series $h_i(\z)=g_i(\z)/g_t(\z)$. We deduce that for every $i$, $1\leq i\leq t-1$, there exists $\gamma_i \in \C$ 
such that $g_i(\z)=\gamma_i g_t(\z)$. We can thus write $g(\z) = a g_t(\z)$ with $a = \sum a_i\gamma_i$, 
which corresponds to the case $t=1$. In that case, we already proved that $\mL_0$  cannot be piecewise syndetic, a contradiction. 
This ends the proof. 
\end{proof}


\section{Mahler's method in families}\label{sec: families}

In this section, we state Theorem \ref{thm: families} which is an axiomatic result concerning 
the algebraic relations of values of several Mahler systems at algebraic points satisfying some ad hoc admissibility 
conditions called (A), (B), and (C), which replace Conditions (a), (b), and (c) in this more general framework.


\subsection{Families of Mahler systems}

Let $r$ be a positive integer. For every $i$, $1\leq i \leq r$, we consider a regular singular Mahler system 
of the form 
\begin{equation}\stepcounter{equation}
\label{eq:mahler2}\tag{\theequation .i}
\left(\begin{array}{c} f_{i,1}(\z_i) \\ \vdots \\ f_{i,m_i}(\z_i) 
\end{array} \right) = A_i(\z_i)\left(\begin{array}{c} 
f_{i,1}(T_i\z_i) \\ \vdots \\ f_{i,m_i}(T_i\z_i)\end{array}\right)\, ,
\end{equation}
where $n_i$ and $m_i$ are positive integers,  
$\z_i:=(z_{i,1},\ldots,z_{i,n_i})$ is a vector of indeterminates, $T_i$ is an $n_i\times n_i$ matrix with non-negative coefficients,  
$A_i(\z_i)$ belongs to  $\in{\rm GL}_{m_i}(\Q(\z_i))$, 
and the functions $f_{i,1}(\z_i),\ldots,f_{i,m_i}(\z_i)$ belong to $\Q\{\z_i\}$. 
Note that we have to replace $A_i(\z_i)$ by $A_i(\z_i)^{-1}$ to obtain a system as in \eqref{eq:mahler}. However, it is 
more natural in our proof to work with systems written in the form \eqref{eq:mahler2}.

In order to lighten the notation, we let $\f_i(\z_i)$ denote the column vector formed by  
the functions $f_{i,1}(\z_i),\ldots,f_{i,m_i}(\z_i)$. We will also set 
\begin{equation}\label{eq:dimensions}
M := \sum_{i=1}^r m_i \qquad \text{ and } \qquad N := \sum_{i=1}^r n_i.
\end{equation}
Iterating $k$ times the system \eqref{eq:mahler2}, one obtains the new system 
\begin{equation}\stepcounter{equation}
\label{eq:mahleritere}\tag{\theequation .i}
\f_i(\z_i) = A_{i,k}(\z_i)\f_{i}(T_i^k\z_i) \,,
\end{equation}
where we let $A_{i,k}$ denote the $k$-th iteration of the matrix $A_i$ by the transformation $T_i$, that is,  
$$
A_{i,k}(\z_i):=A_i(\z_i)A_i(T\z_i) \cdots A_i(T^{k-1}\z_i) \,.
$$
Set $\z:=(\z_1,\ldots,\z_r)$ and by abuse of notation $\f_{i}(\z):=\f_{i}(\z_i)$. 
For every $r$-tuple of positive integers $\k=(k_1,\cdots,k_r)$, one can collect together the systems \eqref{eq:mahleritere} 
in a single one as follows: 
\begin{equation}
\label{eq: blocks}
\left(\begin{array}{c}  \f_{1}(\z) \\ \vdots \\ \vdots \\ 
\f_{r}(\z)\end{array} \right) = \left(\begin{array}{cccc} 
A_{1,k_1}(\z_1) & & & \\ 
& \ddots & & \\
&& \ddots & \\
&&& A_{r,k_r}(\z_r)
\end{array}
\right)
\left(\begin{array}{c}   \f_{1}(T_\k\z) \\ \vdots \\ \vdots \\ 
\f_{r}(T_\k\z) \end{array}\right) \, ,
\end{equation}
where we let $T_\k$ denote the block diagonal matrix ${\rm diag}(T_1^{k_1},\ldots,T_r^{k_r})$. 
Finally, we let denote by $\f(\z)$ the column vector formed by all functions $f_{i,j}(\z_i)$, and by $A_\k(\z)$ 
the block diagonal matrix defined so that \eqref{eq: blocks} can be shortened to
\begin{equation}
\label{eq: blockcompact}
\f(\z) = A_\k(\z)\f(T_\k\z) \,.
\end{equation}
We keep these notations for the rest of the paper. 


\subsection{Multivariate exponential polynomials}

By definition, every matrix $A_i(\z_i)$ is conjugated, in the sense of Definition  \ref{def: regular singular}, 
to a matrix $B_i \in {\rm GL}_{m_i}(\Q)$. Let $\Gamma \subset \C^\star$ denote the multiplicative 
group generated by all eigenvalues of the matrices $B_i$. Iterating $k$ times the system \eqref{eq:mahler2} 
leads to the new system \eqref{eq:mahleritere}, and the corresponding matrix $B_i$ is then transformed 
to $B_i^k$. Iterating each system a suitable number of times if needed, one can assume without loss of 
generality that $\Gamma$ is torsion-free. Let $\R_{\Gamma,r}$ denote 
the $\Z$-module generated by the image of all maps of the form: 
\begin{equation}
\label{eq:polyexponentiel}
\begin{array}{rcl}
 \N^r & \rightarrow & \Z(\Gamma)
\\ 
\k=(k_1,\ldots,k_r) & \mapsto & \prod_{i=1}^r \left(\gamma_i^{k_i}
k_i^{j_i} \right)
\end{array}
\end{equation}
where $\gamma_1,\ldots,\gamma_r \in \Gamma$ and $j_1,\ldots,j_r \in \N$. 
Elements of $\R_{\Gamma,r}$ are called $(\Gamma,r)$-exponential polynomials. 
Let $\A$ be a ring with zero characteristic. One defines the $\A$-algebra of 
$(\Gamma,r)$-exponential polynomials $\R_{\Gamma,r} \otimes_\Z \A$ by extension 
of scalars to $\A$.


\subsection{Statement of the axiomatic theorem}\label{subsec: axiomatic}

Again, there are some rather natural conditions that seem to be inherent to our 
generalization of Mahler's method. 

\begin{defi}\label{def: globaladmissibility}
Let $\balpha_1,\ldots,\balpha_r$ be algebraic points with non-zero coordinates. 
The family of pairs $(T_i,\balpha_i)$ is \emph{admissible} if there exists an infinite set  
$\K \subset \N^r$ and a real number $\rho>1$ such that the following 
conditions hold. 
\begin{enumerate}

\item[\rm{(A)}]   The coefficients of the block diagonal matrix $T_\k$ belong to 
$\mathcal O(\rho^{|\k|})$, for $\k \in \K$. 

\item[\rm{(B)}] $\log \Vert T_\k \balpha\Vert  \leq -c\rho^{|\k|}$, for some positive real number $c$ and all $\k \in \K$.

\item[\rm{(C)}]  If $L$ is a finite-dimensional $\Q$-vector space and 
$\psi \in \R_{\Gamma,r} \otimes_\Z L\{\z\}$ is such that the family 
$\left(\psi(\k,\z)\right)_{\k \in \K}$ is not identically zero, then $\psi(\k,T_\k\balpha)\neq 0$ 
for infinitely many $\k \in \K$.
\end{enumerate}
\end{defi}
  
These admissibility conditions are studied in Section \ref{sec: admgen}. 
One can now state the main result of this section. 

\begin{thm}
\label{thm: families}
Let us consider $r$ regular singular systems \eqref{eq:mahler2}. 
Let $\balpha=(\balpha_1,\cdots,\balpha_r) \in (\C^\star)^N$ an algebraic point such that the 
family $(T_i,\balpha_i)_{1\leq i \leq r}$ is admissible and every 
$\balpha_i$ is regular.  Then if $P \in \Q[\X]$, $\X:=(X_{i,j})_{ i\leq r,\, j \leq m_i}$ is a polynomial, 
homogeneous of degree $d_i$ in the indeterminates 
$(X_{i,j})_{j \leq m_i}$, such that 
$$
P(f_{1,1}(\balpha_1),\ldots,f_{1,m_1}(\balpha_1),f_{2,1}(\balpha_2),\ldots,f_{r,m_r}(\balpha_r)) = 0 \,,
$$
Then there exists a polynomial $Q\in\Q[\z,\X]$, homogeneous of degree $d_i$ in the 
indeterminates $(X_{i,j})_{j \leq m_i}$, such that
$$
Q(\z,f_{1,1}(\z_1),\ldots,f_{r,m_r}(\z_r))=0 \;\;\; \mbox{ and } \;\;\;Q(\balpha,\X)=P(\X) \,.
$$
\end{thm}

Adding if necessary the function identically equal to $1$ to our systems shows that Theorem 
\ref{thm: families} remains true when $P$ is not homogeneous. 
Section \ref{sec: mainproof} is devoted to the proof of Theorem \ref{thm: families}. 
We stress that, in the case where $r=1$, one recovers Theorem \ref{thm: permanence} 
by taking $\K = \N$.

 
\section{Proof of Theorem \ref{thm: families}}\label{sec: mainproof}

Our proof of Theorem \ref{thm: families} follows the same strategy and steps as the proof of the main 
result of \cite{LvdP82}. However, the proof of Lemma 5 in \cite{LvdP82} is not complete. Furthermore, it is not clear 
that the definition of the so-called \emph{index} in  \cite{LvdP82} has the required multiplicative properties asked for Lemma 5. 
This deficiency has already been emphasized by Ku. Nishioka \cite{Ni90}. The present proof overcomes this difficulty and also 
provides more detailed argument at  several places. 

From now on, we assume that for every integer $i$, $1\leq i \leq r$, the matrix $A_i(\z_i)$ is conjugated to a constant 
invertible matrix through an analytic gauge transform.  That is, we assume that for every integer $i$, $1\leq i \leq r$, 
there exists a matrix 
$\Phi_i(\z_i) \in {\rm GL}_{m_i}(\Q\{\z_i\})$  
such that  
\begin{equation}\stepcounter{equation}
\label{eq:gaugeanalytic}\tag{\theequation .i}
\Phi_i(T\z_i)A_i(\z_i)\Phi_i^{-1}(\z_i) \in {\rm GL}_m(\Q) \,.
\end{equation}
We first prove Theorem \ref{thm: families} in that case. Then we 
show in Section \ref{sec:ramified} how to extend our proof to the general case.

 All along the proof of Theorem \ref{thm: families}, we let $\K$ denote a subset of $\mathbb N^r$ satisfying all properties required 
 by Definition \ref{def: globaladmissibility}. We also consider a vector of $M$ indeterminates 
 $\X:=(\X_1,\ldots,\X_{r})$, where $\X_i := (X_{i,1},\ldots,X_{i,m_i})$ and $M$ is defined as in \eqref{eq:dimensions}. 
  Let us now assume that $P\in \Q[\X]$ is a polynomial, homogeneous of degree $d_i$ in $\X_i$, for each $i$, such that
$$
P(\f(\balpha))=0 \,.
$$  
Let $s$ denote the number of distinct monomials of degree exactly $d_i$ in $\X_i$, for each $i$, and let 
us denote by $\X^{\bmu_1},\ldots,\X^{\bmu_s}$ these monomials, where $\bmu_1,\ldots,\bmu_s$ are $M$-tuple of non-negative integers. Then the polynomial $P(\X)$ can be uniquely decomposed as  
$$
P(\X) = \sum_{j=1}^s \tau_j \X^{\bmu_j} \,,
$$
where $\tau_j\in \Q$. Set $\btau:=(\tau_1,\ldots,\tau_s)$ and, 
given $s$ indeterminates $t_1,\ldots,t_s$, $\bt:=(t_1,\ldots,t_s)$. 
Then we define the form $F(\bt,\z)$ by  
\begin{equation}
\label{eq:defF}
F(\bt,\z):=\sum_{j=1}^s t_j \f(\z)^{\bmu_j} \, .
\end{equation}
This is a linear form in $\bt$. 
At the point $(\btau,\balpha)$, one has 
\begin{eqnarray}
\label{eq: annulationF}
F(\btau,\balpha)&=&\sum_{j=1}^s  \tau_j\f(\balpha)^{\bmu_j} \nonumber\\
&=&P(\f(\balpha))\\
&=&0 \,. \nonumber
\end{eqnarray}

\subsection{Iterated relations} 

For $1 \leq i \leq r$, we let $B_i$ be an $m_i \times m_i$ matrix with coefficients in some ring $\R$, 
and we set $B$ the $M\times M$ 
block diagonal matrix ${\rm diag}(B_1,\ldots,B_r)$. We notice that $(B(\X))^{\bmu_j}\in \R[\X]$ is a 
homogeneous polynomial of degree $d_i$ in each set of variables $\X_i$. 
We let $R_{j,l}(B)$ denote the elements of $\R$ defined by 
$$
\left(B(\X)\right)^{\bmu_j} = \sum_{l=1}^s R_{j,l}(B)\X^{\bmu_l} \,.
$$
We stress that the $R_{j,l}$ are polynomials of degree $d:=\max\{d_1,\ldots,d_r\}$ in the coefficients of 
the matrix $B$. 
Let $\k \in \N^r$, we infer from \eqref{eq: blockcompact} that 
\begin{eqnarray*}
F(\bt,\z) & = & \sum_{j=1}^s t_j \f(\z)^{\bmu_j}
\\& =& \sum_{j=1}^s t_j \left(A_\k(\z)\f(T_\k\z)\right)^{\bmu_j}
\\ & = & F\left(\left(\sum_{j=1}^s t_jR_{j,l}\left(A_\k(\z)\right) \right)_{l\leq s},T_\k\z\right).
\end{eqnarray*}
Set 
\begin{equation}\label{eq: taujk}
\tau_{l,\k} := \sum_{j=1}^s \tau_jR_{j,l}\left(A_\k(\balpha)\right) \in \Q
\mbox{ and } \btau_\k:=(\tau_{1,\k},\ldots,\tau_{s,\k}) \, .
\end{equation}
It follows from \eqref{eq: annulationF} that 
\begin{equation}
\label{eq: annulationFitere}
F(\btau_\k,T_\k\balpha)=0 \,,
\end{equation}
for all $\k \in \N^r$. 


\subsection{Structure of the numbers $\tau_{j,\k}$} 

Here is the part of the proof where the restriction to regular singular systems is really needed.  
We use this property to connect the algebraic numbers $\tau_{j,\k}$ to values at $T_\k\balpha$ 
of multivariate exponential polynomials.  
This connection appears to be fundamental in the proof of Theorem \ref{thm: families}.

\begin{lem}\label{lem: connect}
For every $j$, $1\leq j \leq s$, there exists $\psi_j \in \R_{\Gamma,r} \otimes_\Z \C\{\z\}$ such that 
$$
\tau_{j,\k}  = \psi_{j}(\k,T_\k\balpha)\, ,
$$
for all $\k \in \N^r$.  Furthermore, there exists a finite dimensional $\Q$-vector space $L_0$ such that 
for all $\k \in \N^r$ and all $j$, $1\leq j \leq s$, the coefficients of the formal power series $\psi_{j}(\k,\z)$ belong 
to $L_0$.  
\end{lem}

According to Equation \eqref{eq:gaugeanalytic}, for every positive integer $i$, $1\leq i \leq r$, 
there exist a matrix $\Phi_i(\z_i)\in{\rm GL}_{m_i}(\Q\{\z_i\})$ and a 
matrix $B_i \in {\rm GL}_{m_i}(\Q)$ such that 
$$
B_i=\Phi_i(\z_i)^{-1}A_i(\z_i)\Phi_i(T_i\z_i)\, .
$$
Iterating this equation, for every positive integer $k$ we get that 
$$
B_i^k=\Phi_i(\z_i)^{-1}A_{i,k}(\z_i)\Phi_i(T_i^k\z_i) \,,
$$
from which we deduce that 
\begin{equation}
\label{eq:matricefondamentaleitere}
A_{i,k}(\z_i)=\Phi_i(\z_i)B_i^k\Phi_i^{-1}(T_i^k\z_i) \,.
\end{equation}
Given a $r$-tuple of positive integers $\k:=(k_1,\ldots,k_r)$, we thus write
\begin{equation}
\label{eq:matricefondamentalebloc}
A_{\k}(\z) = \Phi(\z)B_\k\Phi(T_\k\z)^{-1}\, ,
\end{equation}
where we set 
$$
\Phi(\z) :=
	\left(\begin{array}{ccc}
	\Phi_1(\z_1)&&
	\\& \ddots &
	\\ & &\Phi_r(\z_r)
	\end{array}\right)\, ,
$$
and 
$$
B_\k:=
	\left(\begin{array}{ccc}
	B_1^{k_1}&&
	\\& \ddots &
	\\ & &B_r^{k_r}
	\end{array}\right)\, .
$$
We need the following result. 

\begin{lem}\label{lem:regularitePhi}
For every $r$-tuple $\k \in \N^r$, the matrix $\Phi(\z)$ is well-defined and non-singular at $T_\k\balpha$.
\end{lem}

\begin{proof}
By assumption, the matrix $\Phi(\z)$ is well-defined in some neighborhood of the origin. It follows from (B) that for 
$\k \in \mathcal K$ such that $\vert \k\vert$ is large enough, the matrix $\Phi(\z)$ is well-defined at $T_\k\balpha$. 
Furthermore, for every $\k \in \N^r$, one has 
\begin{equation}\label{eq:Phi}
\Phi(\z) = A_\k(\z)\Phi(T_\k\z)B_\k^{-1} \, 
\end{equation}
and since the point $\balpha$ is regular, $A_\k(\z)$ is well-defined at $\balpha$ for all $\k$. 
Considering Equality \eqref{eq:Phi} for $\vert\k\vert$ large enough, it follows that $\Phi(\z)$ is well-defined at $\balpha$. 
Inverting \eqref{eq:Phi}, we obtain that 
\begin{equation}\label{eq:Phi2}
\Phi(T_\k\z)=A_\k(\z)^{-1}\Phi(\z)B_\k \, ,
\end{equation}
and since $A_\k(\z)$ is non-singular at $\balpha$ for all $\k$, we deduce that $\Phi(\z)$ is well-defined at $T_\k\balpha$ for all $\k$. 

By assumption, $\det \Phi(\z)\not=0$. It thus follows from (C) that there exists $\k \in \N^r$ such that $\det \Phi(T_\k\balpha)\neq 0$. 
That is, $\Phi(\z)$ is non-singular at $T_\k\balpha$. Using Equality \eqref{eq:Phi}, we get that $\Phi(\z)$ is non-singular at $\balpha$. 
Using \eqref{eq:Phi2}, we deduce that $\Phi(\z)$ is non-singular at $T_\k\balpha$ for all $\k$.
\end{proof}

\begin{proof}[Proof of Lemma \ref{lem: connect}] 
By Lemma \ref{lem:regularitePhi},  we can define the matrix 
$$
B(\k,\z):=\Phi(\balpha)B_\k\Phi(\z)^{-1}\, .
$$  
This matrix is block diagonal, well-defined at $T_\k\balpha$ for all $\k \in \N^r$ and one has 
$$
B(\k,T_\k\balpha)=A_\k(\balpha).
$$
 For $1\leq l\leq s$,  we set 
\begin{equation}
\label{eq:defpsi}
\psi_l(\k,\z):=\sum_{j=1}^s \tau_jR_{j,l}\left(B(\k,\z)\right)
\end{equation}
where the polynomials $R_{j,l}$ are defined in \eqref{eq: taujk}.  
We thus infer from \eqref{eq: taujk} that 
\begin{eqnarray*}
\psi_l(\k,T_\k\balpha)&=&\sum_{j=1}^s \tau_jR_{j,l}\left(B(\k,T_\k\balpha)\right)\\
&=&\sum_{j=1}^s \tau_jR_{j,l}\left(A_\k(\balpha)\right)\\ 
&=&\tau_{l,\k} \, .
\end{eqnarray*}
Using the Jordan decomposition of the matrices $B_i$, one can show that the maps $\k \mapsto \psi_{j}(\k,\z)$, $1 \leq j \leq s$, 
belong to $\R_{\Gamma,r} \otimes_\Z \mathbb L$, where we let 
$$
\mathbb L := \C\left(\z,\Phi(\z)\right)
$$ 
denote the field generated over $\C$ by the indeterminates $\z$ and the coefficients of the matrices $\Phi(\z)$. 

On the other hand, when $\k \in \N^r$ and $j$ are fixed, the power series $\psi_j(\k,\z)$  
has coefficients in the (finite-dimensional) $\Q$-vector space $L_0$ generated by the monomials of degree 
at most $d$ in the coefficients of the matrix $\Phi(\balpha)$. 
\end{proof}

In the sequel we will use the compact notation 
$$
\bpsi(\k,\z):=\left(\psi_1(\k,\z),\ldots,\psi_s(\k,\z)\right) \,.
$$


\subsection{Formalization of the field $\mathbb L$ and valuations}

This part brings a new contribution with respect to the strategy of \cite{LvdP82}. 
In \cite{LvdP82}, the authors define the index of an element 
$E = \sum_\bmu p_\bmu(\bt)\z^\bmu \in \C[\bt][[\z]]$ 
as the smallest integer $h$ such that there does not exist a polynomial $P \in \C[\bt,\z]$ whose 
coefficients agree with those of  $E$ for all powers $\z^{\bmu}$ with $\vert\bmu\vert\leq h$, 
and such that $P(\btau_\k,T_\k\balpha)=0$ for all $\k \in \K$. Lemma 5 in \cite{LvdP82} then claims 
that 
$$
{\rm index}\left(E_1(\bt,\z)E_2(\bt,\z)\right)={\rm index}\ E_1(\bt,\z)+{\rm index}\ E_2(\bt,\z) \, ,
$$
for every $E_1,E_2 \in \C[\bt][[\z]]$.  
In particular, the inequality 
$$
{\rm index}\left(E_1(\bt,\z)E_2(\bt,\z)\right) \leq {\rm index}\ E_1(\bt,\z)+{\rm index}\ E_2(\bt,\z) \, ,
$$
is used at a key point in there proof. However, it is not clear that something even approaching is true. 
It seems that these authors made the following mistake.  
They argue as if given $P_1(\z)$ a polynomial approximation of $E_1(\z)$ at order $r_1$ and $P_2(\z)$ a 
polynomial approximation of $E_2(\z)$ at order $r_2$, the polynomial $P_1(\z)P_2(\z)$ would provide a polynomial 
approximation of $E_1(\z)E_2(\z)$ at order $r_1+r_2$. 
This is clearly not true. Of course, one can  prove that 
$$
{\rm index}\left(E_1(\bt,\z)E_2(\bt,\z)\right) \geq \min\{{\rm index}\ E_1(\bt,\z),{\rm index}\ E_2(\bt,\z)\} \, \
$$
but this is of no help in their proof. 
In order to overcome this problem, we show here how to replace the field $\mathbb L$ by a 
Noetherian ring $\A$ which is just a quotient of a ring of polynomials.   
This allows us to avoid the use of the index  
and to work simply in terms of valuations associated with prime ideals.

\subsubsection{The ring $\A$} 
Let us note that $\mathbb L$ has finite transcendence degree over $\C(\z)$, say $\ell$.  
Among the coefficients of the matrices $\Phi(\z)$, we can pick $\phi_1(\z),\ldots,\phi_{\ell}(\z)$,  
which are algebraically independent over $\C(\z)$. Lemma \ref{lem:regularitePhi} ensures that 
the power series $\phi_i(\z)$ are well-defined at $T^\k\balpha$ for all $\k \in \N^r$. 
The field $\mathbb L$ is a finite algebraic extension of $\C(\z,\phi_1,\ldots,\phi_{\ell})$, say of degree $d_0$.  
Let $\varphi$ be a primitive element of $\mathbb L$, that is, such that $\varphi$ generates $\mathbb L$ over 
$\C(\z,\phi_1,\ldots,\phi_{\ell})$. Multiplying $\varphi$ by an element of $\C[\z,\phi_1,\ldots,\phi_{\ell}]$ if necessary, 
we can assume that $\varphi$ is integer over the ring $\C[\z,\phi_1,\ldots,\phi_{\ell}]$. 
The ring $\C[\z,\phi_1,\ldots,\phi_{\ell}][\varphi]$ is thus a 
free $\C[\z,\phi_1,\ldots,\phi_{\ell}]$-module of rank $d_0$, generated by 
$1,\varphi,\ldots,\varphi^{d_0-1}$.  This is also a subring of the ring $\C\{\z\}$. 
By Lemma \ref{lem:regularitePhi}, the series $\varphi(\z)$ 
is well-defined at $T^\k\balpha$ for all $\k \in \N^r$. 

Let us consider $Y_1,\ldots,Y_\ell, U$, $\ell+1$ indeterminates and let us denote by 
$R \in \C[\z,\phi_1,\ldots,\phi_{\ell}][U]$ the (monic) minimal polynomial of $\varphi$. 
Then we consider the ring 
$$
\A:=\frac{\C[\z,Y_1,\ldots,Y_{\ell},U]}{\mathcal I}\,,
$$
where we let $\mathcal I$ denote the ideal of $\C[\z,Y_1,\ldots,Y_{\ell},U]$ generating by the 
polynomial $R(\z,Y_1,\ldots,Y_\ell,U)$.

\subsubsection{Valuations in $\A$}\label{subsub:valuation}

We first note that the ring $\A$ is Noetherian. We let $(\z)$ denote the ideal generated over $\A$ by the elements  
$z_{1,1},\ldots,z_{r,n_r}$. This ideal is not necessarily prime. However, since $\A$ is Noetherian,  there 
exist distinct prime ideals $\p_1,\ldots,\p_h$ in $\A$ and positive integers $e_1,\ldots,e_h$, such that 
\begin{equation}\label{eq:decompositionideauxz}
\prod_{i=1}^h \p_i^{e_i} \subseteq (\z) \subseteq \bigcap_{i=1}^h \p_i\, .
\end{equation}
For every $i$, $1\leq i \leq h$, we let $\nu_i : \A \mapsto \N$ denote the valuation associated with the prime ideal $\p_i$. 
That is, for every  $a\in\A$, we have  
$$
\nu_i(a):=\sup\{s \in \N :  a \in \p_i^s\}\, .
$$
In particular, $\nu_i(0)=+\infty$. 
Similarly, we set 
$$
\nu_\z(a):=\sup\{s \in \N : a \in (\z)^s\}\, .
$$
But, $\nu_\z$ is not necessarily a valuation for $(\z)$ may not be prime. However, we infer from 
\eqref{eq:decompositionideauxz} that 
\begin{equation}
\label{eq:majorationnuz} \nu_\z(a)  \leq  \nu_i(a)\,  ,
\end{equation}
for all $a \in \A$ and all $i$, $1\leq i\leq h$. 

\subsubsection{Formalization of $\bpsi$} 

Let $\U:= U \bmod \mathcal I \in \A$.  
Since the $\phi_i$'s are algebraically independent over $\C(\z)$, there exists a ring isomorphism $\sigma$ defined from 
$\C[\z,\phi_1,\ldots,\phi_{\ell}][\varphi]$ to $\A$ by 
$$
\begin{array}{rrcl} 
\sigma :& \C[\z,\phi_1,\ldots,\phi_{\ell},\varphi]   &\to& \A
\\
& \phi_i& \mapsto & Y_i
\\ &\varphi  & \mapsto &  \U
\end{array}
$$
The field $\mathbb L$ is the field of fractions of $\C[\z,\phi_1,\ldots,\phi_\ell][\varphi]$, while the field 
$\C(\z,Y_1,\ldots,Y_\ell)[U]/\mathcal I$ 
is the field of fractions of $\A$. Here, we let $\mathcal I$ denote the ideal generating by $R$ in $\C(\z,Y_1,\ldots,Y_\ell)[U]$.  
Thus $\sigma$ extends to a field isomorphism 
$$
\sigma : \mathbb L \rightarrow   \C(\z,Y_1,\ldots,Y_\ell)[U]/\mathcal I \, .
$$
By definition, the matrix $\Phi(\z)$ has coefficients in $\mathbb L$. There thus exists 
$Q(\z,\phi_1,\ldots,\phi_\ell,\varphi) \in \C[\z,\phi_1,\ldots,\phi_\ell,\varphi]$  such that 
 the matrix $Q\Phi(\z)^{-1}$ has coefficients in the ring $\C[\z,\phi_1,\ldots,\phi_\ell,\varphi]$.  
 Applying $\sigma$ to the multivariate exponential polynomials $\psi_j(.,\z)$, 
we get that 
$$
Q(\z,Y_1,\ldots,Y_\ell,\U) \sigma\left(\psi_j(\k,\z)\right) \in \A \, ,
$$
for every $j$ and $\k$. 
Now setting 
$$
\chi_j(\k):=Q(\z,Y_1,\ldots,Y_\ell,\U) \sigma\left(\psi_j(\k,\z)\right)\, ,
$$
the maps $\k \mapsto \chi_j(\k)$ are elements of $\R_{\Gamma,r} \otimes_\Z \A$. 
In the sequel we will use the compact notation 
\begin{equation}\label{eq:chi}
\bchi(\k):=(\chi_1(\k),\ldots,\chi_s(\k)).
\end{equation}


\subsection{Vanishing of polynomials at $(\btau_\k,T_\k\balpha)$}

In this section, we describe the $\Q$-vector space of polynomials in $\Q[\bt,\z]$, which are homogeneous in  
$\bt$ and vanish at $(\btau_\k,T_\k\balpha)$. 

For every $\k\in\mathbb N^r$, we define the morphism: 
$$
{\rm ev}_\k : \left\{\begin{array}{ccc} \Q[\bt,\z] & \rightarrow & \A
\\ P & \mapsto & P(\bchi(\k),\z).\end{array} \right.
$$

\subsubsection{Vanishing Lemma}

Let $V_0$ denote the set of polynomials $P\in \Q[\bt,\z]$ which are homogeneous in $\bt$, 
and such that 
$$
P(\bchi(\k),\z)= 0 \,,
$$ 
for all $\k \in \K$.

\begin{lem}
\label{lem: vanishpol}
Let $P\in \Q[\bt,\z]$ be homogeneous in $\bt$. The following are equivalent. 
\begin{enumerate}
\item[{\rm (i)}] \label{cond1} For all but finitely many $\k \in \K$,  $P(\btau_\k,T_\k\balpha)= 0$.

\item[{\rm (ii)}]  \label{cond2} For all $\k \in \K$, $P(\btau_\k,T_\k\balpha)= 0$. 

\item[{\rm (iii)}]  \label{cond3} $P \in V_0$.
\end{enumerate}
\end{lem}

\begin{proof}
${\rm (ii)} \implies {\rm (i)}$.  Trivial. 

\noindent${\rm (iii)} \implies {\rm (ii)}$. Let us assume  that $P\in V_0$ with degree $\delta$ in $\bt$,  
so that $P(\bchi(\k),\z)=0$ for all $k \in \K$. 
As $\sigma$ is a ring isomorphism and $P$ is a polynomial, we obtain that for all $\k \in \K$,
\begin{eqnarray}
\label{ZeroPullBack}
P(\bpsi(\k,\z),\z)&=&P(\sigma^{-1}(Q^{-1}\bchi(\k)),\z) \nonumber\\
&=&\sigma^{-1}(Q^{-\delta}P(\bchi(\k),\z))\\
&=&0\, ,\nonumber
\end{eqnarray}
where we let $\delta$ denote the degree of $P$ in $\bt$. 
But for all $\k \in \N^r$, $\psi(\k,T_\k\balpha)=\btau_\k$, so that evaluating \eqref{ZeroPullBack} at 
$\z=T_\k\balpha$, 
we get that 
$$
P(\btau_\k,T_\k\balpha)=0,\qquad \text{ for all } \k \in \K \, ,
$$
as wanted.

\noindent${\rm (i)} \implies {\rm (iii)}$. 
Let $P\in \Q[\bt,\z]$ be homogeneous in $\bt$ with degree $\delta$ in $\bt$. 
Now, let us assume that, for all but finitely many $\k \in \K$, we have 
$$
P(\btau_\k,T_\k\balpha)= 0 \,.
$$ 
The map  
$$
\k \mapsto P(\bpsi(\k,\z),\z)
$$
belongs to $\R_{\Gamma,r} \otimes_\Z \C\{\z\}$. Furthermore, there exists a finite dimensional 
$\Q$-vector space $L$ such that $P(\bpsi(\k,\z),\z)\in\R_{\Gamma,r} \otimes_\Z L\{\z\}$ 
since we already observed that $\bpsi(\k,\z)\in\R_{\Gamma,r} \otimes_\Z L_0\{\z\}$, where 
$L_0$ is finite dimensional over $\Q$. 
Taking $\z=T_\k\balpha$, we obtain $\bpsi(\k,T_\k\balpha)=\btau_\k$, and thus 
$$
P(\bpsi(\k,T_\k\balpha),T_\k\balpha) = 0
$$ 
for all but finitely many $\k \in \K$. Then, Condition (C) of Definition \ref{def: globaladmissibility} gives 
$$
P(\bpsi(\k,\z),\z) = 0
$$
for all $\k \in \K$. Applying $\sigma$,  we get that 
\begin{eqnarray*}
P(\bchi(\k),\z)&=&Q^\delta P(Q^{-1}\bchi(\k),\z)\\
&=&Q^\delta \sigma\left(P(\bpsi(\k,\z),\z)\right) \\
&=& 0\,.
\end{eqnarray*}
Thus $P\in V_0$, which ends the proof. 
\end{proof}


\subsubsection{Estimation of the dimension of some vector spaces} 

Our Lemmas \ref{dimensionespace} and \ref{lem:majorationespacepolynome} mainly correspond  to Theorem 3 
and Lemma 4 in \cite{LvdP82}. However, the vector space $V_0$ considered here being not the same as the one defined 
in \cite{LvdP82}, we supply the reader with proofs of these two results.    

Given two positive integers $\delta_1$ and $\delta_2$, we let $V(\delta_1,\delta_2)$ denote 
the set of polynomials $P\in \Q[\bt,\z]$ which are homogeneous of degree 
$\delta_1$ in the indeterminates $\bt$, and whose total degree in $\z$ is at most $\delta_2$. It is a $\Q$-vector space. 
We then set 
$$
V_0(\delta_1,\delta_2):= V_0 \cap V(\delta_1,\delta_2)\,.
$$
We also consider the quotient space 
 $$
 \overline{V}(\delta_1,\delta_2):=V(\delta_1,\delta_2)/V_0(\delta_1,\delta_2)\, .
 $$
 We stress that the value of a polynomial at the point $(\btau_\k,T_\k\balpha)$, 
$\k \in \K$, only depends on its equivalent class in $\overline{V}$. 
This is a direct consequence of Lemma \ref{lem: vanishpol}. 
\defConstant{dim} \setcounter{constant}{0}

\begin{lem}\label{dimensionespace}
The dimension $v(\delta_1,\delta_2)$ of the $\Q$-vector space $\overline{V}(\delta_1,\delta_2)$ 
satisfies 
$$
v(\delta_1,\delta_2) \sim \Cdim(\delta_1)\delta_2^N\, ,
$$
where $\Cdim(\delta_1)$ is a positive real number that does not depend on $\delta_2$.
\end{lem}

\begin{proof}
Let 
$$
P:=\sum_{|\bnu|= \delta_1, |\bmu|\leq \delta_2} p_{\bnu,\bmu}\bt^\bnu\z^\bmu \,
$$
be in $V(\delta_1,\delta_2)$. 
Let $d_0$ be the degree of the field extension generated by $\U$ over $\C(\z,Y_1,\ldots,Y_{\ell})$. 
Then for all $\bnu$ such that $|\bnu| =  \delta_1$, we have 
$$
\bchi(\k)^\bnu= \sum_{|\bomega| \leq \delta'_1,0\leq j< d_0,|\kapp|\leq \delta_1''} 
S_{\bnu,\bomega,j,\kapp}(\k)\z^\bomega \U^jY_1^{\kappa_1}\cdots Y_{\ell}^{\kappa_{\ell}} \in \A \,, 
$$
where $\delta'_1$ and $\delta_1''$ only depend on $\delta_1$, and where the map 
$\k\mapsto S_{\bnu,\bomega,j,\kapp}(\k)$ is an 
element of $\R_{\Gamma,r}\otimes_\Z \C$. 
It thus follows that 
$$ 
P(\bchi(\k),\z) = \sum_{\lambd,j} \left(\sum_{|\bnu|= \delta_1, |\bomega|\leq  \delta'_1, |\kapp|\leq \delta_1''}  
S_{\bnu,\bomega,j,\kapp}(\k)p_{\bnu,\lambd-\bomega}\right) \z^\lambd \U^jY_1^{\kappa_1}\cdots Y_\ell^{\kappa_\ell} \,,
$$ 
where we set $p_{\bnu,\lambd-\bomega}=0$ when $\lambd-\bomega \notin \N^N$. 
A polynomial $P \in V(\delta_1,\delta_2)$ 
belongs to $V_0$ if and only if, for every $(N+\ell+1)$-tuple $(\lambd,j,\kapp)$, we have 
\begin{equation}\label{eq:annulationpolyexpo} 
\sum_{|\bnu|= \delta_1, |\bomega| \leq \delta'_1}  S_{\bnu,\bomega,j,\kapp}(\k)p_{\bnu,\lambd-\bomega} = 0\, ,
\end{equation}
for all $\k \in \K$. We consider a decomposition 
$$
\R_{\Gamma,r}\otimes_\Z \C = W \oplus W^\perp\, ,
$$
where we let $W$ denote the vector space formed by the sequences $S \in \R_{\Gamma,r}\otimes_\Z \C$ 
such that  $S(\k)=0$ for all $\k \in \K$. 
Given $S \in \R_{\Gamma,r}\otimes_\Z \C$, 
we let $S^\perp$ denote the projection parallel to $W$ of $S$, on $W^\perp$.  
For every $(N+\ell+1)$-tuple $(\lambd,j,\kapp)$, Equality \eqref{eq:annulationpolyexpo} is equivalent to 
$$
\sum_{|\bnu|= \delta_1, |\bomega| \leq \delta'_1}  
S^\perp_{\bnu,\bomega,j,\kapp}(.)p_{\bnu,\lambd-\bomega} = 0 \in \R_{\Gamma,r}\otimes_\Z \C\, .
$$
For every tuple $(\bnu,\bomega,j,\kapp)$, we write 
$$
S^\perp_{\bnu,\bomega,j,\kapp}(\k) = 
\sum_{i=1}^q \bet_i^{\k}\sum_{\vert \gamm\vert \leq u} s_{\bomega,j,\bnu,\kapp,i,\gamm} \k^\gamm \,,
$$
where $\bet_i=(\eta_{i,1},\ldots,\eta_{i,r}) \in \Gamma^r$ and $s_{\bomega,j,\bnu,\kapp,i,\gamm} \in \C$. 
Thus, $P \in V_0$ if, and only if, for all $(\lambd,j,\kapp,\gamm,i)$, 
we have 
\defconstant{a}\defconstant{b}
\begin{equation}
\label{definitionV0}
\sum_{|\bnu| = \delta_1, |\bomega|\leq \delta'_1}   s_{\bomega,j,\bnu,\kapp,i,\gamm}p_{\bnu,\lambd-\bomega} = 0.
\end{equation}
Let $\Lambda(\delta_1,\delta_2)$ denote the number of indices $(\lambd,j,\kapp,\gamm,i)$ 
where $\vert\lambd\vert\leq \delta_2$, $\lambd - \bomega \in \N^N$ for all $\bomega$ with $|\bomega|=\delta_1$, 
$j\leq d_0$, $\vert\kapp\vert \leq \delta_1''$, $\vert \gamm\vert\leq u$, 
and $i\leq q$.  Then 
$$
\Lambda(\delta_1,\delta_2)\sim \ca(\delta_1)\delta_2^N \,,
$$ 
as $\delta_2\to \infty$, and where $\ca(\delta_1)$ is a positive real number 
that does not depend on $\delta_2$. 

On the other hand, $V_0$ is defined by a number $L(\delta_1,\delta_2)$ of independent linear equations in the coefficients of 
$P$ given by \eqref{definitionV0}. The family of complex numbers $\{s_{\bomega,j,\bnu,\kapp,i,\gamm}\}$ is independent of $\delta_2$ 
and and we claim that these complex numbers are not all zero. 
Indeed, if we assume that $s_{\bomega,j,\bnu,\kapp,i,\gamm}=0$ for all indices $\bomega,j,\bnu,\kapp,i,\gamm$, 
then the maps $\k \mapsto S^\perp_{\bnu,\bomega,j,\kapp}(\k)$ are all identically zero. 
But $S_{\bnu,\bomega,j,\kapp}(\k)=0$ for all $\k \in \K$ implies that $\bchi(\k)^\bnu=0$ for all $\bnu$ with $\vert \bnu\vert=\delta_1$. 
It follows that $\chi_i(\k)=0$ for all $\k \in \K$ and all $i\in \{1,\ldots,s\}$. 
Applying the isomorphism $\sigma^{-1}$, we get that $\bpsi(\k,\z)=0$ for all $\k \in \K$. Finally, evaluating at $\z = T_\k\balpha$, 
we obtain that $\btau_\k =0$ for all $\k \in \K$, and thus $\btau=(\tau_1,\ldots,\tau_s)=0$. This provides a contradiction.  
Hence, $L(\delta_1,\delta_2)$ is nonzero. 
If $\lambd$ is such that $\lambd - \bomega \in \N^N$ for all $\bomega$ with $|\bomega|=\delta_1$, then 
the corresponding number of independent equations given by \eqref{definitionV0} is a non-zero number $\cb(\delta_1)$ 
that does not depend on $\lambd$.  
Hence, we have
\begin{eqnarray*}
L(\delta_1,\delta_2)&\sim&\cb(\delta_1)\Lambda(\delta_1,\delta_2) \\
&\sim& \ca(\delta_1)\cb(\delta_1)\delta_2^N\, .
\end{eqnarray*}
as $\delta_2\to \infty$. 
Setting $\Cdim(\delta_1):=\ca(\delta_1)\cb(\delta_1)$, we obtain that 
the dimension of $\overline{V}(\delta_1,\delta_2)$ 
satisfies 
$$
v(\delta_1,\delta_2) \sim \Cdim(\delta_1)\delta_2^N\,,
$$
as $\delta_2\to \infty$.
\end{proof}

\begin{lem}\label{lem:majorationespacepolynome}
For all pair of positive real numbers $(\delta_1,\delta_2)$, we have 
$$
v(2\delta_1,\delta_2) \leq (s+1)v(\delta_1,\delta_2) \,.
$$
\end{lem}

\begin{proof}
Let $P \in V(2\delta_1,\delta_2)$.  We claim that $P$ can be decomposed as 
\begin{equation}
\label{reduction}
P(\bt,\z)=P_0(\bt,\z)+\sum_{i=1}^s t_i^{\delta_1}P_i(\bt,\z),
\end{equation}
where $P_i \in V(\delta_1,\delta_2)$, $0 \leq i \leq s$. 
Indeed, let us write 
$$
P(\bt,\z)= \sum_{|\bnu| = 2\delta_1} p_\bnu(\z)\bt^\bnu \,.
$$
For every  $\bnu=(\nu_1,\ldots,\nu_d)$, there exists at most one $i:=i(\bnu)$ such that $\nu_i>\delta_1$. 
Set  
$$
P_i(\bt,\z):=\sum p_\bnu \bt^\bnu t_i^{-\delta_1}\,,
$$ 
where the sum runs along the set of $\bnu$ such that $i(\bnu)=i$. 
We also set $P_0 :=\sum p_\bnu \bt^\bnu$  where the sum runs along the set of $\bnu$ such that 
$\nu_i\leq \delta_1$ for all $i$. 
We thus get the decomposition  \eqref{reduction}. 
Now, if $Q_1,\ldots,Q_v$ is a basis of $\overline{V}(\delta_1,\delta_2)$, 
the set formed by $Q_j$ and $t_i^{\delta_1}Q_j$, for $1\leq j\leq v$ and $1\leq i\leq s$ 
is a generating set of $\overline{V}(2\delta_1,\delta_2)$.This ends the proof. 
\end{proof}


\subsection{Vanishing of $F(\bt,\z)$}

Let us recall that the function $F(\bt,\z)$ is defined by 
$$
F(\bt,\z)=\sum_{i=1}^s t_i \f(\z)^{\bmu_i} \,.
$$
By definition, $F(\bt,\z) \in \C[\bt][[\z]]$. Writing $F$ as a formal power series in $\z$, we get that 
$$
F(\bt,\z)=\sum_{\lambd \in \N^N} l_\lambd(\bt)\z^\lambd,
$$
where the $l_\lambd$ are linear forms in $\bt$. 
For a non-negative integer $q$, we let 
$$
F_q(\bt,\z):=\sum_{\lambd \in \N^N\;:\; |\lambd|< q} l_\lambd(\bt)\z^\lambd 
$$
denote the partial sum of  $F(\bt,\z)$ at order $q$ with respect to the variable $\z$. 
More generally, given $E(\bt,\z) = \sum_{\lambd} e_\lambd(\w)\z^\lambd \in \C[\bt][[\z]]$, 
we set 
$$
E_q(\bt,\z) := \sum_{\lambd \in \N^N\;:\; |\lambd|< q} e_\lambd(\bt)z^\lambd \in \C[\bt,\z]\, .
$$

 Our aim is now to prove the following result which replace Lemma 6 in \cite{LvdP82}.

\begin{lem}
\label{lem:nulliteF}
There exist $\k_0 \in \K$ and $i_0$, $1 \leq i_0 \leq h$, such that 
\begin{equation}
\label{eq:minorationvaluationF}
\nu_{i_0}(F_q(\bchi(\k_0),\z)) \geq q \,,
\end{equation}
for all non-negative integer $q$. 
Furthermore, we can choose $\k_0$ so that $Q(\z,\phi_1(\z),\ldots,\phi_\ell(\z),\varphi(\z))$ 
does not vanish at the point $T_{\k_0}\balpha$.
\end{lem}

\subsection{Proof of Lemma \ref{lem:nulliteF}}

The proof of Lemma \ref{lem:nulliteF} follows some classical arguments  introduced by Mahler \cite{Ma29}. 
We construct an auxiliary function using simultaneous (Pad\'e) approximation of the powers of $F(\bt,\z)$.
Assuming by contradiction that the conclusion of Lemma \ref{lem:nulliteF} does not hold, we ensure to this auxiliary function 
a high order of vanishing at $\z=0$. 
Providing an upper and a lower bound at $(\btau_\k,T_\k\balpha)$ for this function, we then obtain a contradiction.  

\subsubsection{Auxiliary function} 

For every pair of positive integers $(\delta_1,\delta_2)$, we consider a complement $V_1(\delta_1,\delta_2)$ to 
$V_0(\delta_1,\delta_2)$ in $V(\delta_1,\delta_2)$. 
We also set $V_2(\delta_1,\delta_2):= \bigoplus_{i=1}^{\delta_1}V_1(i,\delta_2)$. \defConstant{aux}

\begin{lem}
\label{lem:fonctionauxiliaire}
Let $q$ be a positive integer. For every $\delta_1$ large enough, and $\delta_2$ large enough with respect to 
$\delta_1$, there exist polynomials $P_0,\ldots,P_{\delta_1}\in V_2(\delta_1,\delta_2)$,  
and a positive real number $\Caux$ that depends neither on $\delta_1$, nor on $\delta_2$, nor on $q$, such that the following hold. 
\begin{enumerate}
\item[{\rm (1)}] $P_0 \neq 0$. 
\item[{\rm (2)}]  $E := \sum_{j=0}^{\delta_1} P_jF_q^j \in \Q[\bt,\z]$ satisfies 
$$
\nu_i(E(\bchi(\k),\z) \geq \Caux\delta_1^{1/N}\delta_2 - \delta_1\nu_{i}(F_q(\bchi(\k),\z))\,,
$$
for all $\k \in \K$ and all $i$, $1\leq i\leq h$. 
\end{enumerate}
\end{lem}

\begin{proof} 
Our construction follows the one in the proof of Lemma 6 in \cite{LvdP82}. However, we substitute the notion of valuation 
to the notion of index used there. Let $q$ be a positive integer. 

We first consider the point (2).  We construct a polynomial $E' \in \Q[\bt,\z]$ such that $E'_p$ belong $V_0(2\delta_1,p-1)$ 
for a $p$ large enough. Let us consider the following linear maps: 
{\renewcommand{\arraystretch}{2}
$$ 
\begin{array}{ccc}
{\renewcommand{\arraystretch}{1.3}
\left\{\begin{array}{cc}
\prod_{j=0}^{\delta_1} V_1(2\delta_1-j,\delta_2)
\\
(P_0,\ldots,P_{\delta_1})
\end{array} \right.
}
& \rightarrow &
{\renewcommand{\arraystretch}{1.3}
 \left\{
\begin{array}{c} \Q[\bt,\z]
\\ E':= \sum_{j=0}^{\delta_1}
 P_jF_q^j
\end{array} \right.
}
\\

& & \big{\downarrow}

\\
{\renewcommand{\arraystretch}{1.3}
\left\{\begin{array}{c} \overline{V}(2\delta_1,p-1) 
\\ 
 E'_p \mod V_0
 \end{array}\right.
 }& \leftarrow &
{\renewcommand{\arraystretch}{1.3}
\left\{\begin{array}{c} V(2\delta_1,p-1) 
\\ 
 E'_p
 \end{array}\right.
 }
\end{array}
$$}

These linear maps are well-defined. Indeed, the $P_j$'s are homogeneous polynomials of degree 
$2\delta_1-j$ in $\bt$, while $F$ is homogeneous of degree one in $\bt$, and thus $E'$ is homogeneous 
of degree $2\delta_1$ in $\bt$, and $E'_p \in V(2\delta_1,p-1)$. So this makes sense to consider 
$E'_p \bmod V_0(2\delta_1,p-1)$. 

By Lemma \ref{dimensionespace}, the vector space $\prod_{j=0}^{\delta_1} V_1(2\delta_1-j,\delta_2)$ 
has dimension at least equal to $\Cdim(\delta_1)\delta_1\delta_2^{N}/2$ when $\delta_2$ is large enough. 
By Lemma \ref{lem:majorationespacepolynome}, the vector space $\overline{V}(2\delta_1,p-1)$ has dimension at most 
$(s+1)v(\delta_1,p-1)$. But if $p$ is large enough, Lemma \ref{dimensionespace} ensures that 
$v(\delta_1,p-1) \leq 2\Cdim(\delta_1)\delta_1p^{N}$. 
For such a $p$, the vector space $\overline{V}(2\delta_1,p-1)$ has dimension at most $2(s+1)\Cdim(\delta_1)p^N$. 
We set  
$$
p := \left\lfloor \frac{\delta_1^{1/N}\delta_2}{5(s+1))^{1/N}} \right\rfloor\,,
$$
so that 
$$
2(s+1)\Cdim(\delta_1)p^N < \Cdim(\delta_1)\delta_1\delta_2^{N}/2\, .
$$ 
By comparison of these dimensions, we see that if $\delta_2$ is large enough, then the linear map defined by 
$(P_0,\ldots,P_{\delta_1})\mapsto E'_p(P_0,\ldots,P_{\delta_1})\mod V_0$ 
has a non-trivial kernel. That is, there exist $P_0,\ldots,P_{\delta_1}$ not all zero, such that $E'_p$ belongs to 
the subspace $V_0(2\delta_1,q-1)$. 
Considering $E'$, we have 
\begin{eqnarray*}
E'(\bchi(\k),\z) &=& E'_p(\bchi(\k),\z) + \sum_{|\lambd|\geq p} e_\lambd(\bchi(\k))\z^\lambd \\ 
&=& \sum_{|\lambd|\geq p} e_\lambd(\bchi(\k))\z^\lambd \in \Q[\bt][[\z]],
\end{eqnarray*}
since by construction $E'_p(\bchi(\k),\z) =0$. 
On the other hand, we have $\sum_{|\lambd|\geq p} e_\lambd(\bchi(\k))\z^\lambd \in (\z)^p$.  
Let $i$ be an integer with $1\leq i\leq h$. It thus follows from Inequality \eqref{eq:majorationnuz} that 
$$
\nu_i(E'(\bchi(\k),\z)) \geq p \geq \Caux \delta_1^{1/N}\delta_2\, ,
$$
where $\Caux$ does not depend on $\delta_1$, $\delta_2$, $\k$, $i$, and $q$. 
Let $v$ be the smallest index such that $P_v$ is non-zero. We set 
$$
E := \sum_{j\geq v} P_jF_q^{j-v}\,.
$$
We thus have $EF_q^v = E'$. 
 For $1 \leq i \leq h$, we obtain 
$$
\nu_i(E'(\bchi(\k),\z)) = \nu_i(E(\bchi(\k),\z)) + v\nu_i(F_q(\bchi(\k),\z)) \,.
$$
Thus for all $\delta_1$ and all $\delta_2$ large enough, we have 
\begin{eqnarray*}
\nu_i(E(\bchi(\k),\z))& = &\nu_i(E'(\bchi(\k),\z)) - v\nu_i(F_q(\bchi(\k),\z)) \\
& \geq & \Caux \delta_1^{1/N}\delta_2 - \delta_1\nu_i(F_q(\bchi(\k),\z)),
\end{eqnarray*}
for all $\k \in \K$ and all $i$, $1\leq i\leq h$. This ends the proof.   
\end{proof}

\subsubsection{Choice of an infinite subset of $\K$}

Let us denote by $\K_0$ the set of $\k_0 \in \K$ such that $Q(\z,\phi_1(\z),\ldots,\phi_\ell(\z),\varphi(\z))$ 
does not vanish at $T_{\k_0}\balpha$. From now on, and until the end of the proof of Lemma \ref{lem:nulliteF}, 
we argue by contradiction, assuming that for all $\k_0 \in \K_0$ and all $i$, $1 \leq i \leq h$, there exist an integer 
$q:=q(\k,i)$ tel que 
\begin{equation}\label{eq:assumption}
\nu_i(F_{q}(\bchi(\k),\z)) < q \,.
\end{equation}

\begin{lem}
\label{lem:minorationFgenerale}
Let $P_0$ denote the polynomial constructed in Lemma \ref{lem:fonctionauxiliaire}. 
Under the assumption \eqref{eq:assumption}, there exists an integer $q_0$ and an 
infinite subset $\K'\subset \K_0$ such that, for every $\k \in \K'$ and 
$i$, $1 \leq i \leq h$, the two following properties hold.  
\begin{itemize}
\item[{\rm (1)}] $P_0(\btau_\k,T_\k\balpha)\neq 0$.

\item[{\rm (2)}] $\nu_i(F_q(\bchi(\k),\z)) < q_0$, \mbox{ for all }$q \geq q_0$.
\end{itemize}

\end{lem}

\begin{proof}
By construction, we have that $P_0 \notin V_0$. There thus exist infinitely many 
$\k \in \K$ such that $P(\bpsi(\k,T_\k\balpha),T_\k\balpha) \neq 0$. 
In particular, the map $\k \mapsto P(\bpsi(\k,\z),\z)$ is not identically zero on $\K$. 
Consequently, the map $\k \mapsto P(\bpsi(\k,\z),\z)Q(\z,\phi_1(\z),\ldots,\phi_\ell(\z),\varphi(\z))$  is also 
not identically zero on $\K$. We thus infer from Condition (C) that  there exist infinitely many $\k \in \K$ 
such that 
$$
P_0(\btau_\k,T_\k\balpha)\neq 0, \, \text{ and } 
Q(T_\k \balpha,\phi_1(T_\k \balpha),\ldots,\phi_\ell(T_\k \balpha),\varphi(T_\k \balpha)) \neq 0 \, .
$$
In particular, there exists $\k_0 \in \K_0$ such that $P_0(\btau_\k,T_\k\balpha)\neq 0$. 
Let us consider an integer $i_0$, $1 \leq i_0 \leq h$, and set $q:=q(\k_0,i_0)$. 
We can write  
$$
F_{q}(\bt,\z) = \sum_{|\lambd| < q}\sum_{i=1}^{s} l_{i,\lambd}t_i\z^\lambda,
$$
where $l_{i,\lambd} \in \Q$ for every $(i,\lambd)$. For $1 \leq i \leq s$, we also write 
$$
\chi_i(\k)=\sum_{\bomega,\kapp,j} s_{i,\bomega,\kapp,j}(\k) \z^\bomega \Y^{\kapp}\U^j,
$$
where $j < d_0$, and where both $\bomega \in \N^N$ and $\kapp \in \N^\ell$ belong to a finite set. 
In this decomposition, the maps $\k \mapsto s_{i,\bomega,\kapp,j}(\k)$ belong to $\R_{\Gamma,r} \otimes_\Z \C$. 
We can thus write 
$$
F_{q}(\bchi(\k),\z)
=\sum_{|\lambd|\leq q} \sum_{i=1}^s l_{i,\lambd}\sum_{\bomega,\kapp,j} 
s_{i,\bomega,\kapp,j}(\k)  \Y^\kapp\U^j\z^{\lambd+\bomega} \,.
$$
By \eqref{eq:assumption}, we have $\nu_{i_0}(F_{q}(\bchi(\k_0),\z)) < q$. 
Let us denote by $\Lambda$ an upper bound for the norm of the vectors $\lambd+\btau$, and by $\Omega$ 
an upper bound for the norm of the vectors $\bmu$ occurring in the previous sum. 
Let $V\subset\A$ be the $\C$-vector space formed by the polynomials of degree at most $\Lambda$ in $\z$, at most $\Omega$ in $\Y$, 
and at most $d_0-1$ in $\U$, and let us denote by $V^\star$ its dual space. Let us also consider the vector space $V':=\p_{i_0}^q \cap V$. Let $w_1,\ldots,w_l$ be a basis of the dual of $V'$, in $V^\star$. Let  
$$
a = \sum_{\bomega,\kapp,j} a_{\bomega,\kapp,j}\z^\bomega \Y^{\kapp}\U^j \in V \, ,
$$
with $a_{\bomega,\kapp,j} \in \C$. For every $e$,  $1 \leq e \leq l$, we  
have a decomposition 
$$
w_e(a) = \sum_{\kapp,\bmu,j} w_{e,\bomega,\kapp,j}a_{\bomega,\kapp,j}\, ,
$$
with $w_{e,\bomega,\kapp,j} \in \C$. 
Then, given $a \in V$, we have 
\begin{equation}\label{eq:equivalencepuissanceideal}
\nu_{i_0}(a) \geq q \iff w_e(a)=0, \text{ for } 1 \leq e \leq l\, .
\end{equation}
By assumption, there thus exists $e_0 \leq l$ such that 
$$
w_{e_0}(F_{q}(\bchi(\k_0),\z))\neq 0 \, .
$$
We set 
$$
\epsilon_{i_0} : \k \mapsto \sum_{i,\lambd,\btau,\bmu,j} 
w_{e_0,\kapp,\bmu,j}l_{i,\lambd}s_{i,\btau,\bmu,j}(\k)  \in \R_{\Gamma,r} \otimes_\Z \C \, .
$$
This definition ensures that $\epsilon_{i_0}(\k)= w_{e_0}(F_{s}(\bchi(\k),\z))$. 
In particular, we have that $\epsilon_{i_0}(\k_0) \neq 0$. By 
\eqref{eq:equivalencepuissanceideal}, we also have that  
$$
\epsilon_{i_0}(\k) \neq 0 \Rightarrow \nu_{i_0}(F_{q}(\bchi(\k),\z)) < q \,,
$$
for all $\k \in \N^r$. 
Using the same construction for every $i$, $1 \leq i \leq h$, we obtain maps 
$\epsilon_i \in \R_{\Gamma,r} \otimes_\Z \C$ such that $\epsilon_i(\k_0) \neq 0$ and 
\begin{equation}\label{eq:implicationepsilon}
\epsilon_i(\k) \neq 0 \Rightarrow \nu_i(F_{q(\k_0,i)}(\bchi(\k),\z)) < q(\k_0,i)\, .
\end{equation}
The map 
$$
\k \mapsto P_0(\bpsi(\k,\z),\z)Q(\z,\phi_1(\z),\ldots,\phi_\ell(\z),\varphi(\z))\prod_{i=1}^l\epsilon_i(\k)
$$ 
belongs to $\R_{\Gamma,r} \otimes_\Z L\{\z\}$ for some finite dimensional $\Q$-vector space $L$.  
Furthermore, it does not vanish 
at $\k_0$. We thus infer from condition (C) of Definition \ref{def: globaladmissibility} 
that there exists an infinite set $\K'\subset \K_0$ such that 
\begin{equation}
\label{eq:nonnullitesimultanee}
P_0(\btau_\k,T_\k\balpha)\neq 0\, \text{ and } \epsilon_i(\k) \neq 0 \, ,
\end{equation}
for all $i$, $1 \leq i \leq l$, and all $\k \in \K'$. 

For $k \in \K'$, we set $q_i := q(\k_0,i)$. Then we infer from 
\eqref{eq:implicationepsilon} and \eqref{eq:nonnullitesimultanee} that 
$$
\nu_i(F_{q_i}(\bchi(\k),\z)) < q_i\,.
$$
Given $i$, $1 \leq i \leq h$, we let $q>q_i$. 
Then, we can write 
$
F_q = F_{q_i} + R_{q,q_i}
$
where 
$$
R_{q,q_i} = \sum_{q_i \leq  |\lambd| < q} l_\lambd(\bt)\z^\lambd \,.
$$
By definiton, we have $\nu_i(R_{q,q_i}(\bchi(\k))) \geq  q_i $. 
But, on the other hand, $\nu_i(F_{q_i}(\bchi(\k),\z))<q_i$,  for all $k \in \K'$. 
It follows that 
$$
\nu_i(F_q(\bchi(\k),\z) ) < q_i
$$
for all $q > q_i$. Taking $q_0:= \max\{q_1,\ldots,q_h\}$, this ends the proof. 
\end{proof}


The rest of the proof of Lemma \ref{lem:nulliteF} consists in proving upper and lower bounds for the 
auxiliary function $E$ at the point $(\btau_\k,T_\k\balpha)$, and then to derive a contradiction 
for $\vert \k\vert$ large enough. 
Similar bounds are given in \cite{LvdP82} without too much detail. As our auxiliary function $E$ 
is not exactly defined as the one in  \cite{LvdP82}, we provide explicit computation for these bounds. 
In the rest of the proof, we consider a fix integer $q \geq q_0$, where $q_0$ is given by Lemma 
\ref{lem:minorationFgenerale}, and we let $E \in \Q[\w,\z]$ denote the auxiliary function given by 
Lemma \ref{lem:fonctionauxiliaire} for this integer $q$.


\subsubsection{Upper bound for $|E(\btau_\k,T_\k\balpha)|$}\label{subsec: upper}

By Lemma \ref{lem:fonctionauxiliaire}, we have 
\defConstant{auxter} \defConstant{auxbis}
$$
\nu_i(E(\bchi(\k),\z)) \geq  \Caux \delta_1^{1/N}\delta_2 - \delta_1\nu_i(F_{q})\,,
$$
for every $\k\in\K$ and every $i$, $1 \leq i \leq h$.  
By Lemma \ref{lem:minorationFgenerale},  
we have $\nu_i(F_{q}) < q_0$ for every $\k \in \K'$, and every $i$, $1 \leq i \leq h$. 
Since $\delta_2\geq \delta_1$,  for $\delta_1$ large enough, 
there exists a positive real number $\Cauxter$ such that 
\begin{equation}
\label{eq:minorationvaluationE}
E(\bchi(\k),\z) \in \p_i^{\left\lfloor \Cauxter \delta_1^{1/N}\delta_2\right\rfloor} \,
\end{equation}
for every $i$, $1 \leq i \leq h$, and every  $\k \in \K'$. 
Set $G:=E^{e_1+\cdots+e_h}$, where the $e_i$ are defined by \eqref{eq:decompositionideauxz}.  
Then, we infer from \eqref{eq:decompositionideauxz} that  
$$
G(\bchi(\k),\z) \in (\z)^{\left\lfloor \Cauxter \delta_1^{1/N}\delta_2\right\rfloor}
$$
for every $\k \in \K'$. 
Applying $\sigma^{-1}$ and multiplying $G$ by $Q^{2\delta_1(e_1+\cdots+e_h)}$, 
we get that 
\setcounter{constant}{0} \defconstant{a}\defconstant{b}\defconstant{c}
$$
G(\bpsi(\k,\z),\z) \in (\z)_{\C\{\z\}}^{\left\lfloor \Cauxter \delta_1^{1/N}\delta_2\right\rfloor}\, ,
$$
for every $\k \in \K'$. Here, we let $(\z)_{\C\{\z\}}$ denote the ideal generated by the $z_{i,j}$ 
inside the ring of analytic power series 
$\C\{\z\}$. Though $(\z)$ is not necessarily a prime ideal of $\A$, 
the ideal $(\z)_{\C\{\z\}}$ is prime in $\C\{\z\}$. Let $\nu_\z$ denote the corresponding 
valuation\footnote{We stress that $\nu_\z$ has not here 
the same meaning as in Section \ref{subsub:valuation}. There it is defined on the ring $\mathcal A$ 
and it is not necessarily a valuation, 
while here it is defined on $\C\{\z\}$ and it is a valuation.}, 
that is,  
$$
\nu_\z(f):=\sup\{s \in \N :  f \in (\z)_{\C\{\z\}}^s\}\, 
$$
for $f\in\C\{\z\}$. We thus have 
\begin{eqnarray}\label{eq:valuationE}
\nu_\z(E(\bpsi(\k,\z),\z)) &\geq& \left\lfloor \Cauxter \delta_1^{1/N}\delta_2\right\rfloor (e_1+\cdots+e_h)^{-1}\\
\nonumber &\geq &\left\lfloor \Cauxbis \delta_1^{1/N}\delta_2\right\rfloor\, ,
\end{eqnarray}
for every $\k \in \K'$, and some positive real number $\Cauxbis$ that does not depend on $\k$, $\delta_1$ and $\delta_2$.
We observe now that the radius of convergence of the analytic series defining the map $\k \mapsto \bpsi(\k,\z)$ 
does not depend on $\k$. 
Let $\psi_{j,\lambd}(\k)$ denote the coefficient in $\z^{\lambd}$ of $\psi_j(\k,\z)$. 
There thus exist three positive real numbers $\ca$, $\cb$ and $\cc$, independent of $\k$, $j$ and $\lambd$, 
and such that 
$$
|\psi_{j,\lambd}(\k)| \leq \ca\cb^{|\k|}\cc^{|\lambd|}\,.
$$ 
Set 
$$
p:=\left\lfloor \Cauxbis \delta_1^{1/N}\delta_2\right\rfloor \,.
$$
By \eqref{eq:valuationE}, we have  \defconstant{d}
$$
E(\bpsi(\k,\z),\z) = \sum_{|\lambd|\geq p} e_\lambd(\k) \z^\lambd \,.
$$
There thus exists a positive real number $\cd(\delta_1,\delta_2,q)$ such that 
\defconstant{gamma}
\begin{equation}
\label{eq:majorationcoeffE}
|e_\lambd(\k)| \leq  \cd(\delta_1,\delta_2,q)\cb^{|\k|}\cc^{|\lambd|} \,.
\end{equation}
Condition (B) ensures that $\Vert T_\k\balpha\Vert \leq e^{-\cgamma \rho^{|\k|}}$. 
 It thus follows that \defconstant{e} \defconstant{f} \defconstant{t}
\begin{eqnarray}
\nonumber |E(\btau_\k,T_\k\balpha)| & = & |E(\bpsi(\k,T_\k\balpha),T_\k\balpha)|
\nonumber  \\ & = & | \sum_{|\lambd|\geq p} e_\lambd(\k) (T_\k\balpha)^\lambd |
\nonumber \\ & \leq & \sum_{|\lambd|\geq p} |e_\lambd(\k)| \Vert T_\k\balpha\Vert^\lambd 
\nonumber \\ & \leq & \sum_{|\lambd|\geq p} \cd(\delta_1,\delta_2,q)\cb^{|\k|}\cc^{|\lambd|}e^{-|\lambd|\cgamma\rho^{|\k|}} \,,
\end{eqnarray}
for every $\k \in \K'$.
Finally, we get that there exist two positive real numbers $\ce$ and $\cf$ such that 
$$
|E(\btau_\k,T_\k\balpha)|  \leq \ce(\delta_1,\delta_2,q)\cb^{|\k|}e^{-\cf\delta_1^{1/N}\delta_2 \rho^{|\k|}}\, ,
$$
for all $\k \in \K'$. There thus exists a positive real number $\ct$ such that, 
\begin{equation}
\label{eq:majoE}
|E(\btau_\k,T_\k\balpha)|  \leq e^{-\ct\delta_1^{1/N}\delta_2 \rho^{|\k|}}\, ,
\end{equation}
for all $\k \in \K'$, large enough with respect to $\delta_1,\delta_2$, and $q$.


\subsubsection{Lower bound for $|E(\btau_\k,T_\k\balpha)|$}\label{subsub:min}

Let us recall that, following \eqref{eq: annulationFitere},  
we have 
\defconstant{g}\defconstant{h}\defconstant{i}
$$
F(\btau_\k,T_\k\balpha)=0 \, ,
$$
for every $\k\in \mathbb N^r$.  
The power series $F(\bpsi(\k,\z),\z)-F_{q}(\bpsi(\k,\z),\z)$ has valuation at least $q$ in $\z$. 
Reasoning as in Section \ref{subsec: upper}, we can find three positive real numbers $\cg(\delta_1,\delta_2)$, $\ch$, and $\ci$ 
such that 
\begin{eqnarray*}
|P_j(\btau_\k,T_\k\balpha)F_{q}(\btau_\k,T_\k\balpha)^j| 
&= &|P_j(\btau_\k,T_\k\balpha)(F-F_{q})^j(\btau_\k,T_\k\balpha)| \\ 
& \leq &\cg(\delta_1,\delta_2,q)\ch^{|\k|}e^{-\ci  \rho^{|\k|}q} \,.
\end{eqnarray*}
We thus have 
\defconstant{k} \defconstant{l}
$$
\left|\sum_{j=1}^{\delta_1} P_j(\btau_\k,T_\k\balpha)F_{q}(\btau_\k,T_\k\balpha)^j\right| 
\leq \ck(\delta_1,\delta_2,q)\cl^{|\k|}e^{-\ci  \rho^{|\k|}q}\, ,
$$
where $\ck(\delta_1,\delta_2,q)$ and $\cl$ does not depend on $\k$. 

On the other hand, for all $\k \in \K'$, we know that  
\defconstant{delta} \defconstant{dAj} \defconstant{Aj} \defconstant{dAjkj}  \defconstant{rho} 
\defconstant{gamm} \defconstant{n} 
$$
P_0(\btau_\k,T_\k\balpha) \not=0 \, .
$$
Furthermore, the algebraic numbers $P_0(\btau_\k,T_\k\balpha)$, with $\k \in \K'$, 
all belong to a fixed number field. We infer from the Liouville inequality that there exists $\cdelta>0$, 
independent of $\k$, such that 
$$
|P_0(\btau_\k,T_\k\balpha)| \geq  H(P_0(\btau_\k,T_\k\balpha))^{-\cdelta}\, .
$$
The complex numbers $\tau_j(\k)$ are polynomials of degree $d$ in the coefficients of the matrix $A_{\k}(\balpha)$. 
The coefficients of the matrices $A_i(\z)$, $1 \leq i \leq r$, are rational functions whose numerators and denominators 
have degrees less than, say, $\cdAj$, and coefficients of logarithmic heigth less than, say, $\cAj$. 
Then Condition (A) ensures that the numerators and the denominators of the rational functions composing the matrix 
$A_{\k}(\z)$ have degrees less than
$$
\cdAj\cdAjkj \rho^{|\k|}\, ,
$$
and coefficients of logarithmic height less than 
$$ 
\cAj \log N |\k| \leq \crho \rho^{|\k|}\, ,
$$
where $\crho$ is a positive real number. Let $\cgamm$ be a real positive number such that $\cgamm \geq \crho + \cdAj \cdAjkj \log H(\balpha)$, we have
$$
\log H(A_{\k}(\balpha)) \leq  \cgamm \rho^{|\k|}\, .
$$
There thus exists a positive real number $\cn$ such that 
\defconstant{q} \defconstant{o} \defconstant{s} \defconstant{p} 
$$
\log H(\tau_j(\k)) \leq \cn \rho^{|\k|}\,.
$$
The polynomial $P_0$ has degree at most $2\delta_1$ in $\bt$ and at most $\delta_2$ in $\z$. 
Since $\delta_2 \geq \delta_1$, we can bound the height of $P_0(\btau_\k,T_\k\balpha)$ by 
$$
\log H(P_0(\btau_\k,T_\k\balpha)) \leq \cq(q) + \co \delta_2\rho^{|\k|} \,.
$$
This gives 
$$
|P_0(\btau_\k,T_\k\balpha)| \geq  \cs(q)e^{-\cdelta\co\delta_2\rho^{|\k|}}\, .
$$
We thus get the following lower bound: 
\begin{eqnarray*}
|E(\btau_\k,T_\k\balpha)| &\geq& 
|P_0(\btau_\k,T_\k\balpha)| - |\sum_{i=1}^{\delta_1} P_i(\btau_\k,T_\k\balpha)F_{q}(\btau_\k,T_\k\balpha)| 
\\ 
&\geq & \cs(q)e^{-\cdelta\co\delta_2\rho^{|\k|}}-\ck(\delta_1,\delta_2,q)\cl^{|\k|}e^{-\ci \rho^{|\k|}q}\,.
\end{eqnarray*}
Choosing $\k$ large enough with respect to $q$, and $q$ large enough with respect to $\delta_2$, we obtain that 
\begin{equation}
\label{eq:minorationE}
|E(\btau_\k,T_\k\balpha)| \geq e^{-\cp\delta_2\rho^{|\k|}},
\end{equation}
for every $\k \in \K'$, large enough with respect to $\delta_2$.


\subsubsection{Contradiction}\label{contradiction}
We infer from Inequalities \eqref{eq:majoE} and \eqref{eq:minorationE} that 
$$
e^{-\cp\delta_2\rho^{|\k|}} \leq \vert E(\btau_\k,T_\k\balpha)\vert  
\leq e^{-\ct\delta_1^{1/N}\delta_2 \rho^{|\k|}}\, ,\,
$$
for all $\k \in \K'$, large enough with respect to $\delta_2$. 
Taking the logarithm, dividing by $\delta_2\rho^{|\k|}$, and letting $|\k|$ tend to infinity along $\K'$, 
we obtain that 
$$
\cp \geq \ct\delta_1^{1/N}\, .
$$
This provides a contradiction as soon as $\delta_1$ is large enough. 
This ends the proof of Lemma \ref{lem:nulliteF}.
\subsection{End of the proof of Theorem \ref{thm: families} in the case of an analytic gauge transforms}

Let $\k_0$ be given by Lemma \ref{lem:nulliteF}. 
For every positive integer $q$, we recall that we have the following lower bound: 
\begin{equation}\label{eq:minorationvaluationFbis}
\nu_{i_0}(F_q(\bchi(\k_0),\z)) \geq q\,.
\end{equation}
For every integer $i$, $1\leq i\leq s$, we write 
\begin{equation}\label{decompostionxi}
\chi_i(\k_0) = \sum_{\kapp,\lambd,j} s_{i,\kapp,\lambd,j}Y_1^{\kappa_1}\cdots Y_\ell^{\kappa_\ell}\z^\lambd \U^j \in \A, 
\end{equation}
where the indices $\kapp=(\kappa_1,\ldots,\kappa_\ell) \in \N^\ell$ and $\lambd=(\lambda_1,\ldots,\lambda_N)\in \N^N$ 
belong to some finite sets, and where 
$0 \leq j < d_0$. 
For all summations in the rest of this section, we let 
$\kapp$ denote some element of $\N^\ell$, $\lambd$, $\bomega$, and $\gamm$ some elements of $\N^N$, 
$i$ an element of $\{1,\ldots,s\}$, and $j$ an element of $\{0,\ldots,d_0-1\}$. 
For every triple $(i,\kapp,j)$, we write 
$$
\chi_{i,\kapp,j}(\z) = \sum_{\lambd}s_{i,\kapp,\lambd,j}\z^\lambd \in \C[\z]\, .
$$
We also set $\bchi_{\kapp,j} = (\chi_{i,\kapp,j})_{i \leq s}$. 
Then we have the following decomposition: 
\begin{equation}
\label{eq:decompositionF}
F_q(\bchi(\k_0),\z)=\sum_{\kapp,j} Y_1^{\kappa_1}\cdots Y_\ell^{\kappa_\ell}\U^j F_q(\bchi_{\kapp,j}(\z),\z) \,.
\end{equation}
Then it is possible to evaluate $F(\bt,\z)$ at the points  $\bt=\bchi_{\kapp,j}(\z)$ in $\C\{\z\}$.

\begin{lem}
\label{lem:annulationsoussériesF}
For every  pair $(\kapp,j)$, we have 
$$
F(\bchi_{\kapp,j}(\z),\z)=0 \in \C\{\z\}\,.
$$
\end{lem}
\begin{proof}
The valuation $\nu_{i_0}$ induces a norm $\vert \cdot \vert_0$ on the $\C$-vector space $\A$, setting 
$$
 \vert a\vert_0 = 2^{-\nu_{i_0}(a)}\, ,
$$
for $a\in\A$. We use of course the natural convention $2^{-\infty}=0$, so that $\vert 0\vert_0=0$. 
The sequence $\left(F_q(\bchi(\k_0),\z)\right)_{q \in \N}$ is convergent with respect to this norm, and tends to $0$, 
since we have 
$$
\left|F_q(\bchi(\k_0),\z)\right|_0 \leq 2^{-q} \, ,
$$
by Lemma \ref{lem:nulliteF}.  
For every positive integer $q$, we write 
\begin{eqnarray*}
F_q(\bchi_{\kapp,j}(\z),\z) & = & \sum_{|\lambd| < q}\sum_i l_{\lambd,i}\, \chi_{i,\kapp,j}(\z)  \z^{\lambd}
\\ & = & \sum_{|\lambd|< q}\sum_i \sum_{\gamm}  l_{\lambd,i} s_{i,\kapp,\gamm,j}  \z^{\lambd+\gamm}
\\ & := & \sum_{\kapp} g_{\kapp,j,\bomega,q} \z^\kapp\, ,
\end{eqnarray*}
where for every $(\kapp,j,\bomega,q)$,
$$
g_{\kapp,j,\bomega,q} := \sum_{i}\sum_{|\lambd|<q}l_{\lambd,i} s_{i,\kapp,\bomega-\lambd,j} \, .
$$
If $\kapp$, $j$, and $\bomega$ are fixed, the complex number $g_{\kapp,j,\bomega,q}$ is constant for every $q>|\bomega|$. 
Indeed, the only indices $\lambd$ that occur in the sum defining  $g_{\kapp,j,\bomega,q}$ are 
those for which $|\lambd|<|\bomega|$. If $q > |\bomega|$, we thus have
$$
g_{\kapp,j,\bomega,q} := \sum_{i}\sum_{|\lambd|<|\omeg|}l_{\lambd,i} s_{i,\kapp,\bomega-\lambd,j} \, .
$$
For such quadruples $(\kapp,j,\bomega,q)$, we set 
$$
g_{\kapp,j,\bomega}:=g_{\kapp,j,\bomega,q}\, .
$$
Letting $F_q(\bchi_{\kapp,j}(\z),\z)$ converge for the norm associated with $\nu_\z$ in $\C\{\z\}$, 
we obtain that 
\begin{equation}\label{eq:ecritureFchi}
F(\bchi_{\kapp,j}(\z),\z)= \sum_{\bomega} g_{\kapp,j,\bomega} \z^\bomega\, .
\end{equation}
We are going to show that for every $(\kapp,j,\bomega)$, and every $j$, we have $g_{\kapp,j,\bomega}=0$, 
which will end the proof of the lemma. 
Equality \eqref{eq:decompositionF} can be rewritten as 
\begin{equation}
\label{eq:decompositionF2}
F_q(\bchi(\k_0),\z)=\sum_{\kapp,j,\bomega} g_{\kapp,j,\bomega,q}Y_1^{\kappa_1}\cdots Y_\ell^{\kappa_\ell}\U^j\z^\bomega \, .
\end{equation}
Let us fix a positive integer $q_0$, and set 
\begin{equation}\label{eq:deltas0}
\delta(q_0):=\sum_{|\bomega|<q_0}\sum_{\kapp,j} g_{\kapp,j,\bomega}Y_1^{\kappa_1}\cdots Y_\ell^{\kappa_\ell}\U^j\z^\bomega\, \in \A\, .
\end{equation}
For every $q \geq q_0$, we also set 
$$
\epsilon(q,q_0):=\sum_{|\bomega|\geq q_0} \sum_{\kapp,j} g_{\kapp,j,\bomega,q}Y_1^{\kappa_1}\cdots Y_\ell^{\kappa_\ell}\U^j\z^\bomega 
 \in (\z)^{q_0}\, ,
$$
so that 
$$
F_q(\bchi(\k_0),\z)= \delta(q_0) + \epsilon(q,q_0)\, .
$$ 
Letting $q$ tend to infinity, we see that $\epsilon(q,q_0) \to \delta(q_0)$ with respect to the norm $|\cdot |_0$. 
But, for every $q$, $\epsilon(q,q_0)$ belongs to  $(\z)^{q_0}$, which is a closed set for the topology induced by $|.|_0$. 
Hence, $\delta(q_0) \in (\z)^{q_0}$. We can thus write 
$$
\delta(q_0)=\sum_{|\bomega|\geq q_0} \sum_{\kapp,j} \delta_{\bomega,\kapp,j} \z^\bomega \Y^\kapp \U^j  \, ,
$$
where $\delta_{\bomega,\kapp,j}$ belongs to $\C$. 
By \eqref{eq:deltas0}, we have 
$$
\sum_{|\bomega|<q_0}\sum_{\kapp,j} g_{\kapp,j,\bomega}Y_1^{\kappa_1}\cdots Y_\ell^{\kappa_\ell}\U^j\z^\bomega 
- \sum_{|\bomega |\geq q_0} \sum_{\kapp,j} \delta_{\bomega,\kapp,j} \z^\bomega \Y^\kapp \U^j = 0\, ,
$$
Since the monomials $\z^\kapp\Y^\btau\U^j$ are linearly independent over $\C$ for $j<d_0$, we get that 
\begin{equation}\label{eq:nullitég}
g_{\kapp,j,\bomega}=0
\end{equation}
as soon as $|\bomega|<q_0$. Letting $q_0$ run along $\N$, 
we obtain that Equality \eqref{eq:nullitég}  holds true for every $(\kapp,j,\bomega)$. 
This ends the proof.  
\end{proof}

We are now ready to conclude the proof of Theorem \ref{thm: families} in the case of an analytic gauge transform. 
By definition of $\chi_i(\k)$, we have 
$$
\chi_i(\k) = Q(\z,Y_1,\ldots,Y_\ell,\U)\sigma(\psi_i(\k,\z)).
$$
For the sake of simplicity, we set $D(\z):=Q(\z,\phi_1(\z),\ldots,\phi_\ell(\z),\varphi(\z)) \in \C\{\z\}$.  
Applying the isomorphism $\sigma^{-1}$ to the previous equality, and using the decomposition of the $\chi_i(\k_0)$, 
we obtain that 
\begin{equation}\label{eq:constructionR}
\begin{array}{rcl}
D(\z)\psi_i(\k_0,\z)&=&\sigma^{-1}(\chi_i(\k_0))
\\ &=& \sigma^{-1}\left(\sum_{\kapp,\j}\chi_{i,\kapp,j} (\z)Y_1^{\kappa_1}\cdots Y_\ell^{\kappa_\ell} U^j\right) \\ 
&=& \sum_{\kapp,j} \chi_{i,\kapp,j} (\z)\phi_1(\z)^{\kappa_1}\cdots \phi_\ell(\z)^{\kappa_\ell}\varphi(\z)^j.
\end{array}
\end{equation}
Set $\bbeta := T_{\k_0}\balpha$. 
By Lemma \ref{lem:nulliteF}, this choice of $\k_0$ ensure that  the power series 
$D$ is well-defined and non-zero at $\bbeta$.
Set 
$$
\eta_{i}(\z) := D(\bbeta)^{-1}\sum_{\kapp,j}\phi_1(\bbeta)^{\kappa_1}\cdots
\phi_\ell(\bbeta)^{\kappa_\ell}\vphi(\bbeta)^j \chi_{i,\kapp,j}(\z) \in \C(\z),
$$
and $\bet(\z)=(\eta_i(\z))_{i\leq s}$. The power series $\eta_{i}(\z)$ are well-defined at $\bbeta$ for every $i$. 
Using the fact that  $F$ is linear in $\bt$, we infer from Lemma \ref{lem:annulationsoussériesF} that 
\begin{eqnarray*}
F(\bet(\z),\z) &=& D(\bbeta)^{-1}\sum_{\kapp,j} \bphi^\kapp(\bbeta)\vphi(\bbeta)^j F(\bchi_{\kapp,j}(\z),\z) \\
&=& 0\, .
\end{eqnarray*}
Considering this equality at $T_{\k_0}\z$, the definition of $F$ implies that 
\begin{eqnarray}
\label{eq:constructionpolynomeannulateur}
\sum_{i \leq s} \eta_i(T_{\k_0}\z)\f(T_{\k_0}\z)^{\bmu_i}
& =& F(\bet(T_{\k_0}\z),T_{\k_0}\z)\\
\nonumber&=& 0\, .
\end{eqnarray}
On the other and, evaluating $\bet$ at $\bbeta$, it follows from \eqref{eq:constructionR} that
\begin{eqnarray}
\label{eq:evaluation}
\bet(\bbeta)&=& \bpsi(\k_0,T_{\k_0}\balpha)\\
\nonumber &=&\btau_{\k_0}\,.
\end{eqnarray}
Let us now recall that 
$$
\f(T_{\k_0}\z)=A_{\k_0}(\z)^{-1}\f(\z)\, .
$$
Replacing  $\f(T_{\k_0}\z)$ in \eqref{eq:constructionpolynomeannulateur}, we find a vector of rational function 
$\widetilde{\bet}(\z)=(\widetilde{\eta_1}(\z),\ldots,\widetilde{\eta_s}(\z))$, such that 
$$
\sum_{i=1}^s \widetilde{\eta}_i(\z)\f(\z)^{\bmu_i}=0 \, .
$$
By \eqref{eq:evaluation} and by construction of the $\btau_\k$, we obtain that 
$$
\widetilde{\bet}(\balpha)=\btau \, .
$$
Then the polynomial $Q \in \C(\z)[\X]$  defined by 
$$
Q(\z,X_{1,1},X_{1,2},\ldots,X_{1,m1},X_{2,1},\ldots, X_{r,m_r}) = 
\sum_{i \leq s} \widetilde{\eta}_i(\z)\X^{\bmu_i} 
$$
satisfied
$$
Q(\z,\f(\z))=0 \qquad \text{ and } \qquad Q(\balpha,\X)=P(\X) \, ,
$$
as desired.

It only remains one easy point to handle. 
We want to construct a polynomial with the same properties but 
that belongs to $\Q[\z,\X]$ and not only in $\C(\z)[\X]$. 
 Let $V\subset \C$ denote the $\Q$-vector space generated by the coefficients of $Q$. 
Then $V$ is finite dimensional.  Let $1,\pi_1,\ldots,\pi_t$ be a basis of $V$ over $\Q$. 
Then we can decompose our polynomial $Q$ as
$$
Q(\z,\X)=Q_0(\z,\X)+ \pi_1Q_1(\z,\X)+\cdots+\pi_t Q_t(\z,\X)\, ,
$$
where the polynomials $\Q_i$, $1 \leq i \leq t$ belong to $\Q(\z)[\X]$. 
The analytic power series $f_{i,j}(\z_i)$ having their coefficients in $\Q$,  the $\Q$-linear independence of the 
$\pi_i$ implies that 
$$
Q_i(\z,\f(\z))=0 \, ,
$$
for every $i$, $0 \leq i \leq t$. On the other hand, the polynomial $P$ having algebraic coefficients, we deduce that  
$$
Q_0(\balpha,\X)=P(\X) \qquad \text{ and } \qquad Q_i(\balpha,\X)=0 \ \text{ for } 1 \leq i \leq d \, .
$$
The coefficients of $Q_0$ are elements of $\Q(\z)$, say $r_1(\z),\ldots,r_v(\z)$. The fact that $Q_0(\balpha,\X)=P(\X)$ ensures that 
these rational functions are all defined at $\balpha$. Let $d(\z)\in\Q[\z]$ denote the product of the denominators of 
the $r_i$'s. Thus $d(\balpha)\not=0$. Then the polynomial 
$$
A(\z,\X):= \frac{d(\z)}{d(\balpha)}Q_0(\z,\X) \in \Q[\z,\X]
$$
has all the desired properties. This ends the proof of Theorem \ref{thm: families} in the case where the matrix 
$\Phi(\z)$ belongs to ${\rm GL}_m(\Q\{\z\})$. \qed

\subsection{Proof of the proof of Theorem \ref{thm: families} in the general case}\label{sec:ramified}

For every integer $i$, $1\leq i\leq r$, we let $\widehat\bK_{\z_i,d}$ denote 
the field of fractions of $\Q\{\z_i^{1/d}\}$, where $\z_i^{1/d}=(z_{i,1}^{1/d},\ldots,z_{i,n_i}^{1/d})$. 
We set
$$
\widehat\bK_{\z_i}:= \cup_{d\geq 1} \widehat\bK_{\z_i,d}\, .
$$
We also recall that $\z=(\z_1,\ldots,\z_r)$ and that, given a positive integer $d$, we let 
$\widehat\bK_d$ denote the field of fractions of $\Q\{\z^{1/d}\}$. We also set 
$$
\widehat\bK= \cup_{d\geq 1} \widehat\bK_d\, .
$$

In this section, we explain how to modify our proof of Theorem \ref{thm: families} in order to extend it to 
the case where the gauge transforms $\Phi_i(\z_i)$ 
are not necessarily analytic but are allowed to belong to 
${\rm GL}_{m_i}(\widehat\bK_{\z_i})$.  

We first show that how to reduce to the case where $\Phi_i(\z_i)\in {\rm GL}_{m_i}(\widehat\bK_{\z_i,1})$. 
Let $\balpha=(\balpha_1,\ldots,\balpha_r)$ be such that 
the family of pairs $(T_i,\balpha_i)$ is admissible in the sense of Definition \ref{def: globaladmissibility} 
and such that every $\balpha_i$ is regular with respect to the Mahler system \eqref{eq:mahler2}. 
By assumption, there exists a positive integer $j$ such that 
$\Phi_i(\z_i^j)$ belongs to ${\rm GL}_{m_i}(\widehat\bK_{\z_i,1})$. 
Let $\balpha':=(\balpha'_1,\ldots,\balpha'_r)$ be such that $(\balpha'_i)^j=\balpha_i$ for every $i$, $1\leq i\leq r$.  
Then the study of the system \eqref{eq: blocks} at $\balpha$ is equivalent to the study of the  
the Mahler system 
\begin{equation*}
\left(\begin{array}{c}  \g_{1}(\z) \\ \vdots \\ \vdots \\ 
\g_{r}(\z)\end{array} \right) = \left(\begin{array}{cccc} 
A_{1}(\z_1^j) & & & \\ 
& \ddots & & \\
&& \ddots & \\
&&& A_{r}(\z_r^j)
\end{array}
\right)
\left(\begin{array}{c}   \g_{1}(T\z) \\ \vdots \\ \vdots \\ 
\g_{r}(T\z) \end{array}\right) \, ,
\end{equation*}
at $\balpha'$, where $\g_i(\z):=\g_i(\z^j)$.  
It is obvious that the points $\balpha'_i$ are regular, and that the 
the family of pairs $(T_i,\balpha'_i)$ is still admissible.  
Furthermore, every matrix $A_i(\z^j)$ is conjugated to a constant matrix trough the matrix  
$\Phi_i(\z^j)\in{\rm GL}_m(\widehat\bK_{\z_i,1})$.  
Without any loss of generality, we can thus assume  
that, for every natural number $i$, $1\leq i\leq r$, there exists a matrix 
$\Phi_i(\z_i)\in{\rm GL}_m(\widehat\bK_{\z_i,1})$ 
such that
$$
\Phi_i(T_i\z_i)A_i(\z_i)\Phi_i^{-1}(\z_i) \in {\rm GL}_{m_i}(\Q)\, .
$$

For every integer $i$, $1 \leq i \leq r$, we let $\Delta_i(\z_i)$ be a non-zero analytic function such that 
the coefficients of both $\Delta_i(\z_i)\Phi_i(\z_i)$ and $\Delta(\z_i)\Phi_i^{-1}(\z_i)$ belong to $\Q\{\z_i\}$.  
We also set $\Delta(\z):=\Delta_1(\z_1)\cdots\Delta_r(\z_r)$. 
We then infer from Condition (C)  that
$$
\K_\Delta:=\{\k \in \K \ | \ \Delta(\z) \text{ is well-defined and non-zero at } T_\k\balpha\}
$$
is an infinite set.

\begin{lem}\label{lem: Condition(C)-bis} 
Condition (C) still holds when replacing the set $\K$ with $\K_\Delta$.
\end{lem}

\begin{proof}
The function $\Delta(\z)$ belongs to $\Q\{\z\}$, and so is well-defined in a neighborhood of the origin. 
Hence, for all but finitely many $\k \in \K$, $\Delta(\z)$ is well-defined at $T_\k\balpha$. 
So we may suppose that $\Delta(\z)$ is well-defined at $T_\k\balpha$ for every $\k \in \K$. 
Let $L$ be a finite-dimensional $\Q$-vector space and $\psi \in \R_{\Gamma,r} \otimes_\Z L\{\z\}$ 
is such that the family $\left(\psi(\k,\z)\right)_{\k \in \K_\Delta}$ is not identically zero. It follows that 
the family $\left(\Delta(\z)\psi(\k,\z)\right)_{\k \in \K}$ is also 
not identically zero. Then Condition (C) ensures that $\Delta(T_\k\balpha)\psi(\k,T_\k\balpha)\neq 0$ 
for infinitely many $\k \in \K$. 
By definition of $\K_\Delta$, such $\k$ must belong to $\K_\Delta$, which ends the proof.  
\end{proof}

Without any loss of generality, we can thus assume that $\K=\K_\Delta$, that is, $\Delta(T_\k\balpha) \neq 0$ 
for every $\k \in \K$. Lemma 6.1 should then be modified as follow.

\begin{lem}[Lemma 6.1-bis] \label{lem: connect-bis}
For every integer $j$, $1\leq j \leq s$, 
there exists $\psi_j \in \R_{\Gamma,r} \otimes_\Z \widehat\bK_1$ such that
$$
\tau_{j,\k} = \psi_{j}(\k,T_\k\balpha) \, 
$$
and $\Delta(\z)^d\psi_j(\k,\z) \in \C\{\z\}$, for all $\k \in \N^r$. 
Furthermore, there exists a finite dimensional $\Q$-vector space $L_0$ such that 
for all $\k \in \N^r$ and all $j$, $1\leq j \leq s$, the coefficients of the formal power 
series $\Delta(\z)^d\psi_{j}(\k,\z)$ belong to $L_0$.   
\end{lem}

\begin{proof}
The proof follows the same steps as the one of Lemma \ref{lem: connect}. 
We first stress that Lemma 6.2 still holds true. Indeed, by assumption, the matrix $\Delta(\z)\Phi(\z)$, 
is well-defined in some neighborhood of the origin. 
It follows from Condition (B) that, for $\k \in \mathcal K$ large enough, 
$\Delta(\z)\Phi(\z)$ is well-defined at $T_\k\balpha$. 
By assumption, $\Delta(T_\k\balpha)$ does not vanish for $\k \in \K$. 
The matrix $\Phi(\z)$ is thus well-defined at $T_{\k}\balpha$, for every large $\k \in \K$. 
The end of the proof of Lemma 6.2 remains unchanged. 

Let $\k\in\N^r$. For every integer $i$, $1 \leq j \leq s$, the meromorphic function $\psi_j(\k,\z)$  is a polynomial 
of degree $d$ in the coefficients of the matrix $\Phi^{-1}(\z)$. It follows that  
$\Delta(\z)^d \psi_j(\k,\z) \in \C\{\z\}$, as wanted. 
The last assertion of Lemma \ref{lem: connect-bis} is proved in the same way as in the proof of Lemma \ref{lem: connect} 
choosing $L_0$ to be the $\Q$-vector space generated by the monomials of degree at most $d$ in the coefficients of the matrix 
$\Phi(\balpha)$. 
\end{proof}

Then the proofs of Lemmas \ref{lem: vanishpol}, \ref{dimensionespace}, \ref{lem:majorationespacepolynome}, and 
\ref{lem:fonctionauxiliaire} remain unchanged with one exception.  
We just have to be careful when, in Lemma \ref{lem: vanishpol},  we prove the implication (i) $\Rightarrow$ (iii) using 
using Condition (C). Let us assume that $P(\bpsi(\k,T_\k\balpha),T_\k\balpha)=0$ for all but finitely many $\k \in \K$. 
Then, $P$ being homogeneous, $P(\Delta(T_\k\balpha)^d\bpsi(\k,T_\k\balpha),T_\k\balpha)=0$ for all but finitely many $\k \in \K$. 
Using Condition (C), we infer that $P(\Delta(\z)^d\bpsi(\k,\z),\z)=0$ for all $\k\in \K$. Eventually, we get that 
$P(\bpsi(\k,\z),\z)=0$ for all $\k \in \K$.

The main change occurs in the proof of Lemma \ref{lem:nulliteF}  
when providing an upper bound for the quantity $|E(\btau_\k,T_\k\balpha)|$, in Section \ref{subsec: upper}. 
In order to obtain such an upper bound, we now have to provide a lower bound for the quantity $|\Delta(T_\k\balpha)|$. 
For this purpose, we use a result of Corvaja and Zannier \cite{CZ05}. 

\begin{proof}[End of the proof of Theorem \ref{thm: families}]
Setting $G:=E^{e_1+\cdots+e_h}$ as in Section \ref{subsec: upper}, 
we recall that
$$
G(\bchi(\k),\z) \in (\z)^{\left\lfloor C_1 \delta_1^{1/N}\delta_2\right\rfloor}
$$
for every $\k \in \K'$. Applying $\sigma^{-1}$ and multiplying $G$ by 
$$
(Q\Delta^d)^{2\delta_1(e_1+\cdots+e_h)}\, , 
$$
we get that 
$$
G(\Delta(z)^d\bpsi(\k,\z),\z) \in (\z)_{\C\{\z\}}^{\left\lfloor C_1 \delta_1^{1/N}\delta_2\right\rfloor}\, ,
$$
for every $\k \in \K'$. 
Then, reasoning as in Section \ref{subsub:min}, gives the following equivalent form of Equality \eqref{eq:majoE}:   
$$
\vert E(\Delta(T_\k\balpha)^d\bpsi(\k,T_\k\balpha),T_\k\balpha)\vert \leq e^{-c_1\delta_1^{1/N}\delta_2 \rho^{|\k|}}\, ,
$$
for every $\k \in \K'$, large enough with respect to $\delta_1$, $\delta_2$, and $q$. 
Since $E$ is a homogeneous polynomial of degree $2\delta_1$ in $\w$, we have that  
\begin{equation}
\label{eq: majo-bis}
\begin{array}{rcl}
\vert E(\btau_\k,T_\k\balpha)\vert &=&\vert E(\Delta(T_\k\balpha)^d\bpsi(\k,T_\k\balpha),T_\k\balpha)\vert 
\times |\Delta(T_\k\balpha)|^{-2\delta_1d}
\\ & \leq & e^{-c_1\delta_1^{1/N}\delta_2 \rho^{|\k|}}\times \vert \Delta(T_\k\balpha)\vert^{-2\delta_1d}\, .
\end{array}
\end{equation}
On the other hand, we infer from \cite[Proposition 3]{CZ05} that there exist two positive real numbers $C_2$ and $c_2$ 
such that 
\begin{equation}\label{eq: CZ-bis}
|\Delta(T_\k\balpha)| \geq C_2||T_\k\balpha||^{-c_2}\, ,
\end{equation}
 for infinitely many $\k \in \K'$. 
Indeed, as shown in section 4, \cite[Proposition 3]{CZ05} can be applied to the family of points $(T_\k\balpha)_{\k \in \K'}$. 
Furthermore, our choice of $\K$ ensures that $\Delta(T_\k\balpha) \neq 0$ for every $\k \in \K$, and thus for every $\k \in\K'$ 
since $\K'\subset \K$. 
Using the fact that $\delta^{1/N}\delta_2 \gg \delta_1$, as $\delta_1$ tends to infinity, and  
combining \eqref{eq: majo-bis} and \eqref{eq: CZ-bis}, 
we eventually get an upper bound of the same kind than in Section \ref{subsec: upper}. 
That is, 
$$
|E(\btau_\k,T_\k\balpha)| \leq e^{-c_3\delta_1^{1/N}\delta_2 \rho^{|\k|}}\, ,
$$
for infinitely many $\k \in \K'$. 
The computation leading to the lower bound remains the same. 
Furthermore, as the lower bound holds for every large $\k \in \K'$, 
the contradiction of Section \ref{contradiction} still holds true. 
The last part of the proof of Theorem \ref{thm: families} remains unchanged, which  
ends the proof of Theorem \ref{thm: families} in the general case.  
\end{proof}


\section{Admissibility conditions for Theorem \ref{thm: families}} \label{sec: admgen}

Conditions (A), (B), and (C) required to apply Theorem \ref{thm: families} look somewhat stronger than the corresponding  
Conditions (a), (b), and (c) occurring in Theorem \ref{thm: permanence}.  In particular, the vanishing theorem 
corresponding to Condition (C) is much more general than the one corresponding to Condition (c). 
We show here that it is enough for each pair $(T_i,\balpha_i)_{1\leq i\leq r}$ to satisfy Conditions (a), (b), (c) to ensure that 
Conditions (A), (B), (C) are satisfied by the family $(T_i,\balpha_i)_{1\leq i\leq r}$ at the point $\balpha=(\balpha_1,\ldots,\balpha_r)$.  
More precisely, the goal of this section is to prove the following result. 

\begin{thm}\label{thm:equivalenceadmissibilite}
Let us assume that $T_1,\ldots,T_r$ are matrices with non-negative integer coefficients such that 
$\log \rho(T_i)/\log \rho(T_j)\not\in\mathbb Q$ for all $i,j$, $i\not=j$. 
Then the   family $(T_i,\balpha_i)_{1\leq i\leq r}$ is admissible at the point $\balpha=(\balpha_1,\ldots,\balpha_r)$ 
in the sense of Definition \ref{def: globaladmissibility} if, and only if, every pair $(T_i,\balpha_i)$ is admissible in the sense of Definition 
\ref{def:admissible}.
\end{thm}

All along this section, we assume that $T_1,\ldots,T_r$ are matrices with non-negative integer coefficients 
such that 
$\log \rho(T_i)/\log \rho(T_j)\not\in\mathbb Q$ for all $i,j$, $i\not=j$. 
We recall that $\Theta$ is defined in \eqref{eq:Theta} by 
$$\Theta=\left( \frac{1}{\log \rho(T_1)},\ldots,\frac{1}{\log \rho(T_r)}\right) \, .
$$
We first prove two easy Lemmas. 

\begin{lem}\label{lem:convergence}
Let us assume that the matrices $T_1,\ldots,T_r$ belong to $\M$. 
Let  $\balpha=(\balpha_1,\ldots,\balpha_r) \in \C^N$ such that $\balpha_i \in \mathcal U(T_i)$ for every $i$, 
$1 \leq i \leq r$. 
Let $\K \subset \N^r$ be any infinite set that remains at bounded distance of the line generated 
by the vector $\Theta$. 
Then Conditions $(A)$ and $(B)$ are satisfied with this choice of $\K$ and $\rho := e^{1/|\Theta|}$. 
\end{lem} 

\begin{proof}
The matrix $T_i$ being in the class $\M$, it follows from \cite{LvdP77} that 
\begin{equation}
\label{eq:controleTi}
\Vert T_i^{k}\Vert  = \mathcal O(\rho(T_i)^{k}) \qquad \text{ and } \qquad  -\log \Vert T_i^{k}\balpha_i\Vert  = \mathcal O(\rho(T_i)^{k})  \, ,
\end{equation}
for all non-negative integer $k$. 
Let $\K\subset \N^r$ be a set satisfying the assumption of the Lemma. 
Let $B$ denote an upper bound for the distance of any element of $\K$ to the set $\N.\Theta$. 
For every $\k \in \K$, we choose $l(\k) \in \N$ such that 
$$
\Vert\k - l(\k)\Theta\Vert \leq B \, .
$$
Applying \eqref{eq:controleTi}, we obtain that 
$$
\Vert T_{\k}\Vert  = \mathcal O(e^{l(\k)})  \qquad \text{ and } \qquad \log \Vert T_{\k}\balpha\Vert \leq -c e^{l(\k)} \, .
$$
On the other hand, $|\k| \sim l(\k)|\Theta|$, as $|\k|\to \infty$. It follows that Conditions (A) and (B) 
are satisfied by choosing $\rho := e^{1/|\Theta|}$. 
\end{proof}

Reciprocally, we show that the real number $\rho$ and the set $\K$ have to be chosen of the same form 
as  in Lemma \ref{lem:convergence}.

\begin{lem}\label{lem:conditionadmissibilitematrice}
Let $T_1,\ldots,T_r$ be square matrices with non-negative coefficients. Let us assume that the 
family $(T_i,\balpha_i)_{1\leq i\leq r}$ is admissible at the point $\balpha=(\balpha_1,\ldots,\balpha_r)$ 
in the sense of Definition \ref{def: globaladmissibility}. Then each pair $(T_i,\balpha_i)$ is admissible in the sense of Definition 
\ref{def:admissible}. Furthermore, the elements of $\K$ remains at bounded distance of the line generated by 
$\Theta$ and $\rho = e^{1/\vert\Theta\vert}$. 
\end{lem}

\begin{proof} 
Let us assume that the family $(T_i,\balpha_i)_{1\leq i\leq r}$ is admissible at the point $\balpha=(\balpha_1,\ldots,\balpha_r)$ 
in the sense of Definition \ref{def: globaladmissibility}. We also consider the corresponding real number $\rho$ and set 
$\K\subset \mathbb N^r$. 
We first observe that the projection of $\K$ on the $i$-th coordinate cannot be a finite subset of $\N$. 
Indeed, otherwise the set $\{(T_\k\balpha)^{\be_i}$, $\k \in \K\}$ would be finite, where we let $\be_i$ denote the $i$-th vector 
of the standard basis. Let $\lambda_1,\ldots,\lambda_t$ be the elements of this finite set. Then the non-zero polynomial  
$$
P(\z)=\prod_{j=1}^t( \z^{\be_1}-\lambda_j)\, ,
$$
would satisfy $P(T_\k\balpha)=0$ for all $\k \in \K$, which would contradict Condition (C). 
Now, the fact that the projection of $\K$ on each coordinate is infinite,  
directly implies that each pair $(T_i,\balpha_i)$ is admissible in the sense of Definition \ref{def:admissible}. 
Set  $\k=(k_1, \ldots, k_r)$. Using on the one hand Conditions (a) and (b) for each $i$ , and on the other hand 
Conditions (A) and (B), we get that 
\begin{equation*}
\label{eq:equivalence}
\log \rho(T_i)k_i \sim \log(\rho) |\k| \, .
\end{equation*}
Dividing by $\log \rho(T_i)$, summing over $i$, and then dividing by $|\k|$, we get that 
$$
\log \rho = \left(\sum_{i=1}^r \frac{1}{\log \rho(T_i)} \right)^{-1}\, .
$$
Setting 
$$
\epsilon_i(\k) := k_i - \frac{\log(\rho)|\k|}{\log \rho(T_i)} \, , 
$$
we obtain $\sum_{i} \epsilon_i(\k) = 0$. Now we infer from (B) and from the fact that $T_i$ belongs to $\M$, 
that there exist two positive real numbers  $c_i$ et $\gamma_i$ such that 
$$
c_i \rho(T_i)^{k_i}  \leq \Vert T_i^{k_i}\Vert  \leq \gamma_i \rho^{|\k|} = \gamma_i \rho(T_i)^{k_i-\epsilon_i(\k)} \, .
$$
It follows that the numbers $\epsilon_i(\k), \k \in \K$ are bounded. In other words, $\K$ remains at bounded 
distance of the line generated by $\Theta$. This ends the proof. 
\end{proof}

In the rest of this section, we let $\rho$ be defined as in Lemma \ref{lem:convergence}. 
In view of Lemma \ref{lem:convergence} and \ref{lem:conditionadmissibilitematrice},   
the proof of Theorem \ref{thm:equivalenceadmissibilite} follows from the following result. 

\begin{prop}\label{th:lemmedezerosansindependance}
Let us assume that $T_1,\ldots,T_r$ are matrices with non-negative integer coefficients such that 
$\log \rho(T_i)/\log \rho(T_j)\not\in\mathbb Q$ for all $i,j$, $i\not=j$, and that every pair $(T_i,\balpha_i)$ 
is admissible in the sense of Definition \ref{def:admissible}.  
There exist an infinite set $\K \subset \N^r$ that remains at bounded distance of the line generated by 
the vector $\Theta$, and such that Condition {\rm (C)} is satisfied. 
\end{prop}

Let $s \leq r$ be an integer, we let 
$$
\pi_s : \Z^r \mapsto \Z^s
$$
denote the projection on the first $s$ coordinates. We recall that 
$\R_{\Gamma,s}\otimes_\Z \C\{\z\}$ stands for the algebra form by the $(\Gamma,s)$-multivariate exponential polynomial 
with values in $\C\{\z\}$.

\begin{lem}\label{lem:independanceglobalerayonsspectraux} 
Let $s$ be an integer with $1\leq s\leq r$. Let us assume that  the numbers 
$$
\frac{1}{\log \rho(T_1)},\ldots,\frac{1}{\log \rho(T_s)}
$$
are linearly independent over $\Z$. 
We let $\mL \subset \N$ be a piecewise syndetic set and 
$\K = \{\k_l : \ l \in \mL\}$ be a sequence in $\N^r$ such that 
$$
\k_l=l \Theta + \mathcal O(1) \, .
$$
Let $\psi \in \R_{\Gamma,s}\otimes_\Z \C\{\z\}$ be non-zero (that is, $\psi$ does not identically vanish on $\Z^s$). 
Then the set 
$$
\mL_0 :=\left\{l \in \mL \, | \, \psi(\pi_s(\k_l),T_{\k_l}\balpha)=0\right\} 
$$
is not piecewise syndetic. 
\end{lem}

The proof of Lemma \ref{lem:independanceglobalerayonsspectraux} follows the same strategy as the one of Lemma 3.3.1 in 
\cite{Ni_Liv}. However, our framework is more general and we also need to consider piecewise syndetic sets. 
This makes our proof of Lemma \ref{lem:independanceglobalerayonsspectraux} more technical.  
We invite the reader to look at the proof of Lemma 3.3.1 in \cite{Ni_Liv}. 
This could make the following arguments more transparent.

\begin{proof}
To reduce the amount of notation, we set $\kk := \pi_s(\k)$ for $\k \in \Z^r$. 
We argue by contradiction, assuming that the set $\mL_0$ is piecewise syndetic, with bound $B$. 
We write 
\begin{equation}\label{serieexp} 
\psi(\kk,\z)=\sum_{i=1}^q\bet_{i}^{\kk} \sum_{|\bmu|\leq \delta_i} \kk^\bmu g_{i,\bmu}(\z)\, , 
\end{equation}
where the $r$-tuples $\bet_i=(\eta_{i,1},\ldots,\eta_{i,s})$ are all distinct, and where $q$ and the numbers $\delta_i$, 
$1 \leq i \leq q$, are minimal. 
The decomposition \eqref{serieexp} is then unique up to permutation of indices (see for instance \cite[Th\'eor\`eme 1]{La89}). 
Now, we define $\Delta(\psi)$ as the cardinal of the set 
$$
\left\{ (i,\bmu), 1 \leq i \leq q, \text{ such that either } |\bmu| < \delta_i , \text{ or } |\bmu|=\delta_i \text{ and } g_{i,\bmu} \neq 0 \right \} \, .
$$
Without any loss of generality, we assume that  $\delta_q \geq \delta_i$ for all $i$, $1 \leq i\leq q$. 
We argue by induction on $\Delta(\psi)$. 
If $\Delta(\psi)=1$, then $\psi(\kk,\z) = \bet^\kk g(\z)$ and a contradiction follows form Theorem \ref{thm:lemmedezero}. 
We now assume that $\Delta(\psi):=\Delta > 1$ and that the conclusion of the Lemma holds true for $\Delta(\psi)<\Delta$. 
Without any loss of generality, we can assume that $\bet_q = (1,\ldots,1)$. 
Let $\bnu$ be a $s$-tuple of non-negative integers such that  $|\bnu|=\delta_q$.  We set $g(z) := g_{q,\bnu}(z) \neq 0$. 
For every $\be\in \mathbb N^r$, we define the map 
$$
\xi_\be(\kk,\z) := \psi(\kk+\bbe,T_{\be}\z)g(\z) -  \psi(\kk,\z)g(T_\be\z) \, ,
$$
where $T_\be$ is defined as in \eqref{eq:matriceTe} by 
\begin{equation*}
T_\be:=\left(\begin{array}{ccc} T_1^{e_1} && \\ & \ddots & \\ && T_r^{e_r} \end{array} \right)
\end{equation*}
Then $\xi_\be(\kk,\z)\in\R_{\Gamma,s}\otimes_\Z \C\{\z\}$
One has 
\begin{equation}
\label{eq:definitionxi}
\xi_\be(\kk,\z) =\sum_{i=1}^{q-1} \left( \bet_{i}^{\kk}  \sum_{|\bmu|\leq \delta_i} \kk^\bmu h_{\be,i,\bmu}(\z)\right) 
+  \sum_{|\bmu|\leq \delta_i,\, \bmu \neq \bnu} \kk^\bmu h_{\be,q,\bmu}(\z)\, .
\end{equation}
By construction, $\Delta(\xi_\be) < \Delta(\psi)$ for all $\be$. We can thus apply our assumption to 
$\xi_\be$. Set 
$$
\mL_\be:=\{ l \in \N : \xi_\be(\kk_l,T_{\k_l}\balpha)=0\} \,.
$$
Given an integer $e_1 \geq B$, we consider the set $\mL_1$ formed by the integers $l \in \mL_0$ 
for which there exists $e$,  $e_1\leq e\leq e_1+ B$ such that $l + e \in \mL_0$. 
We infer from Lemma \ref{lem:syndetique} that $\mL_1$ is piecewise syndetic. For such a pair $(l,e)$, 
we set $\be:=\be(l,e)=\k_{l+e}-\k_l$ and we let $\mathcal E_1$ denote the (finite) set of $r$-tuples $\be$ 
obtained in this way. For $\be \in \mathcal E_1$, one has 
\begin{eqnarray*}
\xi_\be(\kk_l,T_{\k_l}\balpha) &=&  \psi(\kk_{l+e},T_{\k_{l+e}}\balpha)g(\balpha) -  
\psi(\kk_l,T_{\k_l}\balpha)g(T_{\k_{l+e}}\balpha) \\
&=& 0 \,.
\end{eqnarray*}
This implies the following inclusion: 
$$
\mL_1 \subset \bigcup_{\be \in \mathcal E_1} \mL_\be \, .
$$
By Lemma \ref{lem:syndetique}, there thus exists $\be(l,e) \in \mathcal E_1$ such that $\mL_\be$ is piecewise syndetic. 
For such a $\be=\be(l,e)$, our assumption implies that $\xi_\be \equiv 0$. 
Letting $e_1$ run along the integers larger than $B$, we can find infinitely many $r$-tuples $\be=\be(l,e)$ such that 
$\xi_\be \equiv 0$. Let  $\mathcal E_2$ denote the infinite set of such  $r$-tuples. For $\be \in \mathcal E_2$, 
we thus have $h_{\be,i,\bmu}(\z)=0$ for all indices $(i,\bmu)$. If $\bmu$ is a $s$-tuple such that $|\bmu| = \delta_q$, 
we obtain that 
\begin{eqnarray*}
0&=&h_{\be, q, \bmu}(\z)\\
&=& g_{q,\bmu}(T_\be\z)g(\z)  - g_{q,\bmu}(\z)g(T_\be\z) \, .
\end{eqnarray*}
Dividing by $g(T_\be\z)g(\z)$,  we get that 
$$
\frac{g_{q,\bmu}(T_\be\z)}{g(T_\be\z)} = \frac{g_{q,\bmu}(\z)}{g(\z)} \, \cdot
$$
Since the matrix $T_\be$ has not root of unity as eigenvalue, we can apply Theorem 3.1 of \cite{Ni_Liv}.  
It follows that there exists a complex number  $\gamma_\bmu$ such that 
\begin{equation}\label{eq:egaliteg}
g_{q,\bmu}(\z) = \gamma_\bmu g(\z) \, .
\end{equation}
Let us remark that, in particular, $\gamma_\bnu = 1$. 
Let us consider now a $s$-tuple of non-negative integer $\bmu_0$ such that 
$\bnu - \bmu_0 \in \N^s$ and $|\bnu - \bmu_0| = 1$. 
Then we have that 
\begin{eqnarray*}
0&=& h_{\be, q, \bmu_0}(\z) \\
&= &\sum_{\bmu > \bmu_0} \bbe^{\bmu - \bmu_0} \binom{\bmu}{\bmu_0} g_{q,\bmu}(T_\be\z)g(\z) + 
g_{q,\bmu_0}(T_\be\z)g(\z) - g_{q,\bmu_0}(\z)g(T_\be\z)
\\ & = & \left(\sum_{\bmu > \bmu_0} \bbe^{\bmu - \bmu_0} \binom{\bmu}{\bmu_0} \gamma_\bmu \right)g(T_\be\z)g(\z) 
+ g_{q,\bmu_0}(T_\be\z)g(\z) - g_{q,\bmu_0}(\z)g(T_\be\z) \, .
\end{eqnarray*}
Indeed, if $\bmu > \bmu_0$, then one has $|\bmu|=\delta_q$, and  \eqref{eq:egaliteg} gives that 
$g_{q,\bmu}(\z) = \gamma_\bmu g(\z)$. 
Dividing by $g(T_\be\z)g(\z)$, we obtain that 
$$
\frac{g_{q,\bmu_0}(T_\be\z)}{g(T_\be\z)} = 
\frac{g_{q,\bmu_0}(\z)}{g(\z)} - \sum_{\bmu > \bmu_0} \bbe^{\bmu - \bmu_0} \binom{\bmu}{\bmu_0} \gamma_\bmu \, .
$$
By Theorem 3.1 in \cite{Ni_Liv}, this implies that 
\begin{equation}
\label{eq:formelineairezero}
\sum_{\bmu > \bmu_0} \bbe^{\bmu - \bmu_0} \binom{\bmu}{\bmu_0} \gamma_\bmu  = 0 \, .
\end{equation}
But if $\bmu > \bmu_0$, our assumption implies that $\bmu - \bmu_0$ is a vector of the standard basis of $\C^s$. 
Thus for every $i \leq s$, there exists a unique $\bmu:=\bmu(i)$ such that $\bbe^{\bmu(i) - \bmu_0} =e_i$. 
Recall that $\bbe = \overline{\be(l,e)}=\kk_{l+e} - \kk_l$, with $l$ and $l+e$ in $\mL_0$. 
As $\mathcal E_2$ is infinite, there exist some  $\be(l,e)$ in $\mathcal E_2$ with arbitrarily large $e$.  
But 
$$
\lim_{e\to\infty}\frac{\kk_{l+e} - \kk_l}{e} =  \left(\frac{1}{\log \rho(T_1)},\ldots,\frac{1}{\log \rho(T_s)}\right)\, .
$$
Dividing Equality \eqref{eq:formelineairezero} by $e$ and taking the limit as $e$ tends to infinity, 
we obtain that  
$$
\sum_{i=1}^s \frac{1}{\log \rho(T_i)} \binom{\bmu(i)}{\bmu_0} \gamma_{\bmu(i)}  = 0 \,.
$$
Since, by assumption, the numbers $\frac{1}{\log \rho(T_1)},\ldots,\frac{1}{\log \rho(T_s)}$ 
are linearly independent over $\mathbb Z$, 
we get that 
$$
\binom{\bmu}{\bmu_0} \gamma_\bmu = 0 \, ,
$$
for every $\bmu$. 
Choosing $\bmu=\bnu$, we obtain that 
$$
\delta_q = \binom{\bnu}{\bmu_0} = 0 \, ,
$$
since $\gamma_\bnu=1$. 
Since $\delta_q \geq \delta_i$ for every $i$, it follows that $\delta_i=0$, $1 \leq i \leq q$. 
Thus $\psi(\k,\z)$ can be written as 
$$
\psi(\kk,\z)=\sum_{i=1}^q\bet_{i}^{\kk}g_{i}(\z) \, .
$$
For $\be \in \mathcal E_2$, $h_{\be,1,\boldsymbol{0}}(\z)$ can thus be written as 
\begin{eqnarray*}
0&=&h_{\be,1,\boldsymbol{0}}(\z)\\
&=& \bet_1^\bbe g_{1}(T_\be\z)g(\z) - g_{1}(\z)g(T_\be\z) \, .
\end{eqnarray*}
By minimality of $q$, we have that $g_1(\z) \neq 0$. Then Theorem 3.1 of \cite{Ni_Liv} implies that 
$\bet_1^\bbe=1$ for all $\be \in \mathcal E_2$. Taking the logarithm, it follows that for all pairs $(l,e)$ 
such that $\be = \be(l,e) =(e_1,\ldots,e_r) \in \mathcal E_2$, one has 
$$
e_1\log \eta_{1,1}  + \cdots + e_s\log \eta_{1,s}  = 0 \, .
$$
Let us assume that the vector $(\log \eta_{1,1},\ldots,\log \eta_{1,s})$ is non-zero. Since $\bbe$ has non-negative integer coordinates, 
there exists a non-zero vector $\bmu = (\mu_1,\ldots,\mu_s)\in \Z^s$ such that 
$$
e_1\mu_1  + \cdots + e_s\mu_s  = 0 \, .
$$
Dividing this equality by $e$ and letting $e$ tend to infinity, we obtain that 
$$
\frac{\mu_1}{\log \rho(T_1)}+ \cdots + \frac{\mu_s}{\log \rho(T_s)} = 0\, .
$$
This provides a contradiction. Thus $\log \eta_{1,1} =\cdots =\log \eta_{1,s}=0$. 
Since by assumption the group $\Gamma$ is torsion-free, we get that $\bet_1 = (1,\ldots,1) = \bet_q $, which 
contradicts the minimality of $q$. This ends the proof. 
\end{proof}

We are now ready to construct the suitable set $\K = \{\k_l,\ l \in \mL\}$ needed for proving Proposition 
\ref{th:lemmedezerosansindependance}.

\begin{lem}\label{lem:constructionK}
Let $\bmu_1,\ldots,\bmu_t\in\Z^r$ be a basis of the orthogonal complement of $\Theta$ in $\Q^r$. 
Then there exists a piecewise syndetic set $\mL\subset \N$ and a sequence of vectors $(\k_l)_{l\in\mL}\in\Z^r$ 
in the orthogonal complement of the vector space generated by the vectors $\bmu_i$, $1 \leq i \leq t$, and such that 
\begin{equation}\label{eq:equivalencekl}
\k_l =l \Theta + \mathcal O(1) \, .
\end{equation}
\end{lem}

\begin{rem}
In the case where $t=0$, that is, when the numbers 
$$
\frac{1}{\log \rho(T_1)},\ldots,\frac{1}{\log \rho(T_r)}
$$
are linearly independent over $\Z$, the situation is simplified and we could choose 
$$
\K = \left\{ \left(\left \lfloor \frac{l}{\log \rho(T_1)}\right\rfloor,\ldots,
\left\lfloor\frac{l}{\log \rho(T_r)}\right\rfloor\right) : l \in \N \right\}\, .
$$
Of course, this condition is automatically satisfied when $r=2$ for $\log \rho(T_1)$ and 
$\log \rho(T_2)$ are assumed to be multiplicatively independent. 
However, it seems to be a hard Diophantine problem to check it as soon as $r\geq 3$.
\end{rem}

\begin{proof}[Proof of Lemma \ref{lem:constructionK}]
We first define the sequence $(\k^0_l)_{l\in \N}$ by  
$$
\k^0_l := \left(\left \lfloor \frac{l}{\log \rho(T_1)}\right\rfloor,\ldots,\left\lfloor\frac{l}{\log \rho(T_r)}\right\rfloor\right) \, .
$$
Since $\bmu_1$ is orthogonal to $\Theta$, the scalar products 
$
\langle \bmu_1 \, ,\,  \k^0_l \rangle
$
remain bounded when $l$ runs along $\mathbb N$. 
By Property \emph{(ii)} of Lemma \ref{lem:syndetique}, there exists an integer $c_1$ 
such that  the set 
$$
\mL_1=\{l \in \N : \langle \bmu_1 \, , \, \k_l \rangle = c_1 \}
$$
is piecewise syndetic. Let $\bnu_1 \in \Z^r$  be such that 
$\langle \bmu_1 \, , \, \bnu_1 \rangle = c_1 $. Then, for all $l \in \mathcal L_1$, 
the vector 
$$
\k^1_l := \k^0_l - \bnu_1
$$
is orthogonal $\bmu_1$ and satisfies 
$$
\k^1_l =l \Theta + \mathcal O(1)  \, .
$$
Since $\bmu_2$ is orthogonal to $\Theta$, the scalar products 
$$
\langle \bmu_2 \, ,\, \k^1_l \rangle
$$
remain bounded when $l$ runs along $\mL_1$. There thus exists an integer $c_2$ 
such that the set 
$$
\mL_2=\{l \in \mL_1 : \langle \bmu_2 \, , \, \k^1_l \rangle = c_2 \}
$$
is piecewise syndetic. 
Let $\bnu_2 \in \Z^r$ be such that $\langle \bmu_2 \, , \, \bnu_2 \rangle = c_2 $ and 
$\langle \bmu_1 \, , \, \bnu_2 \rangle = 0 $.  
We could for instance choose $\bnu_2=\k^1_{l_0}$ for some $l_0 \in \mL_2$. 
Then, for all $l \in \mathcal L_2$, the vector 
$$
\k^2_l = \k^1_l - \bnu_2
$$
is orthogonal to $\bmu_1$ and $\bmu_2$, and satisfies 
$$
\k^2_l =l \Theta + \mathcal O(1) \, .
$$
Keeping on in this way, we can find a piecewise syndetic set $\mL:=\mL_t$, and a sequence of vectors 
$(\k^t_l)_{l\in \mL}:=(\k_l)_{l\in\mL}$ with the desired property. This ends the proof.
\end{proof}

Let $\mL$, and $\K:=\{\k_l,\ l \in \mL\} \subset \N^r$ be defined as in Lemma \ref{lem:constructionK}.  
We are now ready to end the proof of Proposition \ref{th:lemmedezerosansindependance}.

\begin{proof}[Proof of Proposition \ref{th:lemmedezerosansindependance}]
Set $s = r-t$. Without any loss of generality, we can assume that the numbers 
$$
\frac{1}{\log \rho(T_1)},\ldots,\frac{1}{\log \rho(T_s)},
$$
are linearly independent over $\Z$. Let us consider the matrix 
$$
S:=\left(\begin{array}{ccccccc}
1 & 0 & \cdots & \cdots & \cdots & \cdots & 0
\\ 0 & 1 & \ddots & \cdots & \cdots & \cdots &\vdots 
\\ \vdots & \ddots & \ddots & \ddots & \cdots & \cdots & \vdots 
\\ 0 & \cdots & \cdots & 1 & 0 & \cdots & 0
\\ \hline &&& \bmu_1 &&&
\\ &&&\vdots&&& 
\\ &&& \bmu_t &&&
\end{array}
\right)\, .
$$
By assumption, $S$ is non-singular and has integer coefficients. 
We also observe that our choice of the set $\mL$ ensures that 
$S\k_l = (k_{l,1},\ldots,k_{l,s},0,\ldots,0)$ 
for every $l \in \mL$. 
Let us consider $\psi \in \R_{\Gamma,r}\otimes_\Z \C\{\z\}$. 
We let $E$ denote the map from $\Z^s$ to $\Z^r$ defined by 
$E(k_1,\ldots,k_s)=(k_1,\ldots,k_s,0\ldots,0)$. 
We now define a map $\overline{\psi}$ from $\Z^s$ to $\C\{z\}$ by 
$$
\overline{\psi}((k_1,\ldots,k_s),\z) = \psi(S^{-1}E(k_1,\ldots,k_s),\z) \,.
$$
Note that if $S$ is not invertible in $\mathbb Z$ but only in $\mathbb Q$, then 
the vector $S^{-1}E(k_1,\ldots,k_s)$ may have rational coordinates 
(but with a fixed denominator corresponding to the determinant of $S$).  
Since $\psi\in \R_{\Gamma,r}\otimes_\Z \C\{\z\}$, it can naturally be extended to vectors in $\mathbb Q^r$. 
This shows that $\overline{\psi}$ is well-defined. Furthermore, it is not hard to see that 
$\overline{\psi}\in\R_{\Gamma',s}\otimes_\Z \C\{\z\}$. Indeed, because of the possible occurrence of the determinant 
of $S$ as denominator of the coordinates of the vector $S^{-1}E(k_1,\ldots,k_s)$, we may have to replace $\Gamma$ 
by some $\Gamma'$. 
But $\Gamma'$ is still torsion-free in that case. 
The main point now is that for all $l \in \mL$, $E(\overline{\k_l})=S(\k_l)$ and thus 
\begin{eqnarray}\label{eq:pushbackpsi}
  \overline{\psi}(\kk_l,\z)=\psi(\k_l,\z) \,.
\end{eqnarray}
Now if $\psi(\k_l,T_{\k_l}\balpha)=0$ for all but finitely many $l \in \mL$, then 
\eqref{eq:pushbackpsi} implies that the set 
$$
\mL_0 = \{ l \in \mL \, | \, \overline{\psi}(\kk_l,T_{\k_l}\balpha)=0\}
$$
is piecewise syndetic, since $\mL$ is piecewise syndetic. 
Since $\overline{\psi}\in\R_{\Gamma',s}\otimes_\Z \C\{\z\}$, 
Lemma \ref{lem:independanceglobalerayonsspectraux} thus 
implies that  
$$\overline{\psi}(\k,\z)=0$$ 
for all $\k\in\mathbb Z^s$, and 
in particular $\overline{\psi}(\kk_l,\z)=0$ for all $l \in \mL$.  
By \eqref{eq:pushbackpsi}, we get that $\psi(\k_l,\z)=0$ for all $l \in \mL$, concluding the proof. 
\end{proof}


\section{Proofs of Theorems \ref{thm: permanence} and \ref{thm: purity}}\label{sec: final}

In this final section, we complete the proof of our two main results. 

\subsection{Proof of Theorem  \ref{thm: permanence}} 
Concerning the proof of Theorem  \ref{thm: permanence}, there is nothing  more to do. 
The conclusion directly follows from Theorems 
\ref{thm: families} and \ref{thm:equivalenceadmissibilite}. Indeed, if 
$(T,\balpha)$ is admissible in the sense of Definition \ref{def:admissible}, then by 
Theorem \ref{thm:equivalenceadmissibilite} (with $r=1$) 
it is also admissible in the sense of Definition \ref{def: globaladmissibility}. 
We can then apply Theorem \ref{thm: families} 
(with $r=1$) to obtained the desired conclusion.

\subsection{Proof of Theorem  \ref{thm: purity}} 
We are now going to see how to deduce Theorem  \ref{thm: purity} from Theorem \ref{thm: families}. 

Let $\mathbb L$ be a field and $\mathbb K \subset \mathbb L$ be a subfield of $\mathbb L$. 
Let us consider some finite sets  $\mathcal E_1,\ldots,\mathcal E_r\subset \mathbb L$, 
with $\mathcal E_i:=\{\alpha_{i,1},\ldots,\alpha_{i,m_i}\}$. For every $i$, we consider 
the vector of indeterminates $\X_i:=(X_{i,1},\ldots,X_{i,m_i})$. 
We also set $\mathcal E:=\cup \mathcal E_i$, $\X:=(\X_1,\ldots,\X_r)$,  
and ${\check\X}_i:=(\X_1,\ldots,\X_{i-1},\X_{i},\ldots,\X_r)$. 

The $\mathbb K$-vector space formed by 
the $\mathbb K$-linear relations between the elements of $\mathcal E$ is defined by 
$$
{\rm Lin}_{\mathbb K}(\mathcal E_i) 
:= \left\{L(\X_i)=  a_1X_{i,1}+\cdots +a_{m_i}X_{i,m_i}\in \mathbb K[\X_i] 
: L(\alpha_{i,1},\ldots,\alpha_{i,m_i})=0\right\}\, .
$$
We also set
$$
{\rm Lin}_{\mathbb K}(\mathcal E_i \mid \mathcal E) 
:=  {\rm span}_{\mathbb K[\Check{\X}_i]} \{ L(\X_i) : L\in {\rm Lin}_{\mathbb K}(\mathcal E_i)  \}\, .
$$
Then we let $\mathbb K[\X]_{mul}$ denote the set of polynomials that are multilinear with respect to 
every set of variables $\X_i$. Hence, $P$ belongs to $\mathbb K[\X]_{mul}$ if, for every $1\leq i\leq r$, 
it has a decomposition of the form  
$$
P(\X)=\sum_{k=1}^{m_i} a_k(\Check{\X}_i)X_{i,k} \, ,
$$
where $a_k(\Check{\X}_i)\in \mathbb K[\Check{\X}_i]$. 
We finally define 
$$
{\rm Mul}_{\mathbb K}(\mathcal E_1,\ldots,\mathcal E_r) 
:= \left\{P(\X)\in \mathbb K[\X]_{mul} : P(\alpha_{1,1},\ldots,\alpha_{r,m_r})=0 \right\}\, .
$$

Before proving Theorem \ref{thm: purity}, we prove the following result. 

\begin{prop}\label{prop:pur-lin}
Under the assumptions of Theorem \ref{thm: purity}, we have 
$$
{\rm Mul}_{\Q}(\mathcal E_1,\ldots,\mathcal E_r) \subset \sum_{i=1}^r {\rm Lin}_{\Q}(\E_i \mid \E) \, .
$$
\end{prop}

In order to prove Proposition \ref{prop:pur-lin}, we need the following two lemmas. 

\begin{lem}\label{lem:pur-fonc}
Under the assumptions of Theorem \ref{thm: purity}, we have 
$$
{\rm Mul}_{\Q(\z)}\left(\f_1(\z_1),\ldots,\f_r(\z_r)\right) 
\subset \sum_{i=1}^r {\rm Lin}_ {\Q(\z)}(\f_i(\z_i)\ | \ \f(\z)) \,.
$$
\end{lem}

\begin{proof}
We argue by induction on $r$. For $r=1$, there is nothing to do.  Let us now assume that $r>1$ and 
that the result holds for $r-1$. 
Let $s$ denote the rank over $\Q(\z_r)$ of the power series $f_{r,1}(\z_r),\ldots,f_{r,m_r}(\z_r)$. 
Reordering if necessary, we can assume that the functions
$$
f_{r,1}(\z_r),\ldots,f_{r,s}(\z_r)
$$
are linearly independent over $\Q(\z_r)$. 
There thus exist some rational functions $r_{j,k}(\z_r)$, $j>s$, 
$1 \leq k \leq s$ such that for every $j>s$, we have 
\begin{equation}\label{eq:dependancelineairef1}
f_{r,j}(\z_r)=r_{j,1}(\z_r)f_{r,1}(\z_r)+\cdots+r_{j,s}(\z_r)f_{r,s}(\z_r)\, .
\end{equation}
We stress that the functions $f_{r,1}(\z_r),\ldots,f_{r,s}(\z_r)$ are also linearly independent over 
$\Q(\z_r)((\z_1,\ldots,\z_{r-1}))$ for the sets of variables $\z_i$ are pairwise disjoint. 
An element $L(\X)$ in ${\rm Mul}_{\Q(\z)}\left(\f_1(\z_1),\ldots,\f_r(\z_r)\right)$ can be decomposed as  
\begin{equation}\label{eq:decompositionL}
L(\X)=\sum_{j=1}^{m_r} X_{r,j}L_j(\Check{\X}_r) \, .
\end{equation}
For $k \leq s$, we set 
\begin{equation}\label{eq:definitionL'}
L'_{k}(\Check{\X}_r) := L_{k}(\Check{\X}_r) + \sum_{j > s} r_{j,k}(\z_r)L_j(\Check{\X}_r
) \, .
\end{equation}
From \eqref{eq:dependancelineairef1}, \eqref{eq:decompositionL}, and \eqref{eq:definitionL'}, we infer that 
\begin{equation}
\label{eq:specialisationL}
L(\Check{\X}_r,\f_r(\z_r))=\sum_{k=1}^s f_{r,k}(\z_r)L'_{k}(\Check{\X}_r) \, .
\end{equation}
By assumption, $L(\f(\z))=0$. Furthermore, the functions $f_{r,k}(\z_r)$, $1\leq k\leq s$ are linearly independent 
over $\Q(\z_r)((\z_1,\ldots,\z_{r-1}))$. Hence, Equality \eqref{eq:specialisationL} implies that 
$$
L'_{k}(\f_1(\z_1),\ldots,\f_{r-1}(\z_{r-1})) = 0 \, .
$$
Thus 
$$
L'_{k} \in {\rm Mul}_{\Q(\z_1,\ldots,\z_{r-1})}\left(\f_1(\z_1),\ldots,\f_{r-1}(\z_{r-1})\right) \,.
$$ 
By induction, we obtain that  
\begin{equation}\label{eq:decompositionL'}
L'_{k} \in \sum_{i=1}^{r-1} {\rm Lin}_ {\Q(\z_1,\ldots,\z_{r-1})}(\f_i(\z_i)\mid \f(\z))\, .
\end{equation}
Setting $L' := \sum_{k=1}^s L'_{k}(\Check{\X}_r)X_{r,k}$,  we infer from Equality \eqref{eq:specialisationL}  that 
$$
L( \Check{\X}_r ,\f_r(\z_r)) - L'(\Check{\X}_r,\f_r(\z_r))=0\, .
$$
It follows that  
$$
L - L' \in {\rm Lin}_ {\Q(\z)}(\f_{r}(\z_{r})\mid \f(\z)) \,. 
$$
By \eqref{eq:decompositionL'}, we thus obtain 
$$
L \in \sum_{i=1}^r {\rm Lin}_ {\Q(\z)}(\f_i(\z_i)\mid \f(\z)) \,,
$$
concluding the proof. 
\end{proof}

\begin{lem}\label{lem:independancelineaire}
Let us assume that for every $i$, $1\leq i \leq r$, the elements of $\E_i$ are linearly independent over $\Q$. Then 
$$
{\rm Mul}_{\Q}\left(\E_1,\ldots,\E_r\right) = \{0\} \, .
$$
\end{lem}

\begin{proof}
Without any loss of generality, we assume that  
$$
\E_i = \{f_{i,1}(\balpha_i),\ldots,f_{i,s_i}(\balpha_i)\}\, ,
$$
for some $s_i \leq m_i$. Enlarging the sets $\E_i$ if necessary, we can also assume that the complex numbers 
$$
f_{i,j}(\balpha_i),\, 1 \leq j \leq s_i,
$$
form a basis of the $\Q$-vector space generated by the numbers $f_{i,j}(\balpha_i)$, $1\leq j \leq m_i$. 
There thus exist some algebraic numbers $\lambda_{i,j,k} \in \Q$ such that, for every pair $(i,j)$, $1\leq i \leq r$, 
$1 \leq j \leq m_i$, we have 
\begin{equation}\label{eq:baseindependancefij}
f_{i,j}(\balpha_i)=\sum_{k \leq s_i} \lambda_{i,j,k} f_{i,k}(\balpha_i) \, .
\end{equation}
For every $i$, $1\leq i \leq r$, we set $\overline \X_i := (X_{i,j})_{j\leq s_i}$ 
and $\overline \f_i := (f_{i,j})_{j \leq s_i}$. We also set  
$\overline \X=(\overline\X_1,\ldots,\overline\X_r)$.  

Let $L(\overline \X) \in {\rm Mul}_{\Q}\left(\E_1,\ldots,\E_r\right)$. 
Theorem \ref{thm:equivalenceadmissibilite} ensures that the family $(T_i,\balpha_i)_{1\leq i \leq r}$ satisfies 
Conditions (A), (B) and (C). Theorem \ref{thm: families} thus applies, and there exists a polynomial $Q \in \Q(\z,\X)$, 
multilinear with respect to the vector of indeterminates 
$\X_i=(X_{i,1},\ldots,X_{i,m_i})$, and such that 
\begin{equation}
\label{eq:annulationQ}
Q(\z,\f(\z))=0\, \qquad \text{ and } \qquad Q(\balpha,\X)=L(\overline \X) \, .
\end{equation}

Following Lemma \ref{lem:pur-fonc}, there exist some polynomials $Q_1(\z,\X),\ldots,Q_r(\z,\X)$ 
such that $Q_i(\z,\X) \in {\rm Mul}_ {\Q(\z)}(\f_i(\z_i) \mid \f(\z))$ and 
\begin{equation}\label{eq:decompositionQ}
Q(\z,\X)=\sum_{i=1}^r Q_i(\z,\X) \, .
\end{equation}
Every polynomial $Q_i$ can be uniquely decomposed as 
\begin{equation}\label{eq:decompositionQi}
Q_i(\z,\X)=R_i(\z,\overline \X) + S_i(\z,\X) \, ,
\end{equation}
where the polynomial $S_i(\z,\X)$ does not contain any monomial with support in $\overline \X$.  
Following \eqref{eq:annulationQ},  we have 
\begin{equation}\label{eq:nullitesommeS}
\sum_{i=1}^r S_i(\balpha,\X)=0 \, .
\end{equation}
Let us consider the linear map 
$$
\Lambda : \left\{ \begin{array}{ccc} \Q(\z)[\X] & \rightarrow & \Q(\z)[\overline \X]
\\ X_{i,j} & \mapsto & \sum_{k \leq s_i} \lambda_{i,j,k} X_{i,k} \end{array} \right.
$$
By \eqref{eq:baseindependancefij}, we see that $\Lambda$ is defined so that, for every $P \in \Q(\z)[\X]$, 
$$
P\left(\z,(f_{i,j}(\balpha))_{1 \leq i \leq r,1 \leq j \leq m_i}\right) 
= \Lambda(P)\left(\z,(f_{i,j}(\balpha))_{1 \leq i \leq r, 1 \leq j \leq s_i}\right) \,.
$$
Since $Q_i(\z,\X) \in {\rm Mul}_ {\Q(\z)}(\f_i(\z_i) \mid \f(\z))$, the polynomial $Q_i$ vanishes when $\X_i$ is evaluated at $\f_i(\z_i)$. 
Using $\Lambda$, we thus obtain 
$$
\Lambda(Q_i)(\z,\overline \X_1,\ldots,\overline \X_{i-1},\overline{ \f_i(\z_i)},\overline \X_{i+1},\ldots,\overline \X_r)=0 \, .
$$
Evaluating this equality at $\z=\balpha$,  we find that 
$$
\Lambda(Q_i)(\balpha,\overline \X_1,\ldots,\overline \X_{i-1},\overline {\f_i(\balpha_i)},\overline \X_{i+1},\ldots,\overline \X_r)=0 \, .
$$
By assumption, the numbers $f_{i,j}(\balpha_i)$, $1 \leq j \leq s_i$, are linearly independent over $\Q$. 
Hence, 
$$
\Lambda(Q_i)(\balpha,\overline \X)=0 \, .
$$
On the other hand, the definition of $\Lambda$ implies that $\Lambda(R_i(\z,\overline \X))=R_i(\z,\overline \X)$ for every $i$. 
Applying $\Lambda$ to the Equality \eqref{eq:decompositionQi}, and evaluating at $\z= \balpha$, we thus obtain 
\begin{eqnarray*}
0 &= &\Lambda(Q_i(\balpha,\X)) \\
&=&R_i(\balpha,\overline \X) + \Lambda(S_i(\balpha,\X)) \, .
\end{eqnarray*}
By \eqref{eq:nullitesommeS}, we have 
\begin{eqnarray*}
\sum_{i=1}^r R_i(\balpha,\overline \X) &= &- \sum_{i=1}^r \Lambda(S_i(\balpha,\X)) \\
&=& - \Lambda \left(\sum_{i=1}^r S_i(\balpha,\X)\right) \\
&=& 0 \, .
\end{eqnarray*}
But, on the other hand, $L(\overline \X)=\sum_{i=1}^r R_i(\balpha,\overline \X)$. We thus have $L(\overline \X)=0$, 
concluding the proof. 
\end{proof}

We are now ready to prove Proposition \ref{prop:pur-lin}.

\begin{proof}[Proof of Proposition \ref{prop:pur-lin}] 
Without loss of generality, we can assume that  
$$
\E_i=\{f_{i,1}(\z),\ldots,f_{i,s_i}(\z_i)\}
$$
for some integer $s_i \leq m_i$. For every $i$,  we consider a set $\J_i \subset \{1,\ldots,s_i\}$ such that the numbers 
$$
f_{i,j}(\balpha_i),\, j \in \mathcal J_i,
$$
form a basis of the $\Q$-vector space generated by the numbers $f_{i,j}(\balpha_i)$, $1\leq j \leq s_i$. 
There thus exist numbers $\lambda_{i,j,k} \in \Q$ such that 
\begin{equation}\label{eq:baseindependancefij2}
f_{i,j}(\balpha_i)=\sum_{k \in \J_i} \lambda_{i,j,k} f_{i,k}(\balpha_i) \, ,
\end{equation}
for every $j \leq s_i$,
We consider the vector of indeterminates $\X=(X_{i,j})_{1 \leq i \leq r,\, 1 \leq j \leq s_i}$, 
and we let $\X_{\J}$ denote the set of indeterminates of the form $X_{i,j}$ with $j \in \J_i$. 
There is a ring morphism $\Lambda$ defined by 
$$
\Lambda : \left\{ \begin{array}{ccc} \Q[\X] & \rightarrow & \Q[\X_\J]
\\ X_{i,j} & \mapsto & \sum_{k \in \J_i} \lambda_{i,j,k} X_{i,k}\, . \end{array} \right.
$$
We stress that $\Lambda$ is defined so that for every $P \in \Q[\X_i]$, we have 
$$
P\left((f_{i,j}(\balpha_i)_{1 \leq j \leq s_i}\right) = \Lambda(P)\left((f_{i,j}(\balpha_i)_{j \in \J_i}\right)\, .
$$
Let $L(\X) \in  {\rm Mul}_{\Q}(\E_1,\ldots,\E_r)$.  Then 
\begin{eqnarray*}
\Lambda(L)\left(\left(f_{i,j}(\balpha_i)\right)_{\substack{ 1 \leq i \leq r \\ j \in \J_i}}\right) &= 
&L\left(\left(f_{i,j}(\balpha_i)\right)_{\substack{ 1 \leq i \leq r \\ 1 \leq j \leq s_i}} \right)\\
&=& 0 \, .
\end{eqnarray*}
By assumption, for every $i$, the numbers $f_{i,j}(\balpha_i)$, $j \in \J_i$, are linearly independent $\Q$. 
We thus infer from Lemma \ref{lem:independancelineaire} that $\Lambda(L)(\X_\J)=0$. 
Hence, $L$ belongs to the kernel of $\Lambda$. This kernel is generated by the linear relations  
$$
L_{i,j}(\X):=X_{i,j} - \sum_{k \in \J_i} \lambda_{i,j,k} X_{i,k}
$$
for $1 \leq i \leq r$ and $j \notin \J_i$. On the other hand,  Equality \eqref{eq:baseindependancefij2}  ensures that 
$$
L_{i,j}(\X_1,\ldots,\X_{i-1},\f_i(\balpha_i),\X_{i+1},\ldots,\X_r)=0\, .
$$
In other words, 
$$
L_{i,j}(\X)\in {\rm Lin}_{\Q}(\mathcal E_i \mid \mathcal E)\,.
$$ 
This shows that 
$L(\X)\in   \sum_{i=1}^r {\rm Lin}_{\Q}(\E_i \mid \E)$, which ends the proof. 
\end{proof}

Before proving Theorem \ref{thm: purity}, we recall some basic facts about Kronecker product of matrices. 

Let $A=(a_{i,j})$ and  $B=(b_{i,j})$ be two matrices with coefficients in a field  $\mathbb K$, with respective size $m\times n$ and 
 $p\times q$. We let $A \otimes B$ denote the Kronecker product of $A$ and $B$, which is defined as the 
 $mp\times nq$ matrix whose $(i,j)$-coefficient is 
$$
a_{\left\lfloor \frac{i-1}{p} \right \rfloor+1, \left\lfloor \frac{j-1}{q} \right \rfloor+1}
\times b_{i-p\left\lfloor \frac{i-1}{p} \right \rfloor,j-q \left\lfloor \frac{j-1}{q} \right\rfloor} \, .
$$
In other words, we have 
$$
A \otimes B = \left(\begin{array}{ccc} a_{1,1}B & \cdots & a_{1,n}B 
\\ \vdots & \ddots & \vdots \\ a_{m,1}B & \cdots & a_{m,n}B \end{array}\right) \, .
$$
Given a positive integer $d$, we let 
$$
A^{\otimes d} = \underbrace{A \otimes \cdots \otimes A}_{d \text{ times }}\, ,
$$
denote the $d$-th Kronecker power of the matrix $A$.

The following classical properties can be found in \cite{Horn-Johnson}.

\begin{lem}\label{lem: kro}
The two following properties holds true. 
\begin{enumerate}
\item[{\rm (i)}] If $A$ is an $m\times m$ matrix, then 
$$
\det A^{\otimes d} = (\det A)^{md}\, .
$$
\item[{\rm (ii)}]  If $A,\ B,\ C$, and $D$ are matrices such that the product $AB$ and $CD$ are well defined, 
then  
$$
(AB)\otimes(CD)=(A \otimes C)(B \otimes D) \, .
$$
\end{enumerate}
\end{lem}

We are now ready to prove Theorem \ref{thm: purity}. 

\begin{proof}[Proof of Theorem  \ref{thm: purity}]
Of course, $\sum_{i=1}^r{\rm Alg}_{\Q}(\mathcal E_i\mid \mathcal E)\subset {\rm Alg}_{\Q}(\E)$ 
and we thus only have to prove the converse inclusion. 
Again, we assume without loss of generality that  
$$
\E_i = \{f_{i,1}(\balpha_i),\ldots,f_{i,s_i}(\balpha_i)\}\, ,
$$
for some $s_i \leq m_i$. 
We consider a vector of indeterminates $\X:=(\X_1,\ldots,\X_r)$, where 
$\X_i:=(X_{i,j})_{1 \leq j \leq s_i}$.

Let $P(\X) \in {\rm Alg}_{\Q}(\E)$ and let us fix an integer $i$ with $1 \leq i \leq r$. 
We let $d_i$ denote the degree of $P$ with respect to the indeterminates $\X_i$. 
We order, using the lexicographic order, the elements $\bmu_1,\ldots,\bmu_{s_i^{d_i}}$ of the set 
$$
\left\{1,\ldots, s_i\right\}^{d_i}\, .
$$
For $1 \leq j \leq s_i^{d_i}$, we set  
$$
M_{i,j}(\X_i):=X_{i,\mu_{j,1}}\cdots X_{i,\mu_{j,d_i}} \, ,
$$
where $\bmu_j=(\mu_{j,1},\ldots,\mu_{j,d_i})$. 
The $M_{i,j}(\X_i)$, $1 \leq j \leq s_i^{d_i}$, thus run over the \og ordered monomials\fg{} of degree  
$d_i$ in $\X_i$. 
For every pair $(i,j)$, we set $g_{i,j}(\z_i):=M_{i,j}(\f_i(\z_i))$. 
The functions $g_{i,j}$ can be described in terms of Kronecker products for we have 
\begin{eqnarray*}
\g_i(\z_i) &=& \left(\begin{array}{c} g_{i,1}(\z_i) \\ \vdots \\ g_{i,s_i^{d_i}}(\z_i) \end{array}\right)\\
&=& \f_i(\z_i)^{\otimes d_i}\, .
\end{eqnarray*}
We then infer from Property (ii) of Lemma  \ref{lem: kro} that 
\begin{equation}\label{eq:mahlerpuissances}
\g_i(\z_i) = A_i(\z_i)^{\otimes d_i} \g_i(T_i\z_i) \, .
\end{equation}
By assumption, the Mahler system associated with the matrix $A_i(\z_i)$ is 
regular singular. There thus exist two matrices
$\Phi(\z_i) \in {\rm GL}_{m_i}(\Q\{\z_i\})$ and  $B_i \in {\rm GL}_{m_i}(\Q)$, such that 
\begin{equation}\label{eq:reguliersingulier}
A_i(\z_i)=\Phi_i(\z_i)B_i\Phi_i(T_i\z_i)^{-1}\, .
\end{equation}
Applying Property (ii) of Lemma  \ref{lem: kro} to Equation \eqref{eq:reguliersingulier}, we obtain 
$$
A_i(\z_i)^{\otimes d_i} = \Phi_i(\z_i)^{\otimes d_i} B_i^{\otimes d_i}\left(\Phi_i(\z_i)^{-1}\right)^{\otimes d_i}\, .
$$
This shows that the system \eqref{eq:mahlerpuissances} is also regular singular. 
On the other hand,  the poles of $A_i(\z_i)^{\otimes d_i}$ are the same as $A_i(\z_i)$, and,  
following Property (i) of Lemma  \ref{lem: kro},  the determinant of $A_i(\z_i)^{\otimes d_i}$ 
is a power of $\det A_i(\z_i)$. The point $\balpha_i$ thus remains regular with respect to the system 
\eqref{eq:mahlerpuissances}.

Let us now consider a vector of indeterminates $\Y=(Y_{i,j})_{i\leq r,\, j\leq s_i^{d_i}}$. 
Let $Q \in \Q[\Y]$ be defined such that $Q((M_{i,j}(\X_i))_{i,j})=P(\X)$. Hence, 
$$
Q\left((g_{i,j}(\balpha))_{\substack{ 1 \leq i \leq r \\ 1 \leq j \leq s_i^{d_i}}}\right) = 0 \, .
$$
Furthermore, $Q$ is linear with respect to the vector 
$(Y_{i,j})_{j\leq s_i^{d_i}}$, for every $1 \leq i \leq r$.  
By Proposition \ref{prop:pur-lin}, there exist polynomials  
$$
Q_1(Y_{1,1},\ldots,Y_{1,s_1^{d_1}}),\ldots,Q_r(Y_{r,1},\ldots,Y_{r,s_r^{d_r}})\, ,
$$ 
such that 
$$
Q_i((g_{i,j}(\balpha_i))_{j \leq s_i^{d_i}})=0\, ,
$$
for every $i$, $1\leq i \leq r$, and such that $Q$ belongs to the vector space generated by 
$Q_1,\ldots,Q_r$ over $\Q[\Y]$. It follows that $Q_i((M_{i,j}(\X_i))_{j\leq s_i^{d_i}})\in {\rm Alg}_{\Q}(\mathcal E_i)$, 
and finally that $P\in \sum_{i=1}^r{\rm Alg}_{\Q}(\mathcal E_i\mid \mathcal E)$. This concludes the proof. 
\end{proof}


\end{document}